\documentclass[11pt]{article}
\pdfoutput=1
\usepackage{amsfonts}
\usepackage{}
\usepackage[tbtags]{amsmath}
\usepackage{amssymb}
\usepackage{amsthm}
\usepackage[misc]{ifsym}
\usepackage{cases}
\usepackage{mathrsfs}
\usepackage{color, xcolor}
\usepackage{amsmath}
\usepackage{CJKutf8}
\numberwithin{equation}{section}   
\normalsize
\title{\bf Nonlinear Open-Loop Mean Field Stackelberg Stochastic Differential Game}
\author{\normalsize Jianhui Huang\thanks{Department of Applied Mathematics, The Hong Kong Polytechnic University, Hong Kong, China. Email: james.huang@polyu.edu.hk.},\quad Qi Huang\thanks{\it Corresponding author. School of Mathematics, Shandong University, Jinan 250100, P.R. China,
Email: huangqi\_email123@163.com. The present work constitutes a part of her work for her doctoral dissertation.}}

\newtheorem{mythm}{Theorem}[section]
\newtheorem{mydef}{Definition}[section]
\newtheorem{mylem}{Lemma}[section]
\newtheorem{Remark}{Remark}[section]
\begin{document}
\begin{CJK}{UTF8}{gbsn}
\maketitle

\noindent{\bf Abstract:}\quad This paper studies a nonlinear open-loop mean field Stackelberg stochastic differential game by using the probabilistic method through the
FBSDE system and the idea of taking control as the fixed point. We successively construct the decentralized optimal control problems for the followers
and the leader, among which the leader's decentralized optimal control problem is a partial information optimal control problem with the
fully coupled conditional mean-field {\it forward-backward stochastic differential equation} (FBSDE, in short) as the state equation.
We successively derive the maximum principles for the corresponding decentralized
optimal control problems of the followers and the leader. To obtain the existence, uniqueness and estimations of solutions of the state equation,
the variational equation and the adjoint equation for the leader's decentralized optimal control problem, we study the well-posedness of
a new form of conditional mean-field FBSDE. And the decentralized optimal controls of the leader and followers are proved to be the
approximate Stackelberg equilibrium of the nonlinear mean field Stackelberg game. Finally, we apply the theoretical results developed in this paper to solve a nonlinear mean field Stackelberg game problem between a robot control center and unicycle-type swarm robots.

\vspace{2mm}


\vspace{2mm}

\section{Introduction}

\hspace{5mm}In recent years, mean field games with a large number of players have been extensively studied. Each player interacts with the other
players through the mean field term in its cost functional and state equation, where the mean field term captures the average state of the players.
When the number of participants is large, the dimensionality of the state space tends to infinity, resulting in a ``curse of dimensionality" problem.
In this case, the direct calculation of Nash equilibrium in dynamic games like Ba\c{s}ar and Olsder \cite{BO1998} would be very tedious and complicated.

{\it Mean field game} (MFG, in short) was independently introduced by Huang, Malham\'e and Cines (\cite{HCM2003}, \cite{HMC2006}) and Lasry and
Lions (\cite{LL2006-1}, \cite{LL2006-2}, \cite{LL2007}). In the MFG, when the number of players approaches infinity, each player is very insignificant
and can only exert a negligible influence on others through his negligible contribution to community. In the framework of MFG, each individual interacts
with a common medium created in the community, i.e. the mean field term, without explicitly emphasizing the interaction between any two players.
Therefore, each participant will make his own decision based on his own information and the mean field term in the community.

The limiting problem identified in \cite{LL2006-1}, \cite{LL2006-2}, \cite{LL2007} was a forward-backward {\it partial differential equation} (PDE, in
short) system of a {\it Hamilton-Jacobi-Bellman} (HJB, in short) equation and a {\it Fokker-Planck} (FP, in short) equation. And they derived
approximate Nash equilibrium from the solutions of the limiting problem. Inspired by the analysis of large communication networks in engineering,
Huang, Malham\'e and Cines \cite{HMC2006} independently introduced {\it Nash Certainty Equivalence} (NCE, in short) approach to study the mean field game.
Huang \cite{H2010} considered an LQ mean field game with a major player and a large number of minor players where the major player's influence would not
diminish when the number of minor players tended to infinity. Nourian and Caines \cite{NC2013} extended this model to nonlinear case and approximate
Nash equilibrium was proved when minor players were coupled to the major player only through their cost functionals. Carmona and Delarue \cite{CD2013}
developed a probabilistic approach to deal with the mean field game where the limiting system was a fully coupled FBSDE. Carmona and Zhu used the
probabilistic approach to study the major-minor mean field game in \cite{CZ2016}.

The games studied in the literature described above are all Nash games where there is no hierarchy among the participants and decisions are given at
the same time. Heinrich von Stackelberg \cite{St1934} first proposed the concept of a hierarchical solution for markets, i.e. the Stackelberg solution,
in 1934. In a Stackelberg two-person game, the leader has more power than the follower and thus dominates. First, for any decision announced by the
leader, the follower solves his own optimization problem. The optimal response of the follower will be found. The optimal response of the follower is
a function of the leader's decision. Knowing the follower's optimal response, the leader then solves his optimization problem to obtain his optimal
strategy. By substituting the leader's optimal strategy into the follower's optimal response, the follower's optimal strategy can be obtained. The
leader's optimal strategy and the follower's optimal strategy together form the Stackelberg solution. The Stackelberg solution concept was later
extended to a multi-period setting, see Simaan and Cruz \cite{SC1973-a}, \cite{SC1973-b}. Due to the different types of dominance of the leader
over the follower and various information sets, there are various equilibrium concepts in the dynamic Stackelberg game. See \cite{BO1998} for
different Stackelberg equilibrium concepts for discrete-time and continuous-time deterministic cases.

Yong \cite{Y2002} studied the open-loop solution of the leader-follower stochastic {\it linear-quadratic} (LQ, in short) game, where all coefficients
were random and the weight matrices of controls in the cost functionals did not have to be positive definite. Bensoussan, Chen, and Sethi \cite{BCS2015}
obtained the maximum principle of the global solution of Stackelberg stochastic differential game under adapted open-loop and adapted closed-loop
memoryless information structures. Shi, Wang, and Xiong \cite{SWX2016} considered the stochastic maximum principle for leader-follower stochastic
differential games with asymmetric information. For more information on Stackelberg stochastic differential games, see Li, Marelli, Fu, Cai, and
Meng \cite{LMFCM2021}, Chen and Shen \cite{CS2018}, Li and Yu \cite{LY2018}, Feng, Hu, and Huang \cite{FHH2024} and the references therein.

In the real situation, there may be an infinite number of players involved in the game and there may be hierarchy among the players. In order to
study this situation, it is necessary to combine the Stackelberg game with the mean field game to study the mean field game under the Stackelberg
framework. Nourian, Caines, Malham\'e, and Huang \cite{NCMH2012} studied LQ large population leader-follower stochastic systems in which the agents
had linear state equations and coupled through quadratic cost functionals. In their model, the reference trajectory for the leader is unknown to the
followers, which required likelihood ratio estimator techniques. Bensoussan, Chau, and Yam \cite{BCY2017} consider an LQ mean field game between a
leader and a group of followers in a Stackelberg game setting, in which the evolution of each follower was influenced by the delay effects from his/her
state and control variables as well as those of the leader. Moon and Ba\c{s}ar \cite{MB2018} studied an LQ mean field Stackelberg differential game
with adapted open-loop information structure. They proved that the local optimal decentralized controls of the leader and followers are an
$(\varepsilon_1, \varepsilon_2)$-Stackelberg-Nash equilibrium of the original game. Huang, Si, and Wu \cite{HSW2021} discussed a controlled LQ
Gaussian large population system that included three types of interacting agents, namely the major leader, minor leaders, and minor followers.
By combining the major-minor mean field game with the Stackelberg game, they derived the Stackelberg-Nash-Cournot approximate equilibrium. Huang and Yang \cite{HY2023} investigated the feedback strategies of LQ mean field Stackelberg stochastic differential game and proved the corresponding asymptotic Stackelberg equilibrium. Wang \cite{W2024} studied an LQ mean field leader-follower game using the direct method. Feng and Wang \cite{FW2024} considered the social optimal problem of LQ mean field Stackelberg stochastic differential game. The above literatures differ in the specific problems
studied and the methods used, but they all study LQ mean field Stackelberg game.

Bensoussan, Chau, and Yam \cite{BCY2016} studied mean field games between a dominating player and a group of representative agents, where both the state
equation and the objective functionals are nonlinear. They derived necessary condition for the optimal controls and equilibrium condition by two
coupled {\it stochastic partial differential equations} (SPDEs, in short), the first one is {\it Stochastic Hamilton-Jacobi-Bellman} (SHJB, in short) equation
and the second is FP equation. However, they only prove that the optimal control of the representative agent satisfying
the consistency condition is the $\varepsilon$-Nash equilibrium of N representative agents, and do not consider the approximate equilibrium problem of
dominating player. The result of \cite{BCY2016} is generalized to the delayed case in Bensoussan, Chau, and Yam \cite{BCY2015}. Fu and Horst \cite{FH2020} investigated the solvability of linear McKean-Vlasov FBSDEs derived from Stackelberg games with mean-field type control and terminal state constraints. Aurell, Carmona, Dayanikli and Lauri{\`e}re \cite{ACDL2022} studied epidemic control in large populations using finite-state mean field Stackelberg games. Djete \cite{D2023} discussed Principal-agent problems using mean field Stackelberg games. Bergault, Cardaliaguet and Rainer \cite{BCR2024} studied the relaxed solutions of the mean field Stackelberg game where the leader has private information. In the game model they considered, the leader has no state equation, the followers' state equations are linear and relatively simplified, and the forms of the followers' cost functionals are also quite special. See \cite{SW2021}, \cite{WZ2014}, \cite{GHZ2022}, \cite{MIXZ2022}, \cite{SBB2025}, \cite{SXXCZ2018} for more research on mean field Stackelberg games in LQ, discrete time, incentive design, resource allocation and so on.

In this paper, we use the probabilistic method to investigate the nonlinear open-loop mean field Stackelberg stochastic differential game. Eventually,
the consistency condition system we obtain is composed of fully coupled conditional mean-field FBSDEs. The application of the probabilistic method
enables us to subsequently study the solvability of the consistency condition system we have obtained through some classic methods related to FBSDE.
We first heuristically derive the decentralized
equilibrium strategies and the limit state-average of followers. Due to the nonlinear
structure of the state equation, the derivation of the limit state-average of followers in the nonlinear mean field Stackelberg game is much more
complicated than the previous studies on the related derivation in the LQ mean field Stackelberg game. Based on the heavy limit is equal to the
repeated limit and the tailor-made propagation chaos analysis on the exchangeable sequence of random variables, we derive the limit state-average of
followers in an enlightening way. We also heuristically derive the asymptotic limit of the centralized optimal Nash response of followers. When the
followers adopt the optimal Nash response asymptotic limit, we derive the asymptotic limit of the centralized optimal strategy of the leader.
Therefore, for any given
leader's strategy, we first construct the decentralized optimal control problems for followers, and obtain the corresponding
decentralized optimal response functions for followers. Then, after knowing the decentralized optimal response of the representative
follower, the leader proceeds to solve his/her decentralized optimal control problem.

We employ the maximum principle of the classical stochastic optimal control and the fixed-point property to derive the maximum principles for the
decentralized optimal control problems of followers. After substituting in the representative follower's decentralized optimal
response, the leader's decentralized optimal control problem becomes a partial information
optimal control problem with the fully-coupled conditional mean-field FBSDE as the state equation. We obtain the corresponding maximum principle.
Moreover, we discover that the leader's problem in the Stackelberg stochastic differential game under partial information is similar in form to the
leader's decentralized optimal control problem in our nonlinear mean field Stackelberg stochastic differential game (See \cite{SWX2016}). We think
this is a interesting finding.

Inspired by \cite{PW1999}, \cite{W1998}, \cite{NWW2023}, we apply the method of continuation established by Hu and Peng \cite{HP1995} and Peng and Wu
\cite{PW1999} to investigate the existence, uniqueness and estimations of the solution of a new type of fully-coupled conditional
mean-field FBSDE. And the obtained corresponding well-posedness results are used to give the existence, uniqueness and estimations of solutions of the
state equation, variational equation and adjoint equation for the leader's decentralized optimal control problem. The form of the fully-coupled
conditional mean-field FBSDE that we consider is not included in previous related research literature
(e.g. \cite{HT2022}, \cite{NWW2023}, \cite{CDW2024}). It is proposed based on the specific forms of the state equation, variational equation and
adjoint equation of the leader's decentralized optimal control problem in this paper. Then, we prove that the optimal controls of the
decentralized control problems for the leader and followers that we construct are the $\varepsilon$-Stackelberg equilibrium of the nonlinear
mean field Stackelberg game.

Inspired by \cite{EB2019}, \cite{AM2014}, \cite{ST2015}, \cite{HS2016}, \cite{CCWXQF2021}, and the references therein, we apply the theoretical results developed in this paper to study a nonlinear mean field Stackelberg game problem between a robot control center and unicycle-type swarm robots. We first derive the necessary conditions for the decentralized optimal controls of the unicycle-type swarm robots acting as followers. And from these conditions, we obtain the explicit structure that the followers' decentralized optimal controls must satisfy. Under this structure, we then formulate and analyze the decentralized optimal control problem of the robot control center as the leader. Finally, in the deterministic setting, we derive the necessary conditions and explicit form of the leader's decentralized optimal control. In this practical example, the state dynamics of the unicycle-type swarm robots, as followers, are nonlinear, which necessitates the use of nonlinear mean field Stackelberg game theory for their analysis. This demonstrates that the study of nonlinear mean field Stackelberg games is of clear practical relevance.

Next, we compare with the previous related literature to illustrate and further summarize our innovations.

(i) Since the state equation in the nonlinear mean field Stackelberg game we study is nonlinear, the derivation of the limit state-average of the
followers in our problem is much more complicated than that in the LQ mean field Stackelberg game (See Subsection \ref{section 2.3} in this paper).
Moreover, the constructions of the decentralized optimal control problems for the leader and followers, the derivation of the maximum principle for
the decentralized optimal control problems, and the proof of the approximate Stackelberg equilibrium in the nonlinear mean field Stackelberg game are
all significantly different from those in the LQ mean field Stackelberg game.

(ii) We apply the probabilistic method and utilize the FBSDE system to study the open-loop strategy of the nonlinear mean field Stackelberg game.
For any given leader's strategy, unlike previous literature which all take the mean field measure term as the fixed point to solve the mean field game,
we take the control as the fixed point to analyze the mean field game among followers (See Steps 1-5 in Subsection \ref{section 2.4.1} in this paper).
We successively derive the maximum principles for the decentralized optimal control problems of the followers and leader respectively
(See Section \ref{section 5}).
After substituting the optimal response of the representative follower, the leader's decentralized optimal control problem is a partial information
optimal control problem with a fully coupled conditional mean-field FBSDE as the state equation. We derive its maximum principle and find that the
leader's problem in the Stackelberg game under partial information is similar in form to the decentralized optimal control problem of the leader in
our nonlinear mean field Stackelberg game (See \cite{SWX2016}). This is an interesting connection.

(iii) We obtain the final consistency condition system in our problem, which is a fully coupled conditional mean-field FBSDE (See equation
(\ref{the final consistency condition system})). Because the equation forms studied in previous literature on conditional mean-field FBSDEs
(See \cite{NWW2023}, \cite{CDW2024}, \cite{HT2022}) cannot
incorporate the consistency system (\ref{the final consistency condition system}) we ultimately obtain, the results from those previous studies
cannot be applied to equation (\ref{the final consistency condition system}) either. Therefore, based on our final consistency system, we propose
a new kind of fully coupled mean-field FBSDE. We apply the method of continuation to study its well-posedness, and then apply the obtained results to
provide the existence, uniqueness, estimations of the solution for the final consistency system
(See Subsection \ref{some well-posedness results of a new kind of CMF-FBSDEs}). By comparison, it can be found that Bensoussan, Chau, and Yam
\cite{BCY2016} employed the SHJB-FP system to study the nonlinear mean field Stackelberg game, and the consistency condition system they obtained was a
fully coupled forward-backward SPDEs. And they did not consider the solvability issues of the final consistency condition system. Therefore,
our research methods and the final results for the nonlinear mean field Stackelberg game are novel and have not been presented in previous literature.

(iv) Bensoussan, Chau, and Yam \cite{BCY2015},
\cite{BCY2016} only proved the approximate Nash equilibrium among followers. They did not consider the approximate equilibrium problem of the leader.
We finally prove that the optimal controls of the decentralized control problems of the leader and followers are the
$\varepsilon$-Stackelberg equilibrium of the nonlinear open-loop mean field Stackelberg game
(See Section \ref{the approximate Stackelberg equilibrium}).

(v) We apply the theoretical results developed in this paper to solve a nonlinear mean field Stackelberg game problem between a robot control center and unicycle-type swarm robots. First, we derive the necessary conditions for the decentralized optimal controls of the unicycle-type swarm robots acting as followers. Then, in the deterministic setting, we obtain the necessary conditions for the decentralized optimal control of the robot control center acting as the leader. In this practical example, the state dynamics of the unicycle-type swarm robots, as followers, are nonlinear, and thus the nonlinear mean field Stackelberg game framework is required for their analysis. This illustrates the practical significance of studying nonlinear mean field Stackelberg games.

The paper is organized as follows. A nonlinear open-loop mean field Stackelberg game is formulated in Section \ref{problem formulation}.
In Section \ref{section 3}, based on the centralized optimal equilibrium strategies of the leader and followers, the decentralized strategies of
the leader and followers as well as the limit state-average of the followers are heuristically derived. In Subsection \ref{section 2.4.1}, we construct
the followers' decentralized control problem, and after considering the asymptotic limit of the optimal Nash responses of followers,
we construct the leader's decentralized control problem in Subsection \ref{the iterative limit of the leader}. In Section \ref{section 5}, we
successively obtain the maximum principles for the decentralized optimal control problems of the followers and the leader. Moreover,
in Section \ref{some well-posedness results of a new kind of CMF-FBSDEs}, we investigate the well-posedness of a novel type of conditional
mean-field FBSDE. In Section
\ref{the approximate Stackelberg equilibrium}, we demonstrate that the decentralized optimal controls of the leader and followers are
the approximate Stackelberg equilibrium of the nonlinear open-loop mean field Stackelberg game. In Section \ref{application to robots}, we apply the theoretical results developed in this paper to solve a nonlinear mean field Stackelberg game problem between a robot control center and unicycle-type swarm robots.

\section{Problem Formulation}\label{problem formulation}

\subsection{Notaion}

\hspace{4mm} $(\Omega,\mathcal{F},\mathbb{P})$ is a complete probability space. Throughout this paper, we use $\mathbb R^n$ to denote $n$-dimensional
Euclidean space, with the usual norm and the usual inner product
given by $|\cdot|$ and $\langle \cdot, \cdot \rangle$. $W_0, W_1, \cdots, W_N$ are Brownian motions independent of each other and defined on
$(\Omega,\mathcal{F},\mathbb{P})$. $W_0$ takes values in $\mathbb{R}^{j_0}$. And $W_1, \cdots, W_N$ all take values in $\mathbb{R}^{j}$. We define $u:=\left\{u_1, u_2, \cdots, u_N\right\}, u_{-i}=\left\{u_1, \ldots,
u_{i-1}, u_{i+1}, \ldots, u_N\right\}$,
\begin{equation*}
\begin{aligned}
C(0, T; \mathbb{R}^n)&:=\Big\{\psi \Big|\psi: [0, T] \longrightarrow \mathbb{R}^n \textrm{ is a continuous function}\Big\}.\\
\end{aligned}
\end{equation*}
For any filtration $\mathbb{G}=\big\{\mathbb{G}_t\big\}_{0\le t\le T}$ defined on $(\Omega,\mathcal{F},\mathbb{P})$ and any $p\in [1, \infty)$, we
introduce the following notations:\\
\begin{equation*}
\begin{aligned}
L^p_{\mathbb{G}}\left(\Omega; C(0, T; \mathbb{R}^n)\right)&:=\Big\{\psi \Big|\psi: \Omega \times[0, T] \longrightarrow \mathbb{R}^n, \textrm{ is a
continuous } \mathbb{G}\textrm{-adapted process such that}\\
&\qquad\qquad\ \mathbb{E}\sup\limits_{0\le t\le T}|\psi(t)|^p<\infty \Big\},\\
L^p_{\mathbb{G}}\left(0, T; \mathbb{R}^n\right)&:=\Big\{\psi \Big|\psi: \Omega \times[0, T] \longrightarrow \mathbb{R}^n, \textrm{ is an } \mathbb{G}
\textrm{-adapted process such that}\\
&\qquad\qquad\ \mathbb{E}\int_0^T |\psi(t)|^p dt<\infty \Big\}.\\
\end{aligned}
\end{equation*}
And for any sub-$\sigma$-field $\mathcal{G}\subseteq\mathcal{F}$ and any $p\in [1, \infty)$, we define
\begin{equation*}
\begin{aligned}
L^p_{\mathcal{G}}\left(\Omega; \mathbb{R}^n\right)&:=\Big\{\psi \Big|\psi: \Omega \longrightarrow \mathbb{R}^n, \textrm{ is an } \mathcal{G}
\textrm{-measurable random variable such that }\mathbb{E}|\psi|^p<\infty \Big\}.\\
\end{aligned}
\end{equation*}

\subsection{$1+N$-type leader-follower game formulation}

\hspace{4mm} For given terminal time $T>0$ and $i=1,2, \cdots, N$, the state evolutions of the leader and the $i^{th}$ follower are denoted by $x_0$
and $x_i$. And they satisfied the following {\it stochastic differential equations} (SDEs, in short):
\begin{equation}\label{the leader's state euqation}
\left\{
             \begin{array}{lr}
             d x_0(t)=b_0\left(t, x_0(t), u_0(t), x^{(N)}(t)\right) dt+\sigma_0 dW_0(t), & 0\le t \le T, \\
             x_0(0)=\xi_0,
             \end{array}
\right.
\end{equation}

\begin{equation}\label{the ith follower's state euqation}
\left\{
             \begin{array}{lr}
             d x_i(t)=b_1\left(t, x_i(t), u_i(t), x_{0}(t), u_0(t), x^{(N)}(t)\right) dt+\widetilde{\sigma} dW_i(t), & 0\le t \le T, \\
             x_i(0)=\xi_i,
             \end{array}
\right.
\end{equation}
where $x^{(N)}(\cdot)=\frac{1}{N} \sum\limits_{j=1}^N x_j(\cdot)$ is the empirical state-average. And $x_0(\cdot)\in \mathbb{R}^k$, $x_i(\cdot) \in
\mathbb{R}^n, u_0(\cdot)\in U_0$, $u_i(\cdot) \in U_i$, where $U_0$ are non-empty convex subsets of $\mathbb{R}^{m_0}$ and $U_i$ are non-empty convex subsets of $\mathbb{R}^m$, for $i=1,2, \cdots, N$. $u_0$ and $u_i$ are control variables of the leader and the $i^{th}$ follower, respectively. $b_0:[0, T] \times \mathbb{R}^k \times \mathbb{R}^{m_0}
\times \mathbb{R}^n \longrightarrow \mathbb{R}^k, b_1:[0, T] \times \mathbb{R}^n \times \mathbb{R}^m \times \mathbb{R}^k \times \mathbb{R}^{m_0} \times
\mathbb{R}^n \longrightarrow \mathbb{R}^n,$ are measurable functions and $\sigma_0, \widetilde{\sigma}$ are constant matrices with compatible dimensions. $\{\xi_i\}_{i\in\{0, 1, 2,
\cdots, N\}}$ are the leader's and followers' initial states. $W_0$ and $W_i$ are respectively independent Brownian motions in $\mathbb{R}^{j_0}$ and $\mathbb{R}^j$, for $i=1, 2, \cdots, N$. They need to satisfy the following assumptions:

\noindent {\bf (H2.1)}\ {\it {(i) The initial random variables $\{\xi_i\}_{i\in\{1, 2, \cdots, N\}}$ are independent and identically distributed,
and independent of $\xi_0$. And $\mathbb{E}\left[\xi_0\right]=\bar{x}_0, \mathbb{E}\left[\xi_i\right]=\bar{x}, \sup\limits_{i \ge 0} \mathbb{E}
\left[\left|\xi_i\right|^2\right] \le c<\infty$.\\
(ii) $\left\{W_i(t), i=0, 1, 2, \cdots, N\right\}$ are independent $1$-dimensional Brownian motions, which are also independent of
$\{\xi_i, i=0, 1, 2, \cdots, N\}$.}}

The cost functionals for the leader and the $i^{th}$ follower to minimize are as follows:
\begin{equation}\label{the leader's cost functional}
\begin{aligned}
J_0^N\left(u_0, u\right)=\mathbb{E}\left\{\int_0^T g_0\left(t, x_0(t), u_0(t), x^{(N)}(t)\right) dt+G_0\left(x_0(T)\right)\right\}.\\
\end{aligned}
\end{equation}
\begin{equation}\label{the ith follower's cost functional}
\begin{aligned}
J_i^N\left(u_i, u_{-i}, u_0\right)=\mathbb{E}\left\{\int_0^T g_1\left(t, x_i(t), u_i(t), x_0(t), u_0(t), x^{(N)}(t)\right) dt+G_1\left(x_i(T)\right)
\right\},\\
\end{aligned}
\end{equation}
for $i=1, 2, \cdots, N$. And $g_0:[0,T] \times \mathbb{R}^k \times \mathbb{R}^{m_0} \times \mathbb{R}^n \longrightarrow \mathbb{R}, g_1:[0,T] \times
\mathbb{R}^n \times \mathbb{R}^m \times \mathbb{R}^k \times \mathbb{R}^{m_0} \times \mathbb{R}^n \longrightarrow \mathbb{R}, G_0: \mathbb{R}^k
\longrightarrow \mathbb{R}, G_1: \mathbb{R}^n \longrightarrow \mathbb{R}$ are measurable functions.

For $i=0, 1, 2, \cdots, N$, denote $\big(\xi_i, W_i(\cdot)\big)$ by $\zeta_i(\cdot)$ and define $\mathbb {F}_t=\sigma\big(\zeta_i(s), 0\le s
\le t, i=0, 1, 2, \cdots, N\big),$ $\mathbb {F}_t^i=\sigma\left(\zeta_i(s), 0\le s \leq t\right), t\in [0,T]$. And
$\mathbb {F}=\big\{\mathbb {F}_t\big\}_{0\le t\le T}, \mathbb {F}^i=\big\{\mathbb {F}_t^i\big\}_{0\le t\le T}, i=0,1,2,\cdots, N.$
For $i=1, 2, \cdots, N$, $\mathbb {F}^0$ and $\mathbb {F}^i$ are local (or decentralized) information of the leader and the $i^{th}$ follower,
respectively. And $\mathbb {F}$ denotes the global (or centralized) information.

The following assumptions will be in force throughout the paper, where $z$ denotes the position of the mean field term.

\noindent {\bf (H2.2)}\ {\it {(i) The function $b_0$ is linear growth and continuously differentiable with respect to $(x_0, u_0, z).$ All the first
partial derivatives of $b_0$ with respect to $(x_0, u_0, z)$ are bounded. The function $b_1$ is linear growth and twice continuously differentiable
with respect to $(x_i, u_i, x_0, u_0, z)$. All the first and second partial derivatives of $b_1$ with respect to $(x_i, u_i, x_0, u_0, z)$ are bounded.\\
(ii) The functions $g_0, G_0$ are continuously differentiable with respect to $(x_0, u_0, z)$ and $x_0$, respectively. The function $g_1, G_1$ are
twice continuously differentiable with respect to $(x_i, u_i, x_0, u_0, z)$ and $x_i$, respectively. All the second partial derivatives of $g_1$ and
$G_1$ are bounded. And there exists a constant $C>0$ such that for any $t\in[0,T], x_0\in \mathbb{R}^k, x_i\in \mathbb{R}^n, u_0\in \mathbb{R}^{m_0},
u_i\in \mathbb{R}^m, z\in \mathbb{R}^n,$
\begin{equation*}
\begin{aligned}
&\big(1+|x_0|^2+|u_0|^2+|z|^2\big)^{-1}\big|g_0(t, x_0, u_0, z)\big|\\
&+\big(1+|x_0|+|u_0|+|z|\big)^{-1}\Big(\big|g_{0 x_0}(t, x_0, u_0, z)\big|\\
& +\big|g_{0 u_0}(t, x_0, u_0, z)\big|+\big|g_{0 z}(t, x_0, u_0, z)\big|\Big)\le C,\\
&\big(1+|x_i|^2+|u_i|^2+|x_0|^2+|u_0|^2+|z|^2\big)^{-1}\big|g_1(t, x_i, u_i, x_0, u_0, z)\big|\\
& +\big(1+|x_i|+|u_i|+|x_0|+|u_0|+|z|\big)^{-1}\Big(\big|g_{1 x_i}(t, x_i, u_i, x_0, u_0, z)\big|\\
& +\big|g_{1 u_i}(t, x_i, u_i, x_0, u_0, z)\big|+\big|g_{1 x_0}(t, x_i, u_i, x_0, u_0, z)\big|\\
& +\big|g_{1 u_0}(t, x_i, u_i, x_0, u_0, z)\big|+\big|g_{1 z}(t, x_i, u_i, x_0, u_0, z)\big|\Big)\le C,\\
&\big(1+|x_0|^2\big)^{-1}\big|G_0(x_0)\big|+\big(1+|x_0|\big)^{-1}\big|G_{0 x_0}(x_0)\big|\le C,\\
&\big(1+|x_i|^2\big)^{-1}\big|G_1(x_i)\big|+\big(1+|x_i|\big)^{-1}\big|G_{1 x_i}(x_i)\big|\le C.\\
\end{aligned}
\end{equation*}
}}

\section{Centralized optimal equilibrium and decentralized strategies}\label{section 3}
\subsection{Centralized optimal equilibrium strategies}

In the $1+N$-type leader-follower game, the leader first announces all his strategies on $[0,T]$ at the initial moment and commits to it. According
to the known strategy of the leader, the followers play the Nash game with each other. Then, their respective optimal response strategies over the
entire time interval are given at the initial moment to minimize their respective cost functionals. Since the leader is in a higher position than
the followers, he will know the optimal responses of the followers. Finally, he will make his own optimal strategy based on the followers' optimal
responses to minimize his cost functional. Therefore, in the $1+N$-type leader-follower game, the Stackelberg game is played between the leader and
the followers, and the Nash game is played between the followers.

When $N$ is finite, the game is played as described above. Therefore, the leader and every follower need to give their equilibrium strategies based
on the global (or centralized) information $\mathbb{F}$. We call these strategies centralized equilibrium strategies. In the
{\it adapted open-loop} (AOL, in short) information structure, the centralized admissible strategy spaces for the leader and the $i^{th}$ follower are
\begin{equation*}
\begin{aligned}
\mathcal{U}^c_0&:=\Big\{u^c_0 \Big|u^c_0: \Omega \times[0, T] \longrightarrow U_0 \textrm{ is } \mathbb{F}\textrm{-adapted and } \mathbb{E}\int_0^T
|u^c_0(t)|^2 dt<\infty \Big\},\\
\mathcal{U}^c_i&:=\Big\{u^c_i \Big|u^c_i: \Omega \times[0, T]\times \mathcal{U}^c_0 \longrightarrow U_i, u^c_i(\cdot, u^c_0) \textrm{ is }
\mathbb{F}\textrm{-adapted, and }\\
&\qquad\qquad\ \mathbb{E}\int_0^T |u^c_i(t, u^c_0)|^2 dt <\infty \textrm{ for any } u^c_0\in \mathcal{U}^c_0\Big\},\\
\end{aligned}
\end{equation*}\\
where $U_0$ and $U_i$ are nonempty convex subsets of $\mathbb{R}^{m_0}$ and $\mathbb{R}^{m}$, respectively, for $i=1, 2, \cdots, N$.

Assumptions (H2.2) guarantee that for any $u^c_0\in \mathcal {U}^c_0, u^c_i\in \mathcal {U}^c_i, i=1,2,\cdots, N,$ the state equations
$(\ref{the leader's state euqation})$-$(\ref{the ith follower's state euqation})$ has unique solutions and the cost functional
$(\ref{the leader's cost functional})$ and $(\ref{the ith follower's cost functional})$ are well-defined.

Based on any pre-committed centralized strategy $u^c_{0}$, and the centralized strategies of the followers $u^c=\{{u}^c_{i}\}_{i=1}^{N}$,
we introduce mappings ${\mathcal{L}}_{0}: \left(u^c_{0}, u^c\right) \longrightarrow x_0;\ {\mathcal{L}}_{i}: \left(u^c_{0}, u^c\right)
\longrightarrow x_i, 1 \leq i \leq N;\ {\mathcal{L}}_{(N)}: \left(u^c_{0}, u^c\right) \longrightarrow x^{(N)}$ to represent the realized states
of the leader, the followers, and their state-average. Notice that, under mild conditions, the existence and uniqueness of the above coupled state
equations $(\ref{the leader's state euqation})$ and $(\ref{the ith follower's state euqation})$ can be ensured, so the mappings
${\mathcal{L}}_{0}, {\mathcal{L}}_{i}, {\mathcal{L}}_{(N)}$ are well defined and
${\mathcal{L}}_{(N)}=\frac{1}{N}\sum\limits_{i=1}^{N}{\mathcal{L}}_{i}.$

For $\forall u^c_0\in \mathcal{U}^c_0$, we introduce a mapping
\[\mathcal{I}(u^c_0)=\big\{\mathcal{I}_1(u^c_0), \mathcal{I}_2(u^c_0), \cdots, \mathcal{I}_N(u^c_0)\big\}: \mathcal{U}^c_0\longrightarrow
\mathcal{U}^c_1\times \mathcal{U}^c_2\times \cdots\times \mathcal{U}^c_N\]
that satisfies the following conditions:
\begin{equation}\label{centralized equilibrium condition 1}
J^N_i\big(\mathcal{I}_i(u^c_0), \mathcal{I}_{-i}(u^c_0), u^c_0\big)\le J^N_i\big(u^c_i, \mathcal{I}_{-i}(u^c_0), u^c_0\big),\ \forall u^c_i\in
\mathcal{U}^c_i,\ \forall i=1, 2, \cdots, N.
\end{equation}
And $\mathcal{I}_{-i}(u^c_0)=\big\{\mathcal{I}_1(u^c_0),\cdots, \mathcal{I}_{i-1}(u^c_0), \mathcal{I}_{i+1}(u^c_0), \cdots,
\mathcal{I}_N(u^c_0)\big\},$
\begin{equation*}
\begin{aligned}
J^N_i\big(u^c_i, \mathcal{I}_{-i}(u^c_0), u^c_0\big)=&\mathbb{E}\bigg\{\int_0^T g_1\Big(t, \mathcal{L}_{i}(u^c_{0}, {\mathcal{I}}_{-i}(u^c_{0}),
u^c_{i}), u^c_i, \mathcal{L}_{0}(u^c_{0}, {\mathcal{I}}_{-i}(u^c_{0}), u^c_{i}), u^c_{0},\\
&\quad{\mathcal{L}}_{(N)}(u^c_{0}, {\mathcal{I}}_{-i}(u^c_{0}), u^c_{i})\Big) dt+G_1\left(\mathcal{L}_{i}(u^c_{0},
{\mathcal{I}}_{-i}(u^c_{0}), u^c_{i})(T)\right)\bigg\}.
\end{aligned}
\end{equation*}
$\mathcal{I}(u^c_0)$ is the centralized optimal Nash responses given by all followers based on the pre-commitment strategy $u^c_0$ announced by
the leader. If there exists $u^{c,*}_0\in \mathcal{U}^c_0$ satisfying
\begin{equation}\label{centralized equilibrium condition 2}
J^N_0\big(\mathcal{I}(u^{c,*}_0), u^{c,*}_0\big)\le J^N_0\big(\mathcal{I}(u^{c}_0), u^{c}_0\big),\ \forall u^c_0\in \mathcal{U}^c_0.
\end{equation}

Then, we call the strategies $\big(u^{c,*}_0, \mathcal{I}(u^{c,*}_0)\big)$ that satisfied conditions $(\ref{centralized equilibrium condition 1})$,
$(\ref{centralized equilibrium condition 2})$ as the centralized optimal equilibrium strategies in the $1+N$-type leader-follower game.

\subsection{Decentralized strategies and limit state-average of followers}\label{section 2.3}

In the preceding section, for $i=0, 1, 2, \cdots, N$, we denote $\big(\xi_i, W_i(\cdot)\big)$ by $\zeta_i(\cdot)$ and define
$\mathbb {F}_t=\sigma\big(\zeta_i(s), 0\le s \le t, i=0, 1, 2, \cdots, N\big),$ $\mathbb {F}_t^i=\sigma\left(\zeta_i(s), 0\le s \leq t\right),
t\in [0,T]$. When $N$ is finite,The leader announces any centralized admissible strategy (but maybe not optimal) $u_0^c\in \mathcal{U}^c_0$
in {\it open-loop} (OL, in short) sense at beginning time that mainly depends on the filtration generated by $\zeta_0$, although also depends
on the filtrations generated by other Brownian motions $\{\zeta_{i}\}_{i=1}^{N}$. Based on the leader's strategy $u_0^c\in \mathcal{U}^c_0$,
followers play Nash games. For each $i=1, 2, \cdots, N$, the $i^{th}$ follower announces his centralized admissible strategy
(but maybe not optimal) ${u}^c_{i}\in \mathcal{U}^c_i$ in {\it open-loop} (OL, in short) sense at beginning time that mainly depends on the
filtration generated by $\zeta_0$ and $\zeta_i$, although also depends on the filtrations generated by other Brownian motions
$\{\zeta_{j}\}_{j=1, j\ne i}^N$.

Roughly, the open-loop means that $u^c_0, u^c_i$ are all $\mathbb{F}$-adapted for $i=1, 2, \cdots, N$. Therefore,
$u^c_0(t, \omega)=L_0(\zeta_0(\omega, \cdot\wedge t), \zeta_1(\omega, \cdot\wedge t), \cdots, \zeta_N(\omega, \cdot\wedge t)),
u^c_i(t, \omega)=L_i(\zeta_0(\omega, \cdot\wedge t), \zeta_1(\omega, \cdot\wedge t), \cdots, \zeta_N(\omega, \cdot\wedge t))$ for
some progressive measurable functionals $L_0, L_i, i=1, 2, \cdots, N$. We say $u^c_0$ mainly depends on the filtration generated by
$\zeta_0$ if the following rate estimates hold:

\begin{equation}\begin{aligned}
& \bigg{|}\frac{\partial{L_0(\zeta_0(\omega, \cdot\wedge t), \zeta_1(\omega, \cdot\wedge t), \cdots, \zeta_N(\omega, \cdot\wedge t))}}
{{\partial \zeta_0}}\bigg{|}=O(1);\\
& \bigg{|}\frac{\partial{L_0(\zeta_0(\omega, \cdot\wedge t), \zeta_1(\omega, \cdot\wedge t), \cdots, \zeta_N(\omega, \cdot\wedge t))}}
{{\partial \zeta_1}}\bigg{|}=o(1);\\& \cdots \\
& \bigg{|}\frac{\partial{L_0(\zeta_0(\omega, \cdot\wedge t), \zeta_1(\omega, \cdot\wedge t), \cdots, \zeta_N(\omega, \cdot\wedge t))}}
{{\partial \zeta_N}}\bigg{|}=o(1) \quad  \text{as} \quad N \longrightarrow +\infty\end{aligned}\end{equation}under an appropriate norm $|\cdot|.$
And we say $u^c_i$ mainly depends on the filtration generated by $\zeta_0$ and $\zeta_i$ if the following rate estimates hold:

\begin{equation}\begin{aligned}
& \bigg{|}\frac{\partial{L_i(\zeta_0(\omega, \cdot\wedge t), \zeta_1(\omega, \cdot\wedge t), \cdots, \zeta_N(\omega, \cdot\wedge t))}}
{{\partial \zeta_0}}\bigg{|}=O(1);\\
& \bigg{|}\frac{\partial{L_i(\zeta_0(\omega, \cdot\wedge t), \zeta_1(\omega, \cdot\wedge t), \cdots, \zeta_N(\omega, \cdot\wedge t))}}
{{\partial \zeta_i}}\bigg{|}=O(1) \quad  \text{as} \quad N \longrightarrow +\infty;
\end{aligned}\end{equation}
and for $j=1, 2, \cdots, N, j\ne i$,
\begin{equation}\begin{aligned}
& \bigg{|}\frac{\partial{L_i(\zeta_0(\omega, \cdot\wedge t), \zeta_1(\omega, \cdot\wedge t), \cdots, \zeta_N(\omega, \cdot\wedge t))}}
{{\partial \zeta_j}}\bigg{|}=o(1) \quad  \text{as} \quad N \longrightarrow +\infty.\end{aligned}\end{equation}

As $N \longrightarrow +\infty$, by the symmetric structures among all followers, it is anticipated that the leader's centralized (open-loop)
equilibrium strategy $u_0^c \text{ which is } \mathbb{F}\text{-adapted}$ converges to the leader's decentralized (open-loop) equilibrium
strategy $u_{0} \text{ which is } \mathbb{F}^0\text{-adapted}$, and the follower's centralized (open-loop) equilibrium strategy
$u_i^c \text{ which is } \mathbb{F}\text{-adapted}$ converges to the follower's decentralized (open-loop) equilibrium strategy
$u_{i} \text{ which is } \mathbb{F}^i\vee \mathbb{F}^0\text{-adapted } \big(\mathbb{F}^i\vee \mathbb{F}^0=\sigma\big(\mathbb{F}^i$\\
$\bigcup \mathbb{F}^0\big)\big)$, for $i=1, 2, \cdots$. Hereafter,
for $i=1, 2, \cdots$, $u_{0}$ is $\mathbb{F}^0$-adapted and $u_{i}$ is $\mathbb{F}^i\vee \mathbb{F}^0$-adapted
represents ``\emph{mainly depends}" or ``\emph{Asymptotically depends}" in above sense. For simplicity of notation, we
define $\overline{\mathbb{F}}^i:=\mathbb{F}^i\vee \mathbb{F}^0$ for $i=1, 2, \cdots$.

Define the decentralized admissible strategy spaces of the leader and the $i^{th}$ follower, respectively, as follows:
\begin{equation*}
\begin{aligned}
\mathcal{U}_0&:=\Big\{u_{0} \Big|u_{0}: \Omega \times[0, T] \longrightarrow U_0 \textrm{ is } \mathbb{F}^0\textrm{-adapted and }
\mathbb{E}\int_0^T |u_0(t)|^2 dt <\infty \Big\},\\
\mathcal{U}_i&:=\Big\{u_{i} \Big|u_{i}: \Omega \times[0, T]\times \mathcal{U}_0 \longrightarrow U_i, u_{i}(\cdot, u_{0}) \textrm{ is }
\overline{\mathbb{F}}^i\textrm{-adapted, and }\\
&\qquad\qquad\ \mathbb{E}\int_0^T |u_i(t, u_0)|^2 dt <\infty \textrm{ for any } u_{0}\in \mathcal{U}_0 \Big\},\\
\end{aligned}
\end{equation*}

As $N \longrightarrow +\infty$, the leader adopts decentralized (open-loop) strategy $u_{0}$ which is $\mathbb{F}^0$-adapted.
We will next consider the limit state-average of all followers. Followers are first assumed to adopt exchangeable centralized decisions
$u^c=\{{u}^c_{i}\}_{i=1}^{N}\in \mathcal{U}^c_1\times \mathcal{U}^c_2\times \cdots\times \mathcal{U}^c_N$, respectively. Then, according to
$(\ref{the ith follower's state euqation})$, we know
\begin{equation}
\left\{
             \begin{array}{lr}
             d x_i=b_1\left(t, x_i, u^c_i, x_{0}, u_{0}, x^{(N)}\right) dt+\widetilde{\sigma} dW_i, & 0\le t \le T, \\
             x_i(0)=\xi_i,
             \end{array}
\right.
\end{equation}
where  $x_i=\mathcal{L}_i(u_{0}, u^c), x_0=\mathcal{L}_0(u_{0}, u^c), x^{(N)}=\mathcal{L}_{(N)}(u_{0}, u^c)$. When $N\longrightarrow+\infty$, we get
\begin{equation}\label{a}
\left\{
             \begin{array}{lr}
             d \bigg(\lim\limits_{N \longrightarrow +\infty} \frac{1}{N}\sum\limits_{i=1}^N x_i\bigg)=\lim\limits_{N \longrightarrow +\infty}
             \bigg[\frac{1}{N}\sum\limits_{i=1}^N b_1\left(t, x_i, u^c_i, x_{0}, u_{0}, x^{(N)}\right)\bigg] dt+\lim\limits_{N \longrightarrow
             +\infty}\bigg(\frac{1}{N}\sum\limits_{i=1}^N \widetilde{\sigma} dW_i\bigg),\\
             \lim\limits_{N \longrightarrow +\infty}\bigg(\frac{1}{N}\sum\limits_{i=1}^N x_i(0)\bigg)=\lim\limits_{N \longrightarrow +\infty}
             \bigg(\frac{1}{N}\sum\limits_{i=1}^N \xi_i\bigg),
             \end{array}
\right.
\end{equation}
From the strong law of large numbers, $\lim\limits_{N \longrightarrow +\infty}\frac{1}{N}\sum\limits_{i=1}^N \xi_i=\bar{x}$ and
$\lim\limits_{N \longrightarrow +\infty}\frac{1}{N}\sum\limits_{i=1}^N \widetilde{\sigma} dW_i$ is negligible.

For $i=1, 2, \cdots$, we denote $\left(\zeta_0(\cdot), \zeta_i(\cdot)\right)$ by $\eta_i(\cdot)$. And define
$\mathcal {G}_t =\bigcap\limits_{n=1}^{\infty} \sigma\big(\eta_i(s), i\ge n, 0\le s\le t\big), t\in[0,T],$ which is the tail $\sigma$-algebra of
$\{\eta_i(s), 0\le s\le t, i=1, 2, \cdots\}$. So, $\mathbb{F}^0_t\subseteq \mathcal {G}_t$ for $t\in [0,T]$. Based on the fact that the double limit
is equal to the repeated limit and the tailor-made propagation chaos analysis on the exchangeable sequence of random variables (for example, Theorem 2.1
in Chapter 2 of \cite{CD2019-2}), we derive the following heuristically.
\begin{equation}\label{e}
\begin{aligned}
&\lim_{N \longrightarrow +\infty} \frac{1}{N}\sum\limits_{i=1}^{N}b_1\left(t, x_i, u^c_{i}, x_{0}, u_{0}, x^{(N)}\right)\\
&={\lim_{N \longrightarrow +\infty} \frac{1}{N}\sum\limits_{i=1}^{N}\left[\lim_{N \longrightarrow +\infty}b_1\left(t, a, x_{0}, u_{0}, x^{(N)}
\right)\right]_{a=\left(x_i, u^c_{i}\right)}},\\
&=\lim_{N \longrightarrow +\infty} \frac{1}{N}\sum\limits_{i=1}^{N}b_1\left(t, x_i, u^c_{i}, \breve{x}_{0}, u_{0}, \breve{z}\right),\\
&={\mathbb{E}\left[b_1\left(t, \breve{x}_1, u_{1}, \breve{x}_{0}, u_{0}, \breve{z}\right)|\mathcal {G}_t\right]},\\
\end{aligned}
\end{equation}
where
\begin{equation}\label{bbb}
\begin{aligned}
&\breve{x}_1=\lim\limits_{N \longrightarrow +\infty} x_1=\lim\limits_{N \longrightarrow +\infty} \mathcal{L}_1(u_{0}, u^c),\ u_{1}=\lim_{N
\longrightarrow +\infty} u^c_{1},\\
&\breve{x}_{0}= \lim_{N \longrightarrow +\infty}x_{0}=\lim\limits_{N \longrightarrow +\infty} \mathcal{L}_0(u_{0}, u^c),\\
&\breve{z}=\lim_{N \longrightarrow +\infty} x^{(N)}=\lim_{N \longrightarrow +\infty}\mathcal{L}_{(N)}(u_{0}, u^c).
\end{aligned}
\end{equation}
In the last equality of $(\ref{e})$, we choose follower 1 as the representative due to the symmetric position of the followers. Of course, we could
choose other followers to represent as well.

\begin{mylem}\label{Kolmogorov's zero-one law}
For any $D\in \bigcap\limits_{n=1}^{\infty} \sigma\big(\zeta_i(s), i\ge n, 0\le s\le t\big)$, we have $P(D)=0$ or $1$.
\end{mylem}

According to Lemma $\ref{Kolmogorov's zero-one law}$, we have
\[{\mathbb{E}\left[b_1\left(t, \breve{x}_1, u_{1}, \breve{x}_{0}, u_{0}, \breve{z}\right)|\mathcal {G}_t\right]}={\mathbb{E}\left[b_1
\left(t, \breve{x}_1, u_{1}, \breve{x}_{0}, u_{0}, \breve{z}\right)|\mathbb {F}^0_t\right]}.\]

Take $N \longrightarrow +\infty$ for the leader's state equation $(\ref{the leader's state euqation})$ and the follower's state equation
$(\ref{the ith follower's state euqation})$ in the $1+N$-type leader-follower game. And according to $(\ref{bbb})$, we get the leader's limit
state process $\breve{x}_0$ and the follower's limit state process $\breve{x}_{i}$  satisfy the following equations, respectively.
\begin{equation}\label{the limit leader's state equation}
\left\{\begin{array}{lr}
d \breve{x}_0=b_0\left(t, \breve{x}_0, u_{0}, \breve{z}\right) dt+{\sigma}_0 dW_0,\\
\breve{x}_0(0)=\xi_0.
\end{array}
\right.
\end{equation}

\begin{equation}\label{the limit ith follower's state equation}
\left\{\begin{array}{lr}
d \breve{x}_i=b_1\left(t, \breve{x}_i, u_{i}, \breve{x}_{0}, u_{0}, \breve{z}\right) dt+\widetilde{\sigma} dW_1,\\
\breve{x}_i(0)=\xi_i.
\end{array}
\right.
\end{equation}
for $i=1, 2, \cdots.$ Finally, based on the previous analysis, the limit state-average of all followers $\breve{z}$ satisfies the following equation:
\begin{equation}\label{b}
\left\{\begin{array}{lr}
d \breve{x}_0(t)=b_0\left(t, \breve{x}_0(t), u_{0}(t), \breve{z}(t)\right) dt+\sigma_0 dW_0(t),\\
d \breve{x}_1(t)=b_1\left(t, \breve{x}_1(t), u_{1}(t), \breve{x}_{0}(t), u_{0}(t), \breve{z}(t)\right) dt+\widetilde{\sigma} dW_1(t),\\
d\breve{z}(t)=\mathbb{E}\left[b_1\left(t, \breve{x}_1(t), u_{1}(t), \breve{x}_{0}(t), u_{0}(t), \breve{z}(t)\right)|{\mathbb{F}^0_t}\right]dt,\\
\breve{x}_0(0)=\xi_0,\quad \breve{x}_1(0)=\xi_1,\quad \breve{z}(0)=\bar{x}.
\end{array}
\right.
\end{equation}
We observe that $\breve{z}(t)=\mathbb{E}\big[\breve{x}_1(t)|\mathbb{F}^0_t\big] \text{ for } t\in[0,T]$. Therefore, equation $(\ref{b})$ can be
reduced to
\begin{equation}\label{consistency condition}
\left\{\begin{array}{lr}
d \breve{x}_0(t)=b_0\left(t, \breve{x}_0(t), u_{0}(t), \mathbb{E}\big[\breve{x}_1(t)|\mathbb{F}^0_t\big]\right) dt+\sigma_0 dW_0(t),\\
d \breve{x}_1(t)=b_1\left(t, \breve{x}_1(t), u_{1}(t), \breve{x}_{0}(t), u_{0}(t), \mathbb{E}\big[\breve{x}_1(t)|\mathbb{F}^0_t\big]\right) dt
+\widetilde{\sigma} dW_1(t),\\
\breve{z}(t)=\mathbb{E}\big[\breve{x}_1(t)|\mathbb{F}^0_t\big],\quad 0\le t\le T,\\
\breve{x}_0(0)=\xi_0,\quad \breve{x}_1(0)=\xi_1.
\end{array}
\right.
\end{equation}

\section{Heuristic asymptotic limit of the centralized equilibrium tuple}

\subsection{Limit of the centralized optimal Nash responses for followers}\label{section 2.4.1}

When $N$ is infinite, given the pre-committed decentralized admissible control $u_0\in\mathcal{U}_0$ of the leader, recall that the centralized
optimal Nash response of the $i^{th}$ follower is:
\begin{equation}\label{centralized optimal nash response of i follower}
\begin{aligned}&{\mathcal{I}}_{i}(u_{0}) \in \arg\min\limits_{u^c_i\in\mathcal{U}^c_i} J_i^{N}\left(u^c_i, {\mathcal{I}}_{-i}(u_{0}), u_{0}\right),\\
\end{aligned}\end{equation}
where
\begin{equation}\begin{aligned}
&J_i^{N}\left(u^c_i, {\mathcal{I}}_{-i}(u_{0}), u_{0}\right)\\
&=\mathbb{E}\bigg\{\int_0^T g_1\left(t, \mathcal{L}_{i}(u_{0}, {\mathcal{I}}_{-i}(u_{0}), u^c_{i}), u^c_i, \mathcal{L}_{0}(u_{0},
{\mathcal{I}}_{-i}(u_{0}), u^c_{i}), u_{0}, {\mathcal{L}}_{(N)}(u_{0}, {\mathcal{I}}_{-i}(u_{0}), u^c_{i})\right) dt\\
&+G_1\left(\mathcal{L}_{i}(u_{0}, {\mathcal{I}}_{-i}(u_{0}), u^c_{i})(T)\right)\bigg\},\\
\end{aligned}\end{equation}
for $i=1, 2, \cdots.$

When each follower adopts the centralized optimal Nash response strategy, the state-average of all followers is
${\mathcal{L}}_{(N)}(u_{0}, {\mathcal{I}}(u_{0}))$. When the $i^{th}$ follower adopts an arbitrary centralized admissible strategy
$u^c_i\in \mathcal {U}^c_i$ and all the rest of the followers adopt the centralized optimal Nash response strategies, the state-average of the
followers is ${\mathcal{L}}_{(N)}(u_{0}, {\mathcal{I}}_{-i}(u_{0}), u^c_{i})$. Due to the symmetric status of the followers and the weak-coupling
between their state equations as well as the fact that the number of followers tends to be infinite, the effect of a single follower's control
perturbation on the state-average over the entire follower population can be almost negligible when $N\longrightarrow+\infty$. So we have
$$\lim_{N\longrightarrow+\infty}
{\mathcal{L}}_{(N)}(u_{0}, {\mathcal{I}}_{-i}(u_{0}), u^c_{i})=\lim_{N\longrightarrow+\infty}
{\mathcal{L}}_{(N)}(u_{0}, {\mathcal{I}}_{-i}(u_{0}), {\mathcal{I}}_{i}(u_{0})).$$
And we assume
$$\lim\limits_{N\longrightarrow+\infty}
{\mathcal{L}}_{(N)}(u_{0}, {\mathcal{I}}(u_{0}))=\overline{\mathcal{L}}(u_{0}, {{\overline{\mathcal{I}}}}(u_{0}))=\breve{z}(u_{0},
{{\overline{\mathcal{I}}}}(u_{0})),$$
where $\overline{\mathcal{I}}(u_0):=\lim\limits_{N \longrightarrow +\infty}{\mathcal{I}}(u_0)$ is the limit of the centralized optimal Nash
responses for all followers.

Based on the analysis of the followers' limit state-average in the previous Subsection $\ref{section 2.3}$, it follows that
$\breve{z}(u_{0}, {\overline{\mathcal{I}}}(u_{0}))$ is characterized by
\begin{equation}\label{c}
\left\{\begin{array}{lr}
d {\breve{x}}_0(t)=b_0\left(t, {\breve{x}}_0(t), u_{0}(t),  \breve{z}(t)\right) dt+\sigma_0 dW_0(t),\\
d \breve{x}_1(t)=b_1\left(t, \breve{x}_1(t), {{\overline{\mathcal{I}}}}_1(u_{0})(t), \breve{x}_{0}(t), u_{0}(t), \breve{z}(t)\right) dt
+\widetilde{\sigma} dW_1(t),\\
d\breve{z}(t)=\mathbb{E}\left[b_1\left(t, \breve{x}_1(t), {{\overline{\mathcal{I}}}}_1(u_{0})(t), \breve{x}_{0}(t), u_0(t),
\breve{z}(t)\right)|{\mathbb{F}^0_t}\right]dt,\\
\breve{x}_0(0)=\xi_0,\quad \breve{x}_1(0)=\xi_1,\quad \breve{z}(0)=\bar{x}.
\end{array}
\right.
\end{equation}
We define $\breve{x}_0:={\overline{\mathcal{L}}}_{0}(u_{0}, \breve{z}(u_{0}, {{\overline{\mathcal{I}}}}(u_{0})))$, where $\breve{x}_0$ is
the solution of equation $(\ref{c})$. Ane we also set $x_i:={\overline{\mathcal{L}}}_{i}(u_{0},
\breve{z}(u_{0}, {{\overline{\mathcal{I}}}}(u_{0})), u_i),$ where $x_i$ satisfies the following equation:
\begin{equation}
\left\{
             \begin{array}{lr}
             d x_i(t)=b_1\left(t, x_i(t), u_i(t),  {\overline{\mathcal{L}}}_{0}(u_{0}, \breve{z}(u_{0},
             {{\overline{\mathcal{I}}}}(u_{0})))(t), u_0(t), \breve{z}(u_{0}, {{\overline{\mathcal{I}}}}(u_{0}))(t)\right) dt\\
             \qquad\qquad+\widetilde{\sigma} dW_i(t),\ \ 0\le t \le T, \\
             x_i(0)=\xi_i,
             \end{array}
\right.
\end{equation}
for $\forall u_i\in \mathcal{U}_i,\ i=1, 2, \cdots.$

From the previous analysis of the centralized and decentralized strategies, we have $\lim\limits_{N\longrightarrow+\infty} u^c_i=u_i$.
Therefore, we derive heuristically
$$\lim_{N\longrightarrow+\infty}
{\mathcal{L}}_{i}(u_{0}, {\mathcal{I}}_{-i}(u_{0}), u^c_{i})=\lim_{N\longrightarrow+\infty}
{\mathcal{L}}_{i}(u_{0}, {\mathcal{I}}_{-i}(u_{0}), u_{i})=
{\overline{\mathcal{L}}}_{i}(u_{0}, \breve{z}(u_{0}, {{\overline{\mathcal{I}}}}(u_{0})), u_i),$$
$$\lim_{N\longrightarrow+\infty}
{\mathcal{L}}_{0}(u_{0}, {\mathcal{I}}_{-i}(u_{0}), u^c_{i})=\lim_{N\longrightarrow+\infty}
{\mathcal{L}}_{0}(u_{0}, {\mathcal{I}}_{-i}(u_{0}), u_{i})=
{\overline{\mathcal{L}}}_{0}(u_{0}, \breve{z}(u_{0}, {{\overline{\mathcal{I}}}}(u_{0}))).$$

So, by taking $N\longrightarrow+\infty$ for $(\ref{centralized optimal nash response of i follower})$, we get that the limit of the
centralized optimal Nash responses for followers $\overline{\mathcal{I}}(u_0)$ satisfies:
\begin{equation}
\begin{aligned}
\overline{\mathcal{I}}_{i}(u_{0}) &\in \arg\min\limits_{u_i\in \mathcal{U}_i} \mathbb{E}\bigg\{\int_0^T g_1\left(t,
\overline{\mathcal{L}}_{i}(u_{0}, \breve{z}(u_{0}, {{\overline{\mathcal{I}}}}(u_{0})), u_{i}), u_i,
\overline{\mathcal{L}}_{0}(u_{0}, \breve{z}(u_{0}, {{\overline{\mathcal{I}}}}(u_{0}))), u_{0}, \breve{z}(u_{0},
{{\overline{\mathcal{I}}}}(u_{0}))\right) dt\\
&+G_1\left(\overline{\mathcal{L}}_{i}(u_{0}, \breve{z}(u_{0}, {{\overline{\mathcal{I}}}}(u_{0})), u_{i})(T)\right)\bigg\}.
\end{aligned}\end{equation}

We point out the following two schemes to analyze consistency conditions:

Consistency scheme 1: Firstly, given the followers' state-average limit $m$ which is to be determined, the decentralized optimal control
problems for all followers are obtained by substituting $m$ for the $x^{(N)}$ in the state equations and the cost functionals. Then,
solving the decentralized optimal control problem for the followers, we obtain the corresponding optimal controls
$\big\{u^{*}_i(m)\big\}_{i=1}^{\infty}$ and optimal state trajectories $\big\{x^{*}_i(u_i^{*}(m))\big\}_{i=1}^{\infty}$. Ultimately, the
fixed-point property $m=\lim\limits_{N\longrightarrow+\infty} \frac{1}{N} \sum\limits_{i=1}^N x^{*}_i(u_i^{*}(m))$ needs to be satisfied to
determine $m$.

Consistency scheme 2: Firstly, given the follower's centralized optimal Nash response limit ${{\overline{\mathcal{I}}}}(u_{0})$ which is to be
determined, the corresponding follower's realized (centralized optimal Nash response limit of the followers are already applied) state
$\big\{\breve{x}_i\big\}_{i=1}^{\infty}$ and realized state-average limit $\breve{z}(u_{0}, {{\overline{\mathcal{I}}}}(u_{0}))$ can be obtained.
Then, by using the realized state-average limit of the followers, the related decentralized optimal control problems of the followers can
be constructed. By solving the decentralized optimal control problems, we obtain the optimal controls
$\big\{u^{*}_i(u_{0}, {{\overline{\mathcal{I}}}}(u_{0}))\big\}_{i=1}^{\infty}$. Finally, we require the decentralized optimal controls
$\big\{u^{*}_i(u_{0}, {{\overline{\mathcal{I}}}}(u_{0}))\big\}_{i=1}^{\infty}$ to satisfy the fixed-point property
$\big\{u^{*}_i(u_{0}, {{\overline{\mathcal{I}}}}(u_{0}))\big\}_{i=1}^{\infty}={{\overline{\mathcal{I}}}}(u_{0})$ to determine
${{\overline{\mathcal{I}}}}(u_{0})$.

The mean field Stackelberg game we study involves an infinite number of followers. Therefore, these two schemes are essentially equivalent for
studying the follower's centralized optimal Nash response limit, but may produce different analytical complexity. Usually, Scheme 1 is often
applied that starts from the limit of realized \emph{state-average}, then a fixed-point cycle argument returning to it. Here, we adopt Scheme 2,
starting from the realized \emph{control} of a generic follower, then a fixed-point argument recovering it.

More precisely, we may implement the following detailed steps.

\textbf{Step 1}: For a pre-committed $u_0,$ fix $u_{1}={{\overline{\mathcal{I}}}}_1(u_{0}),$ as implied in $(\ref{c})$. That is, the control
pair $(u_0, \overline{\mathcal{I}}(u_0))$ is fixed with the limit Nash mapping operator $\overline{\mathcal{I}}$ is to be determined.

\textbf{Step 2}: Construct the associated realized state for a generic follower, together with the state of the leader
\begin{equation}
\left\{\begin{array}{lr}
d {\breve{x}}_0(t)=b_0\left(t, {\breve{x}}_0(t), u_{0}(t),  \mathbb{E}\big[{\breve{x}}_1(t)|{\mathbb{F}^0_t}\big]\right) dt+\sigma_0 dW_0(t),\\
d \breve{x}_1(t)=b_1\left(t, \breve{x}_1(t), {\overline{\mathcal{I}}}_1(u_{0})(t), \breve{x}_{0}(t), u_{0}(t), \mathbb{E}\big[{\breve{x}}_1(t)|
{\mathbb{F}^0_t}\big]\right) dt+\widetilde{\sigma} dW_1(t),\\
\breve{x}_0(0)=\xi_0,\quad \breve{x}_1(0)=\xi_1.
\end{array}
\right.
\end{equation}

\textbf{Step 3}: Based on the analysis in Subsection $\ref{section 2.3}$, construct the related realized state-average limit of the followers
$$\breve{z}(t)=\mathbb{E}\big[\breve{x}_1(t)|\mathbb{F}^0_t\big],\ \ t\in[0,T].$$

\textbf{Step 4}: Construct the related decentralized optimal control problem for the generic $i^{th}$ follower:
\begin{equation}
\begin{aligned}
u^{*}_i(u_{0}, {{\overline{\mathcal{I}}}}(u_{0})) &\in \arg\min\limits_{u_i\in \mathcal{U}_i} \mathbb{E}\bigg\{\int_0^T g_1\big(t,
\overline{\mathcal{L}}_{i}(u_{0}, \breve{z}(u_{0}, {{\overline{\mathcal{I}}}}(u_{0})), u_{i}), u_i, \overline{\mathcal{L}}_{0}(u_{0},
\breve{z}(u_{0}, {{\overline{\mathcal{I}}}}(u_{0}))), u_{0},\\
&\quad \breve{z}(u_{0}, {{\overline{\mathcal{I}}}}(u_{0}))\big) dt+G_1\left(\overline{\mathcal{L}}_{i}(u_{0}, \breve{z}(u_{0},
{{\overline{\mathcal{I}}}}(u_{0})), u_{i})(T)\right)\bigg\}.
\end{aligned}\end{equation}
where the $i^{th}$ follower's decentralized controlled state $x_i={\overline{\mathcal{L}}}_{i}(u_{0}, \breve{z}(u_{0},
{{\overline{\mathcal{I}}}}(u_{0})), u_i)$ is given by the following equation
\begin{equation}
\left\{
             \begin{array}{lr}
             d x_i(t)=b_1\bigg(t, x_i(t), u_i(t),  {\overline{\mathcal{L}}}_{0}(u_{0}, \breve{z}(u_{0},
             {{\overline{\mathcal{I}}}}(u_{0})))(t), u_0(t),\\
             \qquad\qquad \breve{z}(u_{0}, {{\overline{\mathcal{I}}}}(u_{0}))(t)\bigg) dt+\widetilde{\sigma} dW_i(t),\ \ 0\le t \le T, \\
             x_i(0)=\xi_i,
             \end{array}
\right.
\end{equation}
for $i=1, 2, \cdots.$ Recall that $\breve{x}_0={\overline{\mathcal{L}}}_{0}(u_{0}, \breve{z}(u_{0}, {{\overline{\mathcal{I}}}}(u_{0})))$.

\textbf{Step 5}: Determine ${{\overline{\mathcal{I}}}}(u_{0})$ by the fixed-point property $\big\{u^{*}_i(u_{0},
{{\overline{\mathcal{I}}}}(u_{0}))\big\}_{i=1}^{\infty}={{\overline{\mathcal{I}}}}(u_{0})$.
$\Box$

Next, we illustrate the advantages of using Scheme 2 to analyze the mean field game between infinite followers:

(1) According to the Steps 1-4 of Scheme 2 above, in the mean field Stackelberg game with infinite followers, the state dynamic of the
decentralized optimal control problem of the generic $i^{th}$ follower $x_i(\cdot)$ is affected by the realized state-average limit of the
followers $\breve{z}(\cdot)$, which has the form of conditional expectations, i.e. $\breve{z}(t)=\mathbb{E}\big[\breve{x}_1(t)|\mathbb{F}^0_t\big]$,
for any
$t\in[0,T]$. This cannot be clearly reflected when using Scheme 1 to analyze the mean field game between followers (See \cite{BCY2016}).

(2) Scheme 2 provides an alternative to the fixed point step by formulating it in a space of controls instead of flows of empirical measures of the
state variables.

\subsection{Limit of the centralized optimal strategy of the leader}\label{the iterative limit of the leader}

\hspace{5mm}As $N \longrightarrow +\infty,$ by the symmetric structure among all minor followers, it is anticipated that the centralized (open-loop)
equilibrium strategy $u_0^c \text{ which is } \mathbb{F}\text{-adapted}$ converges to decentralized (open-loop) equilibrium strategy
$u_{0} \text{ which is } \mathbb{F}^0\text{-adapted}$. According to $(\ref{centralized equilibrium condition 2})$, the leader's centralized
optimal equilibrium strategy $u_{0}^{c, *}$ is
\begin{equation}
\begin{aligned}
&u_{0}^{c, *} \in \arg\min\limits_{u_{0}^{c}\in\mathcal{U}^c_0} J_0^{N}\left(u_0^c, \mathcal{I}(u_0^c)\right).\\
\end{aligned}
\end{equation}

As $N \longrightarrow +\infty$, based on the follower's centralized optimal Nash response limit $\overline{\mathcal{I}}(u_0)$ obtained in the
Subsection $\ref{section 2.4.1}$ and the continuity of mappings, we heuristically derive that the leader's decentralized optimal control
$u_0^{\dag}\in \mathcal{U}_0$ satisfies the following

\begin{equation}
\begin{aligned}
&\overline{J}_0\left(u_0^{\dag}, \overline{\mathcal{I}}(u_0^{\dag})\right) \leq \overline{J}_0\left(u_0,
\overline{\mathcal{I}}(u_0)\right)\ \ \ \forall u_{0} \in {\mathcal{U}}_{0}.
\end{aligned}
\end{equation}
And the cost functional of the leader's decentralized control problem $\overline{J}_0\left(u_0, \overline{\mathcal{I}}(u_0)\right)$ is defined
as follows:
\begin{equation}\begin{aligned}
\overline{J}_0\left(u_{0}, \overline{\mathcal{I}}(u_{0})\right)&=\mathbb{E}\left\{\int_0^T g_0\left(t, \breve{x}_0(t), u_{0}(t),
\breve{z}(t)\right) dt+G_0\left(\breve{x}_0(T)\right)\right\}\\
&=\mathbb{E}\bigg\{\int_0^T g_0\left(t, \overline{\mathcal{L}}_{0}(u_{0}, \breve{z}(u_{0}, \overline{\mathcal{I}}(u_{0})))(t), u_{0}(t),
\breve{z}(u_{0}, \overline{\mathcal{I}}(u_{0}))(t)\right) dt\\
&\quad +G_0\left(\overline{\mathcal{L}}_{0}(u_{0}, \breve{z}(u_{0}, \overline{\mathcal{I}}(u_{0})))(T)\right)\bigg\},
\end{aligned}\end{equation}
where $\breve{x}_0$ and $\breve{z}$ are the solutions of the following equation:
\begin{equation}
\left\{\begin{array}{lr}
d {\breve{x}}_0(t)=b_0\left(t, {\breve{x}}_0(t), u_{0}(t),  \breve{z}(t)\right) dt+\sigma_0 dW_0(t),\\
d \breve{x}_1(t)=b_1\left(t, \breve{x}_1(t), {{\overline{\mathcal{I}}}}_1(u_{0})(t), \breve{x}_{0}(t), u_{0}(t), \breve{z}(t)\right) dt
+\widetilde{\sigma} dW_1(t),\\
d\breve{z}(t)=\mathbb{E}\left[b_1\left(t, \breve{x}_1(t), {{\overline{\mathcal{I}}}}_1(u_{0})(t), \breve{x}_{0}(t), u_0(t),
\breve{z}(t)\right)|{\mathbb{F}^0_t}\right]dt,\\
\breve{x}_0(0)=\xi_0,\quad \breve{x}_1(0)=\xi_1,\quad \breve{z}(0)=\bar{x}.
\end{array}
\right.
\end{equation}

Therefore, the leader's decentralized control problem is to find $u_0^{\dag}\in \mathcal{U}_0$ such that
\begin{equation}
\overline{J}_0\left(u_0^{\dag}, \overline{\mathcal{I}}(u_0^{\dag})\right) \leq \overline{J}_0\left(u_0, \overline{\mathcal{I}}(u_0)\right),
\ \forall u_0\in \mathcal{U}_0.
\end{equation}
After obtaining $u_0^{\dag}\in \mathcal{U}_0$, we derive the consistency condition for the leader-follower-Nash equilibrium in the mean field
Stackelberg game is
\begin{equation}
\left\{\begin{array}{lr}
d {\breve{x}}_0(t)=b_0\left(t, {\breve{x}}_0(t), u_{0}^{\dag}(t),  \breve{z}(t)\right) dt+\sigma_0 dW_0(t),\\
d \breve{x}_1(t)=b_1\left(t, \breve{x}_1(t), {{\overline{\mathcal{I}}}}_1(u_{0}^{\dag})(t), \breve{x}_{0}(t), u_{0}^{\dag}(t), \breve{z}(t)\right) dt
+\widetilde{\sigma} dW_1(t),\\
d\breve{z}(t)=\mathbb{E}\left[b_1\left(t, \breve{x}_1(t), {{\overline{\mathcal{I}}}}_1(u_{0}^{\dag})(t), \breve{x}_{0}(t), u_0^{\dag}(t),
\breve{z}(t)\right)\Big|{\mathbb{F}^0_t}\right]dt,\\
\breve{x}_0(0)=\xi_0,\quad \breve{x}_1(0)=\xi_1,\quad \breve{z}(0)=\bar{x}.
\end{array}
\right.
\end{equation}

\section{The leader's and the followers' decentralized optimal control problems}\label{section 5}

\subsection{The maximum principles for the followers' decentralized optimal control problems}\label{subsection 5.1}

Arbitrarily given the leader's decentralized control $u_0\in\mathcal{U}_0$, and assuming that ${{\overline{\mathcal{I}}}}(u_{0})$ is the
decentralized optimal control for the followers, it follows that the related realized state-average limit of the followers is
\[\breve{z}(t)=\breve{z}(u_{0}, {{\overline{\mathcal{I}}}}(u_{0}))(t)=\mathbb{E}\big[\breve{x}_i(t)|\mathbb{F}^0_t\big],\ \ t\in[0,T],\]
where we choose the generic $i^{th}$ follower as the representative. And $\breve{x}_i$ is the realized state of the generic $i^{th}$ follower
and satisfy the following equation:
\begin{equation}\label{d}
\left\{\begin{array}{lr}
d {\breve{x}}_0(t)=b_0\left(t, {\breve{x}}_0(t), u_{0}(t), \mathbb{E}\big[{\breve{x}}_i(t)|{\mathbb{F}^0_t}\big]\right) dt+\sigma_0 dW_0(t),\\
d \breve{x}_i(t)=b_1\left(t, \breve{x}_i(t), {\overline{\mathcal{I}}}_i(u_{0})(t), \breve{x}_{0}(t), u_{0}(t), \mathbb{E}\big[{\breve{x}}_i(t)|
{\mathbb{F}^0_t}\big]\right) dt+\widetilde{\sigma} dW_i(t),\\
\breve{x}_0(0)=\xi_0,\quad \breve{x}_i(0)=\xi_i.
\end{array}
\right.
\end{equation}
In the previous section, we also denote the solution $\breve{x}_0$ of equation $(\ref{d})$ as ${\overline{\mathcal{L}}}_{0}(u_{0},
\breve{z}(u_{0}, {{\overline{\mathcal{I}}}}(u_{0})))$, denoting the relevant realized state of the leader. With the leader's realized
state $\breve{x}_0$ and the followers' realized state-average limit $\breve{z}$, we can give the state equation of the decentralized optimal
control problem for the generic $i^{th}$ follower as follows:
\begin{equation}
\left\{
             \begin{array}{lr}
             d x_i(t)=b_1\bigg(t, x_i(t), u_i(t), \breve{x}_0(t), u_0(t), \breve{z}(t)\bigg) dt+\widetilde{\sigma} dW_i(t),\ \ 0\le t \le T, \\
             x_i(0)=\xi_i,
             \end{array}
\right.
\end{equation}
for $i=1, 2, \cdots.$

Similarly, we can give the cost functional that needs to be minimized by the generic $i^{th}$ follower in his decentralized optimal control problem:
\begin{equation}
\begin{aligned}
\overline{J}_{i}(u_i, u_{0})=\mathbb{E}\bigg\{\int_0^T g_1\left(t, x_i(t), u_i(t), \breve{x}_0(t), u_{0}(t), \breve{z}(t)\right) dt
+G_1\left(x_i(T)\right)\bigg\},
\end{aligned}\end{equation}
where $u_i\in\mathcal{U}_i,$ for $i=1, 2, \cdots.$

Ultimately, we need ${{\overline{\mathcal{I}}}}(u_{0})$ to satisfy the fixed-point property, i.e.,
\[{{\overline{\mathcal{I}}}}_i(u_{0})\in \arg\min\limits_{u_i\in\mathcal{U}_i} \overline{J}_{i}(u_i, u_{0}),\]
for $\forall i=1, 2, \cdots.$

Define the Hamiltonian function for the generic $i^{th}$ follower as follows:
\begin{equation}
\begin{aligned}
H_1\left(t, x_i, u_i, x_0, u_0, z, p_i\right):= & \left\langle p_i, b_1\left(t, x_i, u_i, x_0, u_0, z\right)\right\rangle
+g_1\left(t, x_i, u_i, x_0, u_0, z\right),
\end{aligned}
\end{equation}
where $H_1:[0, T] \times \mathbb{R}^n \times \mathbb{R}^m \times \mathbb{R}^k \times \mathbb{R}^{m_0} \times \mathbb{R}^n \times \mathbb{R}^n
\rightarrow \mathbb{R}$, for $i=1, 2, \cdots.$

Then, it follows from the maximum principle of classical stochastic optimal control problems (See \cite{B2006}, \cite{W1998}) and the fixed-point
property that there exists a pair of adapted processes $(p_i, l_i, q_i)\in L^2_{\overline{\mathbb{F}}^i}\left(\Omega; C(0, T; \mathbb{R}^n)\right)
\times L^2_{\overline{\mathbb{F}}^i}\left(0, T; \mathbb{R}^{n\times j_0}\right)\times L^2_{\overline{\mathbb{F}}^i}\left(0, T; \mathbb{R}^{n\times j}\right)$ such that
\begin{equation}\label{ith follower Hamiltonian system}
\left\{
             \begin{array}{lr}
             d\breve{x}_i(t)=b_1\left(t, \breve{x}_i(t), {{\overline{\mathcal{I}}}}_i(u_{0})(t), \breve{x}_0(t), u_0(t),
             \mathbb{E}\big[\breve{x}_i(t)|\mathbb{F}^0_t\big]\right) dt+\widetilde{\sigma} dW_i(t), \\
             -dp_i(t)=\bigg\{\left(\frac{\partial b_1}{\partial x_i}\right)^{\top}\left(t, \breve{x}_i(t), {{\overline{\mathcal{I}}}}_i(u_{0})(t),
             \breve{x}_0(t), u_0(t), \mathbb{E}\big[\breve{x}_i(t)|\mathbb{F}^0_t\big]\right)p_i(t)\\
             \qquad\qquad +\frac{\partial g_1}{\partial x_i}\left(t, \breve{x}_i(t), {{\overline{\mathcal{I}}}}_i(u_{0})(t), \breve{x}_0(t), u_0(t),
             \mathbb{E}\big[\breve{x}_i(t)|\mathbb{F}^0_t\big]\right)\bigg\} dt-l_i(t) dW_0(t)-q_i(t) dW_i(t),\\
             \breve{x}_i(0)=\xi_i,\\
             p_i(T)=\frac{\partial G_1}{\partial x_i}\left(\breve{x}_i(T)\right),\\
             \Big\langle {\partial}_{u_i} H_1\left(t,\breve{x}_i(t),{{\overline{\mathcal{I}}}}_i(u_{0})(t), {\breve{x}}_0(t), u_{0}(t),
             \mathbb{E}\big[{\breve{x}}_i(t)|{\mathbb{F}^0_t}\big], p_i(t)\right), u_i-{{\overline{\mathcal{I}}}}_i(u_{0})(t)\Big\rangle\ge 0,\\
             \qquad\qquad\qquad\qquad\qquad\qquad\qquad\qquad\qquad \ \forall u_i\in U_i,\ a.e.\ t\in[0,T],\ a.s..
             \end{array}
\right.
\end{equation}

\subsection{The leader's decentralized optimal control problem formulation}

\hspace{5mm}For any given the leader's decentralized control $u_0\in \mathcal {U}_0$, after we obtained the corresponding maximum principle for the
followers' decentralized optimal control ${{\overline{\mathcal{I}}}}(u_{0})$, we will next consider the maximum principle for the leader's decentralized
optimal control problem. Since the symmetric structure of all followers, the equations satisfied by $\big\{(p_i, l_i, q_i)\big\}_{i=1}^{\infty}$ are
identical. From the last inequality of system $(\ref{ith follower Hamiltonian system})$, it follows that ${{\overline{\mathcal{I}}}}_i(u_{0})$ is
dependent on $p_i$ for $i=1, 2, \cdots$. According to $(\ref{ith follower Hamiltonian system})$ and the symmetric structure of all followers, we
assume
\begin{equation}\label{yy}
{\overline{\mathcal{I}}}_i(u_{0})(t):={\alpha}_1\left(t, \breve{x}_i(t), \breve{x}_0(t), u_0(t), \mathbb{E}\big[\breve{x}_i(t)|\mathbb{F}^0_t\big],
p_i(t)\right),\ t\in[0,T].
\end{equation}

If we choose the follower 1's realized state $\breve{x}_1$ and realized adjoint variable $p_1$ as representatives to consider, the state equation
of the leader's decentralized optimal control problem is
\begin{equation}\label{f}
\left\{
             \begin{array}{lr}
             d {\breve{x}}_0(t)=b_0\left(t, {\breve{x}}_0(t), u_{0}(t),  \mathbb{E}\big[{\breve{x}}_1(t)|{\mathbb{F}^0_t}\big]\right) dt
             +\sigma_0 dW_0(t),\\
             d \breve{x}_1(t)=b_1\left(t, \breve{x}_1(t), {\overline{\mathcal{I}}}_1(u_{0})(t), \breve{x}_{0}(t), u_{0}(t),
             \mathbb{E}\big[{\breve{x}}_1(t)|{\mathbb{F}^0_t}\big]\right) dt+\widetilde{\sigma} dW_1(t),\\
             -dp_1(t)=\bigg\{\left(\frac{\partial b_1}{\partial x_1}\right)^{\top}\left(t, \breve{x}_1(t),
             {{\overline{\mathcal{I}}}}_1(u_{0})(t), \breve{x}_0(t), u_0(t), \mathbb{E}\big[\breve{x}_1(t)|\mathbb{F}^0_t\big]\right)p_1(t)\\
             \qquad\qquad +\frac{\partial g_1}{\partial x_1}\left(t, \breve{x}_1(t), {{\overline{\mathcal{I}}}}_1(u_{0})(t),
             \breve{x}_0(t), u_0(t), \mathbb{E}\big[\breve{x}_1(t)|\mathbb{F}^0_t\big]\right)\bigg\} dt-l_1(t) dW_0(t)-q_1(t) dW_1(t),\\
             \qquad\qquad\qquad\qquad\qquad\qquad\qquad\qquad\qquad\qquad\qquad 0\le t\le T,\\
             \breve{x}_0(0)=\xi_0,\ \breve{x}_1(0)=\xi_1,\ p_1(T)=\frac{\partial G_1}{\partial x_1}\left(\breve{x}_1(T)\right).\\
             \end{array}
\right.
\end{equation}
According to the analysis in Subsection $\ref{the iterative limit of the leader}$, the cost functional of the leader's decentralized control problem is:
\begin{equation}\begin{aligned}
\overline{J}_0\left(u_{0}, \overline{\mathcal{I}}(u_{0})\right)&=\mathbb{E}\left\{\int_0^T g_0\left(t, \breve{x}_0(t), u_{0}(t),
\mathbb{E}\big[{\breve{x}}_1(t)|{\mathbb{F}^0_t}\big]\right) dt+G_0\left(\breve{x}_0(T)\right)\right\}.\\
\end{aligned}\end{equation}
Let
\begin{equation*}
\begin{aligned}
&B_1\left(t, \breve{x}_1(t), \breve{x}_{0}(t), u_{0}(t), \mathbb{E}\big[{\breve{x}}_1(t)|{\mathbb{F}^0_t}\big], p_1(t)\right)\\
&:=b_1\left(t, \breve{x}_1(t), {\alpha}_1\left(t, \breve{x}_1(t), \breve{x}_0(t), u_0(t), \mathbb{E}\big[\breve{x}_1(t)|
\mathbb{F}^0_t\big], p_1(t)\right), \breve{x}_{0}(t), u_{0}(t), \mathbb{E}\big[{\breve{x}}_1(t)|{\mathbb{F}^0_t}\big]\right);\\
&\Phi\left(t, \breve{x}_1(t), \breve{x}_{0}(t), u_{0}(t), \mathbb{E}\big[{\breve{x}}_1(t)|{\mathbb{F}^0_t}\big], p_1(t)\right)\\
&:=\left(\frac{\partial b_1}{\partial x_1}\right)^\top\left(t, \breve{x}_1(t), {\alpha}_1\left(t, \breve{x}_1(t), \breve{x}_0(t), u_0(t),
\mathbb{E}\big[\breve{x}_1(t)|\mathbb{F}^0_t\big], p_1(t)\right), \breve{x}_{0}(t), u_{0}(t), \mathbb{E}\big[{\breve{x}}_1(t)|{\mathbb{F}^0_t}\big]
\right)p_1(t)\\
&\qquad +\frac{\partial g_1}{\partial x_1}\left(t, \breve{x}_1(t), {\alpha}_1\left(t, \breve{x}_1(t), \breve{x}_0(t), u_0(t), \mathbb{E}\big[
\breve{x}_1(t)|\mathbb{F}^0_t\big], p_1(t)\right), \breve{x}_{0}(t), u_{0}(t), \mathbb{E}\big[{\breve{x}}_1(t)|{\mathbb{F}^0_t}\big]\right).
\end{aligned}
\end{equation*}
Then, the state equation of the leader's decentralized optimal control problem $(\ref{f})$ can be rewritten as
\begin{equation}\label{m}
\left\{
             \begin{array}{lr}
             d {\breve{x}}_0(t)=b_0\left(t, {\breve{x}}_0(t), u_{0}(t), \mathbb{E}\big[{\breve{x}}_1(t)|{\mathbb{F}^0_t}\big]\right) dt+\sigma_0
             dW_0(t),\\
             d \breve{x}_1(t)=B_1\left(t, \breve{x}_1(t), \breve{x}_{0}(t), u_{0}(t), \mathbb{E}\big[{\breve{x}}_1(t)|{\mathbb{F}^0_t}\big], p_1(t)\right) dt
             +\widetilde{\sigma} dW_1(t),\\
             -dp_1(t)=\Phi\left(t, \breve{x}_1(t), \breve{x}_0(t), u_0(t), \mathbb{E}\big[\breve{x}_1(t)|\mathbb{F}^0_t\big], p_1(t)\right)dt-l_1(t)
             dW_0(t)-q_1(t) dW_1(t),\\
             \breve{x}_0(0)=\xi_0,\ \breve{x}_1(0)=\xi_1,\ p_1(T)=\frac{\partial G_1}{\partial x_1}\left(\breve{x}_1(T)\right).\\
             \end{array}
\right.
\end{equation}

\subsection{Some well-posedness results of a new kind of conditional mean-field FBSDE}\label{some well-posedness results of a new kind of CMF-FBSDEs}

\hspace{5mm}It can be observed that the state equation of the leader's decentralized optimal control problem $(\ref{m})$ is a fully
coupled mean-field FBSDE containing the conditional expectation term $\mathbb{E}\big[{\breve{x}}_1(t)|{\mathbb{F}^0_t}\big]$. Therefore, in this
subsection, we will use the method of continuation to discuss the well-posedness of a new kind of fully coupled conditional mean-field FBSDE.
Some well-posedness results of this type of conditional mean-field FBSDE are obtained. They will be used to provide the existence, uniqueness and
estimations of the solution to the state equation $(\ref{m})$ in the leader's decentralized optimal control problem, as well as the existence,
uniqueness and  estimations of the solutions to the variational equation and the adjoint equation derived subsequently from the corresponding maximum
principle. The specific form of the conditional mean-field FBSDE we studied is proposed based on the state equation, the subsequent corresponding
variational equation and adjoint equation of the leader's decentralized optimal control problem in the nonlinear mean field Stackelberg game we
investigated. In the subsequent discussion of this article, for any stochastic process $X$, to simplify the notation, we use $\widehat{X}(t)$ to
represent $\mathbb{E}\big[X(t)|{\mathbb{F}^0_t}\big], t\in[0,T]$.

We consider the following new kind of conditional mean-field FBSDE:
\begin{equation}\label{aaa}
\left\{
             \begin{array}{lr}
             d X(t)=\Psi\big(t, X(t), Y(t), \widehat{X}(t), \widehat{Y}(t)\big) dt+\theta_0 dW_0(t)+\theta_1 dW_1(t),\\
             -dY(t)=\Big[D\big(t, X(t), Y(t), \widehat{X}(t), \widehat{Y}(t)\big)+\mathbb{E}\big[ M\big(t, X(t), Y(t), \widehat{X}(t), \widehat{Y}(t)
             \big)\big|{\mathbb{F}}^0_t\big]\Big] dt\\
             \qquad\qquad\ -L_0(t) dW_0(t)-L_1(t) dW_1(t),\\
             X(0)=\chi_0,\ Y(T)=h\big(X(T)\big),\\
             \end{array}
\right.
\end{equation}
where $\Psi: [0,T]\times \Omega\times \mathbb{R}^{n_1}\times \mathbb{R}^{m_1}\times \mathbb{R}^{n_1} \times \mathbb{R}^{m_1}
\longrightarrow\mathbb{R}^{n_1}, D: [0,T]\times \Omega\times \mathbb{R}^{n_1}\times \mathbb{R}^{m_1}\times \mathbb{R}^{n_1} \times \mathbb{R}^{m_1}
\longrightarrow\mathbb{R}^{m_1}, M: [0,T]\times \Omega\times \mathbb{R}^{n_1}\times \mathbb{R}^{m_1}\times \mathbb{R}^{n_1} \times \mathbb{R}^{m_1}
\longrightarrow\mathbb{R}^{m_1}, h:\Omega\times \mathbb{R}^{n_1}\longrightarrow\mathbb{R}^{m_1}$. $\chi_0$ is a random variable that takes value
in $\mathbb{R}^{n_1}$ and $\mathbb{E}\big[|\chi_0|^2\big]<\infty$. $\theta_0, \theta_1$ are constant matrices that respectively takes value in $\mathbb{R}^{n_1\times j_0}$ and $\mathbb{R}^{n_1\times j}$.

To the best of our knowledge, Nie, Wang, and Wang \cite{NWW2023} and Chen, Du, and Wu \cite{CDW2024} have respectively studied the existence and
uniqueness of solutions to different forms of conditional mean-field FBSDEs by different methods. We explain the differences and significance of
our discussion on the new form of conditional mean-field FBSDE $(\ref{aaa})$ in three points.

Firstly, since the conditional mean-field FBSDE
$(\ref{aaa})$ we study in this subsection contains terms like $\mathbb{E}\big[ M\big(t, X(t), Y(t), \widehat{X}(t), \widehat{Y}(t)\big)\big|
{\mathbb{F}}^0_t\big]$, this leads to the fact that the results of \cite{NWW2023} and \cite{CDW2024} cannot be used to provide the existence
and uniqueness of the solution to equation $(\ref{aaa})$. And the study of the fully-coupled conditional mean-field FBSDE $(\ref{aaa})$ containing
such a term as $\mathbb{E}\big[ M\big(t, X(t), Y(t), \widehat{X}(t), \widehat{Y}(t)\big)\big|{\mathbb{F}}^0_t\big]$ is inspired by the form of the
adjoint equation derived from the maximum principle of the leader's decentralized optimal control problem, aiming to obtain results that can guarantee
the existence and uniqueness of its solution. Secondly, unlike \cite{NWW2023} and \cite{CDW2024}, inspired by \cite{W1998}, when proving the existence
of the solution to equation $(\ref{aaa})$, for the two different cases of $n_1\ge m_1$ and $n_1<m_1$, we respectively introduced two families of
simplified  parameterized FBSDEs of distinct forms. These simplified parameterized FBSDEs decouple when the parameters take the initial value $0$,
ensuring the existence and uniqueness of the solutions at this time. This will be beneficial for simplifying the proof of the existence of the
solution of equation $(\ref{aaa})$. Thirdly, inspired by the special forms of the state equation, the variational equation and the adjoint equation
of the leader's decentralized optimal control problem, we attempt to study the fully coupled conditional mean-field FBSDE $(\ref{aaa})$, where
the drift term of the forward equation and the generator of the backward equation do not contain $L_0$ and $L_1$. We find that under the assumptions
(H5.1) and (H5.2), which only impose conditions on variables $X, Y, \widehat{X}$, and $\widehat{Y}$, the existence and uniqueness of the solution
to equation $(\ref{aaa})$ can be guaranteed.

\begin{mydef}
A quadruple of process $\big(X, Y, L_0, L_1\big): [0,T]\times \Omega \longrightarrow\mathbb{R}^{n_1}\times \mathbb{R}^{m_1}\times \mathbb{R}^{m_1\times j_0}
\times \mathbb{R}^{m_1\times j}$ is called an adapted solution of conditional mean-field FBSDE $(\ref{aaa})$ if $\big(X, Y, L_0, L_1\big)\in
L^2_{\overline{\mathbb{F}}^1}\left(0, T; \mathbb{R}^{n_1}\times \mathbb{R}^{m_1}\times \mathbb{R}^{m_1\times j_0}\times \mathbb{R}^{m_1\times j}\right)$ and it
satisfies the conditional mean-field FBSDE $(\ref{aaa})$.
\end{mydef}

Now, we introduce assumptions (H5.1) and (H5.2). Given an $m_1\times n_1$ full-rank matrix $G$ and we introduce the following assumption

\noindent {\bf (H5.1)}\ {\it {(i) For every $\big(x, y, z^x, z^y\big)\in \mathbb{R}^{n_1}\times \mathbb{R}^{m_1}\times \mathbb{R}^{n_1}\times
\mathbb{R}^{m_1}$, we have $\Psi\big(\cdot, x, y, z^x, z^y\big)\in L^2_{\overline{\mathbb{F}}^1}\left(0, T; \mathbb{R}^{n_1}\right),
D\big(\cdot, x, y, z^x, z^y\big)\in L^2_{\overline{\mathbb{F}}^1}\left(0, T; \mathbb{R}^{m_1}\right), M\big(\cdot, x, y, z^x, z^y\big)\in
L^2_{\overline{\mathbb{F}}^1}\left(0, T; \mathbb{R}^{m_1}\right)$\\
(ii) For every $x \in \mathbb{R}^{n_1}, h(x)\in L^2_{\overline{\mathbb{F}}^1_T}\left(\Omega; \mathbb{R}^{m_1}\right)$.\\
(iii) $\Psi\big(\cdot, x, y, z^x, z^y\big), D\big(\cdot, x, y, z^x, z^y\big), M\big(\cdot, x, y, z^x, z^y\big)$ is uniformly Lipschitz with
respect to $x, y, z^x$ and $z^y$, respectively.\\
(iv) $h(x)$ is uniformly Lipschitz with respect to $x$.}}

For any $\big(X^1, Y^1\big), \big(X^2, Y^2\big)\in L^2_{\overline{\mathbb{F}}^1}\left(0, T; \mathbb{R}^{n_1}\times \mathbb{R}^{m_1}\right)$,
we need the following monotone conditions.

\noindent {\bf (H5.2)}\ {\it {
\begin{equation*}
\begin{aligned}
&\left[\left\langle
\left(\begin{array}{ccc}
-G^\top \Delta D(t)-G^\top\mathbb{E}\big[\Delta M(t)\big|{\mathbb{F}}^0_t\big]\\
G \Delta \Psi(t)\\
\end{array}\right), \left(\begin{array}{ccc}
\Delta X(t)\\
\Delta Y(t)\\
\end{array}\right)\right\rangle\right]\\
&\le -\beta^1_1\mathbb{E}\left[\left|G\Delta X(t)\right|^2\right]-\beta^2_1\mathbb{E}\left[\left|G\Delta \widehat{X}(t)\right|^2\right]
-\beta^1_2\mathbb{E}\left[\left|G^\top \Delta Y(t)\right|^2\right]-\beta^2_2\mathbb{E}\left[\left|G^\top \Delta \widehat{Y}(t)\right|^2\right],\\
&\mathbb{E}\left[\left\langle h\left(X^1(T)\right)-h\left(X^2(T)\right), G\Delta X(T)\right\rangle\right]\ge
\alpha_1^1 \mathbb{E}\left[\left|G\Delta X(T)\right|^2\right],
\end{aligned}
\end{equation*}
}}
where $\Delta R(\cdot):=R\big(\cdot, X^1(\cdot), Y^1(\cdot), \widehat{X^1}(\cdot), \widehat{Y^1}(\cdot)\big)-R\big(\cdot, X^2(\cdot),
Y^2(\cdot), \widehat{X^2}(\cdot), \widehat{Y^2}(\cdot)\big), R=D, M, \Psi;$\\ $\Delta X(\cdot):=X^1(\cdot)-X^2(\cdot),
\Delta Y(\cdot):=Y^1(\cdot)-Y^2(\cdot), \Delta \widehat{X}(t):= \widehat{X^1}(\cdot)-\widehat{X^2}(\cdot),
\Delta \widehat{Y}(t):= \widehat{Y^1}(\cdot)-\widehat{Y^2}(\cdot)$; and $\beta^1_1, \beta^2_1, \beta^1_2, \beta^2_2, \alpha^1_1$ are non-negative
constants with $\beta^1_1+\beta^1_2>0, \alpha^1_1+\beta^1_2>0$ for almost all $(t, \omega) \in [0,T]\times\Omega$. Moreover we have
$\beta^1_1>0, \alpha^1_1>0$ (resp., $\beta^1_2>0$) when $m_1>n_1$ (resp., $n_1\ge m_1$).

\begin{Remark}
When $n_1=m_1$, we can assume either $\beta^1_1>0, \alpha^1_1>0$ or $\beta^1_2>0$. Without loss of generality, we assume $\beta^1_2>0$ here.
\end{Remark}

\begin{mylem}\label{uniqueness lemma of CMF-FBSDE}
Let (H5.1) and (H5.2) hold. Then there exists at most one adapted solution to conditional mean-field FBSDE $(\ref{aaa})$.
\end{mylem}

\begin{proof}
Let $\big(X^1, Y^1, L^1_0, L^1_1\big), \big(X^2, Y^2, L^2_0, L^2_1\big) \in L^2_{\overline{\mathbb{F}}^1}\left(0, T; \mathbb{R}^{n_1}\times
\mathbb{R}^{m_1}\times \mathbb{R}^{m_1\times j_0}\times \mathbb{R}^{m_1\times j}\right)$ be two solutions of the conditional mean-field FBSDE $(\ref{aaa})$.
Set $\Delta S:=S^1-S^2, S=X, Y, L_0, L_1; \Delta \widehat{X}(t):= \widehat{X^1}(\cdot)-\widehat{X^2}(\cdot),
\Delta \widehat{Y}(t):= \widehat{Y^1}(\cdot)-\widehat{Y^2}(\cdot).$ First, apply It\^{o}'s formula to
$\big\langle\Delta Y(t), G \Delta X(t)\big\rangle$. Then, integrate from $0$ to $T$ and take the expectation. According to assumption (H5.2), we can
obtain
\begin{equation}\label{ccc}
\begin{aligned}
&\alpha^1_1\mathbb{E}\big[|G\Delta X(T)|^2\big]+\beta^1_1 \mathbb{E}\int_0^T |G\Delta X(t)|^2 dt+\beta^2_1 \mathbb{E}\int_0^T |G\Delta
\widehat{X}(t)|^2 dt\\
&+\beta^1_2 \mathbb{E}\int_0^T |G^\top\Delta Y(t)|^2 dt+\beta^2_2 \mathbb{E}\int_0^T |G^\top\Delta \widehat{Y}(t)|^2 dt\equiv 0.
\end{aligned}
\end{equation}

(1) When $n_1\ge m_1, \beta^1_2>0$. Then, we have $\mathbb{E}\int_0^T |G^\top\Delta Y(t)|^2 dt\equiv 0$. So, $Y^1(\cdot)=Y^2(\cdot),
\ a.e.\ t\in[0,T],\ a.s..$ By using usual techniques and Gronwall's inequality, the estimations of the difference $\Delta X$ can be derived
\[\mathbb{E}\big[\sup\limits_{0\le t\le T}|\Delta X(t)|^2\big]\le K \mathbb{E}\int_0^T |\Delta Y(s)|^2 ds\equiv 0,\]
where $K$ denotes a generic constant. It is noted that its values in different lines may vary in the subsequent discussions. Thus
$\mathbb{P}\big(\big\{\omega|X^1(t,\omega)=X^2(t,\omega), \forall t\in[0,T]\big\}\big)=1$. Applying It\^{o}'s formula to
$|\Delta Y(t)|^2\equiv 0,\ a.e. t\in [0,T],\ a.s.,$ integrating from $0$ to $T$, and taking expectation, we get
$\mathbb{E}\int_0^T |\Delta L_0(t)|^2+|\Delta L_1(t)|^2 dt\equiv 0$. Therefore, we have $L^1_0(\cdot)=L^2_0(\cdot), L^1_1(\cdot)=L^2_1(\cdot),
\ a.e.\ t\in[0,T],\ a.s..$

(2) When $n_1\le m_1, \beta^1_1>0, \alpha^1_1>0$. From $(\ref{ccc})$, it follows that $\mathbb{E}\big[|G\Delta X(T)|^2\big]=0,
\mathbb{E}\int_0^T |G\Delta X(t)|^2 dt=0$. So, we can derive $X^1(\cdot)=X^2(\cdot),\ a.e.\ t\in[0,T],\ a.s.$ and $X^1(T)=X^2(T),\ a.s..$

Since the generator of BSDE in equation $(\ref{aaa})$ contains $\mathbb{E}\big[ M\big(t, X(t), Y(t), \widehat{X}(t), \widehat{Y}(t)\big)
\big|{\mathbb{F}}^0_t\big]$, this form of conditional mean-field FBSDE has not been studied in previous relevant literature
(e.g. \cite{HT2022}, \cite{NWW2023}, \cite{CDW2024}). Currently, there are no existing conclusions that can be used. Therefore, we need to estimate
the differences of solutions for the BSDE, namely $\Delta Y, \Delta L_0, \Delta L_1$. We apply It\^{o}'s formula to $|\Delta Y(s)|^2$, integrate
from $t$ to $T$, and get
\begin{equation}\label{ddd}
\begin{aligned}
&|\Delta Y(t)|^2+\int_t^T |\Delta L_0 (s)|^2 ds+\int_t^T |\Delta L_1 (s)|^2 ds\\
&=|\Delta Y(T)|^2+2\int_t^T \left\langle\Delta Y(s), \big[\Delta D(s)+\mathbb{E}\big[\Delta M(s)\big|{\mathbb{F}}^0_s\big]\big] \right\rangle ds\\
&\quad -2\int_t^T \left\langle\Delta Y(s), \Delta L_0 (s)\right\rangle dW_0(s)-2\int_t^T \left\langle\Delta Y(s), \Delta L_1 (s)\right\rangle dW_1(s),
\end{aligned}
\end{equation}
for $\forall t\in [0,T]$. Based on some classical techniques and Gronwall's inequality, we obtained that there exists a positive constant $K$ such that
\begin{equation}\label{eee}
\begin{aligned}
&\mathbb{E}\big[|\Delta Y(t)|^2\big]+\mathbb{E}\int_t^T |\Delta L_0 (s)|^2 ds+\mathbb{E}\int_t^T |\Delta L_1 (s)|^2 ds\\
&\le K\mathbb{E}\int_0^T |\Delta X (s)|^2 ds+K\mathbb{E}\big[|\Delta X(T)|^2\big]=0,
\end{aligned}
\end{equation}
for $\forall t\in [0,T]$. By further estimation based on $(\ref{ddd})$ and {\it Burkholder-Davis-Gundy's inequality} (BDG inequality, in short), it
can be concluded that there exists a positive constant $K$ such that
\begin{equation}\label{fff}
\begin{aligned}
&\mathbb{E}\big[\sup\limits_{0\le t\le T}|\Delta Y(t)|^2\big]\le K\mathbb{E}\big[|\Delta X(T)|^2\big]+K\mathbb{E}\int_0^T \left(|\Delta X (s)|^2
+|\Delta Y (s)|^2\right) ds\\
&+K\mathbb{E}\int_0^T \left(|\Delta L_0 (s)|^2|+\Delta L_1 (s)|^2\right) ds=0.
\end{aligned}
\end{equation}
There, according to $(\ref{eee})$ and $(\ref{fff})$, we know $L^1_0(\cdot)=L^2_0(\cdot), L^1_1(\cdot)=L^2_1(\cdot)\ a.e.\ t\in[0,T],\ a.s.$ and
$\mathbb{P}\big(\big\{\omega|Y^1(t,\omega)=Y^2(t,\omega), \forall t\in[0,T]\big\}\big)=1$.

\end{proof}

Next, we prove the existence of the solution to conditional mean-field FBSDE $(\ref{aaa})$. First, for the case of $n_1\ge m_1$, introduce the
following FBSDE parameterized by $\lambda\in [0,1]$:
\begin{equation}\label{n1>m1 parameterized FBSDE}
\left\{
             \begin{array}{lr}
             d X^{\lambda}(t)=\left[(1-\lambda)\beta^1_2(-G^{\top} Y^{\lambda}(t))+\lambda\Psi\big(t, X^{\lambda}(t), Y^{\lambda}(t),
             \widehat{X^{\lambda}}(t), \widehat{Y^{\lambda}}(t)\big)+\phi(t)\right] dt\\
             \qquad\qquad\quad +\theta_0 dW_0(t)+\theta_1 dW_1(t),\\
             -dY^{\lambda}(t)=\Big[\lambda D\big(t, X^{\lambda}(t), Y^{\lambda}(t), \widehat{X^{\lambda}}(t), \widehat{Y^{\lambda}}(t)\big)
             +\lambda\mathbb{E}\big[ M\big(t, X^{\lambda}(t), Y^{\lambda}(t), \widehat{X^{\lambda}}(t), \widehat{Y^{\lambda}}(t)\big)
             \big|{\mathbb{F}}^0_t\big]\\
             \qquad\qquad\quad +\Lambda(t)\Big] dt-L^{\lambda}_0(t) dW_0(t)-L^{\lambda}_1(t) dW_1(t),\\
             X^{\lambda}(0)=\chi_0,\ Y^{\lambda}(T)=\lambda h\big(X^{\lambda}(T)\big)+\xi,\\
             \end{array}
\right.
\end{equation}
where $\xi\in L^2_{\overline{\mathbb{F}}^1_T}\left(\Omega; \mathbb{R}^{m_1}\right)$, $\phi$ and $\Lambda$ are given process in
$L^2_{\overline{\mathbb{F}}^1}\left(0, T; \mathbb{R}^{n_1}\right)$ and $L^2_{\overline{\mathbb{F}}^1}\left(0, T; \mathbb{R}^{m_1}\right)$ respectively.
Then, we introduce the following lemma.

\begin{mylem}\label{The iterative lemma of n1>m1 parameterized FBSDE}
Let (H5.1), (H5.2) and $n_1\ge m_1$ hold. Then there exists a positive constant $\delta_0$, such that if, a prior, for $\lambda=\lambda_0\in[0,1)$
there exists a unique solution $\big(X^{\lambda_0}, Y^{\lambda_0}, L_0^{\lambda_0}, L_1^{\lambda_0}\big)\in L^2_{\overline{\mathbb{F}}^1}
\left(0, T; \mathbb{R}^{n_1}\times \mathbb{R}^{m_1}\times \mathbb{R}^{m_1\times j_0}\times \mathbb{R}^{m_1\times j}\right)$ of equation
$(\ref{n1>m1 parameterized FBSDE})$, then for each $\delta\in [0,\delta_0]$ there exists a unique solution
$\big(X^{\lambda_0+\delta}, Y^{\lambda_0+\delta}, L_0^{\lambda_0+\delta}, L_1^{\lambda_0+\delta}\big)\in L^2_{\overline{\mathbb{F}}^1}
\left(0, T; \mathbb{R}^{n_1}\times \mathbb{R}^{m_1}\times \mathbb{R}^{m_1\times j_0}\times \mathbb{R}^{m_1\times j}\right)$ of equation
$(\ref{n1>m1 parameterized FBSDE})$ for $\lambda=\lambda_0+\delta$.
\end{mylem}

\begin{proof}
To ensure the rationality of the mapping definition when proving the existence of the solution later, we first briefly prove the uniqueness of the
solution of equation $(\ref{n1>m1 parameterized FBSDE})$. For any $t\in[0,T]$, we set
\begin{equation*}
\begin{aligned}
\Psi^{\lambda}\big(t, X^{\lambda}(t), Y^{\lambda}(t), \widehat{X^{\lambda}}(t), \widehat{Y^{\lambda}}(t)\big)&:=(1-\lambda)
\beta^1_2(-G^{\top} Y^{\lambda}(t))\\
&\quad\ +\lambda\Psi\big(t, X^{\lambda}(t), Y^{\lambda}(t), \widehat{X^{\lambda}}(t), \widehat{Y^{\lambda}}(t)\big)+\phi(t),\\
D^{\lambda}\big(t, X^{\lambda}(t), Y^{\lambda}(t), \widehat{X^{\lambda}}(t), \widehat{Y^{\lambda}}(t)\big)&:=\lambda
D\big(t, X^{\lambda}(t), Y^{\lambda}(t), \widehat{X^{\lambda}}(t), \widehat{Y^{\lambda}}(t)\big)+\Lambda(t),\\
M^{\lambda}\big(t, X^{\lambda}(t), Y^{\lambda}(t), \widehat{X^{\lambda}}(t), \widehat{Y^{\lambda}}(t)\big)&:=
\lambda M\big(t, X^{\lambda}(t), Y^{\lambda}(t), \widehat{X^{\lambda}}(t), \widehat{Y^{\lambda}}(t)\big),\\
h^{\lambda}\big(X^{\lambda}(T)\big)&:=\lambda h\big(X^{\lambda}(T)\big)+\xi.
\end{aligned}
\end{equation*}
It is easy to verify that $\Psi^{\lambda}, D^{\lambda}, M^{\lambda}, h^{\lambda}$ satisfy the corresponding assumptions (H5.1) and (H5.2) with
respect to the variables $\big(X^{\lambda}, Y^{\lambda}, \widehat{X^{\lambda}}, \widehat{X^{\lambda}}\big)$. From Lemma
$\ref{uniqueness lemma of CMF-FBSDE}$, we know that for any $\lambda\in[0,1]$, equation $(\ref{n1>m1 parameterized FBSDE})$ has at most one
corresponding solution $\big(X^{\lambda}, Y^{\lambda}, L_0^{\lambda}, L_1^{\lambda}\big)\in L^2_{\overline{\mathbb{F}}^1}
\left(0, T; \mathbb{R}^{n_1}\times \mathbb{R}^{m_1}\times \mathbb{R}^{m_1\times j_0}\times \mathbb{R}^{m_1\times j}\right)$.

Next, let's further discuss the existence of solutions. Let $\delta_0$ be a positive constant which will be determined later on, and let
$\delta\in [0,\delta_0]$. Since for each $\phi \in L^2_{\overline{\mathbb{F}}^1}\left(0, T; \mathbb{R}^{n_1}\right),
\Lambda\in L^2_{\overline{\mathbb{F}}^1}\left(0, T; \mathbb{R}^{m_1}\right), \xi\in L^2_{\overline{\mathbb{F}}^1_T}\left(\Omega; \mathbb{R}^{m_1}\right)$,
there exists a unique solution of equation $(\ref{n1>m1 parameterized FBSDE})$ for $\lambda=\lambda_0\in[0,1)$. Therefore, for any
$\widetilde{U}:=\big(\widetilde{X}, \widetilde{Y}, \widetilde{L}_0, \widetilde{L}_1\big)\in L^2_{\overline{\mathbb{F}}^1}\left(0, T; \mathbb{R}^{n_1}
\times \mathbb{R}^{m_1}\times \mathbb{R}^{m_1\times j_0}\times \mathbb{R}^{m_1\times j}\right)$, there exists a unique solution $U=\big(X, Y, L_0, L_1\big)
\in L^2_{\overline{\mathbb{F}}^1}\left(0, T; \mathbb{R}^{n_1}\times \mathbb{R}^{m_1}\times \mathbb{R}^{m_1\times j_0}\times \mathbb{R}^{m_1\times j}\right)$
satisfying the following conditional mean-field FBSDE:
\begin{equation}
\left\{
             \begin{array}{lr}
             d X(t)=\Big[(1-\lambda_0)\beta^1_2(-G^{\top} Y(t))+\lambda_0\Psi\big(t, X(t), Y(t), \widehat{X}(t), \widehat{Y}(t)\big)\\
             \qquad\qquad\ +\delta\big(\beta^1_2G^{\top} \widetilde{Y}(t)+\Psi\big(t, \widetilde{X}(t), \widetilde{Y}(t), \widehat{\widetilde{X}}(t),
             \widehat{\widetilde{Y}}(t)\big)\big)+\phi(t)\Big] dt\\
             \qquad\qquad\ +\theta_0 dW_0(t)+\theta_1 dW_1(t),\\
             -dY(t)=\Big[\lambda_0 D\big(t, X(t), Y(t), \widehat{X}(t), \widehat{Y}(t)\big)+\lambda_0\mathbb{E}\big[ M\big(t, X(t), Y(t),
             \widehat{X}(t), \widehat{Y}(t)\big)\big|{\mathbb{F}}^0_t\big]\\
             \qquad\qquad\ +\delta D\big(t, \widetilde{X}(t), \widetilde{Y}(t), \widehat{\widetilde{X}}(t), \widehat{\widetilde{Y}}(t)\big)+\delta
             \mathbb{E}\big[M\big(t, \widetilde{X}(t), \widetilde{Y}(t), \widehat{\widetilde{X}}(t), \widehat{\widetilde{Y}}(t)\big)\big|
             {\mathbb{F}}^0_t\big]\\
             \qquad\qquad\ +\Lambda(t)\Big] dt-L_0(t) dW_0(t)-L_1(t) dW_1(t),\\
             X(0)=\chi_0,\ Y(T)=\lambda_0 h\big(X(T)\big)+\delta h\big(\widetilde{X}(T)\big)+\xi.\\
             \end{array}
\right.
\end{equation}

We now proceed to prove that, if $\delta$ is sufficiently small, the mapping defined by
\begin{equation*}
\begin{aligned}
I_{\lambda_0+\delta}\big(\big(\widetilde{U}, \widetilde{X}(T)\big)\big)=\big(U, X(T)\big): &L^2_{\overline{\mathbb{F}}^1}\left(0, T;
\mathbb{R}^{n_1}\times \mathbb{R}^{m_1}\times \mathbb{R}^{m_1\times j_0}\times \mathbb{R}^{m_1\times j}\right)\times L^2_{\overline{\mathbb{F}}^1_T}\left(\Omega;
\mathbb{R}^{n_1}\right)\\
&\longrightarrow L^2_{\overline{\mathbb{F}}^1}\left(0, T; \mathbb{R}^{n_1}\times \mathbb{R}^{m_1}\times \mathbb{R}^{m_1\times j_0}\times \mathbb{R}^{m_1\times j}\right)
\times L^2_{\overline{\mathbb{F}}^1_T}\left(\Omega; \mathbb{R}^{n_1}\right)
\end{aligned}
\end{equation*}
is a contraction.

Let $\widetilde{U}':=\big(\widetilde{X}', \widetilde{Y}', \widetilde{L}'_0, \widetilde{L}'_1\big)\in L^2_{\overline{\mathbb{F}}^1}\left(0, T;
\mathbb{R}^{n_1}\times \mathbb{R}^{m_1}\times \mathbb{R}^{m_1\times j_0}\times \mathbb{R}^{m_1\times j}\right),$ $\big(U', X'(T)\big):=$\\
$I_{\lambda_0+\delta}\big(\big(\widetilde{U}', \widetilde{X}'(T)\big)\big),$ and $U':=\big(X', Y', L'_0, L'_1\big)$. Set
\[\Delta \widetilde{U}:=\big(\Delta \widetilde{X}, \Delta \widetilde{Y}, \Delta \widetilde{L}_0, \Delta \widetilde{L}_1 \big):=\big(\widetilde{X}-
\widetilde{X}', \widetilde{Y}-\widetilde{Y}', \widetilde{L}_0-\widetilde{L}'_0, \widetilde{L}_1-\widetilde{L}'_1 \big),\]
\[\Delta U:=\big(\Delta X, \Delta Y, \Delta L_0, \Delta L_1 \big):=\big(X-X', Y-Y', L_0-L'_0, L_1-L'_1 \big),\]
\[\Delta \widehat{X}:=\widehat{X}-\widehat{X'},\ \Delta \widehat{Y}:=\widehat{Y}-\widehat{Y'}, \Delta \widehat{\widetilde{X}}:=\widehat{\widetilde{X}}-
\widehat{\widetilde{X}'},\ \Delta \widehat{\widetilde{Y}}:=\widehat{\widetilde{Y}}-\widehat{\widetilde{Y}'}.\]

After applying It\^{o}'s formula to $\big\langle\Delta Y(s), G \Delta X(s)\big\rangle$ and taking the expectation from $0$ to $T$, based on assumptions
(H5.1) and (H5.2), along with some basic derivations, it can be concluded that there exists a positive constant $K$ such that
\begin{equation}\label{hhh}
\begin{aligned}
&\mathbb{E}\int_0^T |G^\top \Delta Y(s)|^2 ds\\
&\le K\delta \mathbb{E}\int_0^T \left(|\Delta X(s)|^2+|\Delta Y(s)|^2+|\Delta \widetilde{X}(s)|^2+|\Delta \widetilde{Y}(s)|^2\right) ds\\
&\quad +K\delta \mathbb{E}\big[|\Delta X(T)|^2+|\Delta \widetilde{X}(T)|^2\big].
\end{aligned}
\end{equation}
By carrying out some similar estimation ideas and techniques as those in the proof of the uniqueness lemma $\ref{uniqueness lemma of CMF-FBSDE}$, it can
be proved that there exists a positive constant $K$ such that
\begin{equation}\label{iii}
\begin{aligned}
&\mathbb{E}\big[\sup\limits_{0\le t\le T}|\Delta X(t)|^2 \big]\le K\mathbb{E}\int_0^T |\Delta Y(s)|^2ds+K\delta\mathbb{E}\int_0^T\left(|\Delta
\widetilde{X}(s)|^2+|\Delta \widetilde{Y}(s)|^2\right) ds,\\
&|\Delta Y(0)|^2+\mathbb{E}\int_0^T |\Delta L_0(s)|^2ds+\mathbb{E}\int_0^T |\Delta L_1(s)|^2ds\\
&\le K\mathbb{E}\int_0^T |\Delta Y(s)|^2ds+K\delta\mathbb{E}\int_0^T\left(|\Delta \widetilde{X}(s)|^2+|\Delta \widetilde{Y}(s)|^2\right) ds+K\delta
\mathbb{E}\big[|\Delta \widetilde{X}(T)|^2\big].
\end{aligned}
\end{equation}
Combining $(\ref{hhh})$ and $(\ref{iii})$, it can be derived that
\begin{equation}\label{jjj}
\begin{aligned}
&\mathbb{E}\int_0^T \left(|\Delta X(s)|^2+|\Delta Y(s)|^2+|\Delta L_0(s)|^2+|\Delta L_1(s)|^2\right)ds+\mathbb{E}\big[|\Delta X(T)|^2\big]\\
&\le K\delta\mathbb{E}\int_0^T \left(|\Delta X(s)|^2+|\Delta Y(s)|^2+|\Delta \widetilde{X}(s)|^2+|\Delta \widetilde{Y}(s)|^2\right)ds\\
&\quad +K\delta \mathbb{E}\big[|\Delta X(T)|^2+|\Delta \widetilde{X}(T)|^2\big],
\end{aligned}
\end{equation}
where the value of $K$ is independent of the selection of $\lambda_0\in[0,1)$. Then, we can take a sufficiently small $\delta$, for instance, let
$\delta_0=\frac{1}{3K}$. Then for any $\delta\in[0,\delta_0]$, it can be deduced that
\begin{equation}\label{kkk}
\begin{aligned}
&\mathbb{E}\int_0^T \left(|\Delta X(s)|^2+|\Delta Y(s)|^2+|\Delta L_0(s)|^2+|\Delta L_1(s)|^2\right)ds+\mathbb{E}\big[|\Delta X(T)|^2\big]\\
&\le \frac{1}{2}\mathbb{E}\int_0^T \left(|\Delta \widetilde{X}(s)|^2+|\Delta \widetilde{Y}(s)|^2+|\Delta \widetilde{L}_0(s)|^2+|\Delta
\widetilde{L}_1(s)|^2\right)ds+\frac{1}{2} \mathbb{E}\big[|\Delta \widetilde{X}(T)|^2\big].
\end{aligned}
\end{equation}
Therefore, for any $\delta\in[0,\delta_0]$, it can be obtained that $I_{\lambda_0+\delta}$ is a contraction mapping, and there exists a unique fixed
point such that $I_{\lambda_0+\delta}\big(\big(U, X(T)\big)\big)=\big(U, X(T)\big)$, where $\big(U, X(T)\big)\in L^2_{\overline{\mathbb{F}}^1}
\left(0, T; \mathbb{R}^{n_1}\times \mathbb{R}^{m_1}\times \mathbb{R}^{m_1\times j_0}\times \mathbb{R}^{m_1\times j}\right)\times L^2_{\overline{\mathbb{F}}^1_T}
\left(\Omega; \mathbb{R}^{n_1}\right)$. So, for $\lambda=\lambda_0+\delta$, equation $(\ref{n1>m1 parameterized FBSDE})$ has a unique solution
$\big(X^{\lambda_0+\delta}, Y^{\lambda_0+\delta}, L_0^{\lambda_0+\delta}, L_1^{\lambda_0+\delta}\big)\in L^2_{\overline{\mathbb{F}}^1}
\left(0, T; \mathbb{R}^{n_1}\times \mathbb{R}^{m_1}\times \mathbb{R}^{m_1\times j_0}\times \mathbb{R}^{m_1\times j}\right)$.

\end{proof}

When $n_1<m_1$, introduce the following FBSDE parameterized by $\tau\in [0,1]$:
\begin{equation}\label{n1<m1 parameterized FBSDE}
\left\{
             \begin{array}{lr}
             d X^{\tau}(t)=\left[\tau\Psi\big(t, X^{\tau}(t), Y^{\tau}(t),
             \widehat{X^{\tau}}(t), \widehat{Y^{\tau}}(t)\big)+\psi(t)\right] dt\\
             \qquad\qquad\quad +\theta_0 dW_0(t)+\theta_1 dW_1(t),\\
             -dY^{\tau}(t)=\Big[(1-\tau)\beta^1_1G X^{\tau}(t)+\tau D\big(t, X^{\tau}(t), Y^{\tau}(t), \widehat{X^{\tau}}(t),
             \widehat{Y^{\tau}}(t)\big)\\
             \qquad\qquad\quad +\tau\mathbb{E}\big[ M\big(t, X^{\tau}(t), Y^{\tau}(t), \widehat{X^{\tau}}(t), \widehat{Y^{\tau}}(t)\big)
             \big|{\mathbb{F}}^0_t\big]+\Gamma(t)\Big] dt\\
             \qquad\qquad\quad -L^{\tau}_0(t) dW_0(t)-L^{\tau}_1(t) dW_1(t),\\
             X^{\tau}(0)=\chi_0,\ Y^{\tau}(T)=\tau h\big(X^{\tau}(T)\big)+(1-\tau)G X^{\tau}(T)+\chi,\\
             \end{array}
\right.
\end{equation}
where $\chi\in L^2_{\overline{\mathbb{F}}^1_T}\left(\Omega; \mathbb{R}^{m_1}\right)$, $\psi$ and $\Gamma$ are given process in
$L^2_{\overline{\mathbb{F}}^1}\left(0, T; \mathbb{R}^{n_1}\right)$ and $L^2_{\overline{\mathbb{F}}^1}\left(0, T; \mathbb{R}^{m_1}\right)$ respectively.
 And we introduce the following lemma.

\begin{mylem}\label{The iterative lemma of n1<m1 parameterized FBSDE}
Let (H5.1), (H5.2) and $n_1< m_1$ hold. Then there exists a positive constant $\kappa_0$, such that if, a prior, for $\tau=\tau_0\in[0,1)$
there exists a unique solution $\big(X^{\tau_0}, Y^{\tau_0}, L_0^{\tau_0}, L_1^{\tau_0}\big)\in L^2_{\overline{\mathbb{F}}^1}
\left(0, T; \mathbb{R}^{n_1}\times \mathbb{R}^{m_1}\times \mathbb{R}^{m_1\times j_0}\times \mathbb{R}^{m_1\times j}\right)$ of equation
$(\ref{n1<m1 parameterized FBSDE})$,
then for each $\kappa\in [0,\kappa_0]$ there exists a unique solution $\big(X^{\tau_0+\kappa}, Y^{\tau_0+\kappa}, L_0^{\tau_0+\kappa},
L_1^{\tau_0+\kappa}\big)\in L^2_{\overline{\mathbb{F}}^1}\left(0, T; \mathbb{R}^{n_1}\times \mathbb{R}^{m_1}\times \mathbb{R}^{m_1\times j_0}\times
\mathbb{R}^{m_1\times j}\right)$ of equation $(\ref{n1<m1 parameterized FBSDE})$ for $\tau=\tau_0+\kappa$.
\end{mylem}

\begin{proof}
By applying the techniques used in the proof of Lemma $\ref{uniqueness lemma of CMF-FBSDE}$, combining them with the proof of Lemma 2.3 in
\cite{W1998},
and following similar steps as those in
the proof of Lemma $\ref{The iterative lemma of n1>m1 parameterized FBSDE}$, this lemma can be proved. For the sake of brevity,
the details are not elaborated here.
\end{proof}

\begin{mythm}\label{the existence and uniqueness of new kind of CMF-FBSDE}
Let (H5.1) and (H5.2) hold. Then there exists a unique adapted solution $\left(X, Y, L_0, L_1\right)\in L^2_{\overline{\mathbb{F}}^1}
\left(\Omega; C(0, T; \mathbb{R}^{n_1})\right)\times L^2_{\overline{\mathbb{F}}^1}\left(\Omega; C(0, T; \mathbb{R}^{m_1})\right)
\times L^2_{\overline{\mathbb{F}}^1}\left(0, T; \mathbb{R}^{m_1\times j_0}\right)\times L^2_{\overline{\mathbb{F}}^1}\big(0, T;$\\
$\mathbb{R}^{m_1\times j}\big)$
to conditional mean-field FBSDE $(\ref{aaa})$.
\end{mythm}

\begin{proof}
From Lemma $\ref{uniqueness lemma of CMF-FBSDE}$, the uniqueness of the solution to equation $(\ref{aaa})$ can be obtained.
Next, we consider the existence of the solution to equation $(\ref{aaa})$.
When $n_1\ge m_1$, the parameterized FBSDE $(\ref{n1>m1 parameterized FBSDE})$ decouples at $\lambda=0$ and has a unique solution.
By Lemma $\ref{The iterative lemma of n1>m1 parameterized FBSDE}$, there exists a positive constant $\delta_0$ such that for any
$\delta\in[0, \delta_0]$, equation $(\ref{n1>m1 parameterized FBSDE})$ has a unique solution for $\lambda=\delta$. We can repeat the step N
times until $1\le N\delta_0< 1+\delta_0$. Then, for $\lambda=1$, equation $(\ref{n1>m1 parameterized FBSDE})$ has a unique solution. When taking
$\phi(\cdot)\equiv 0, \Lambda(\cdot)\equiv 0, \xi\equiv 0$, it can be known that the conditional mean-field FBSDE $(\ref{aaa})$ has a unique solution.
When $n_1<m_1$, by combining parameterized FBSDE $(\ref{n1<m1 parameterized FBSDE})$ and Lemma $\ref{The iterative lemma of n1<m1 parameterized FBSDE}$,
and following similar steps as when $n_1\ge m_1$, it can be proved that the conditional mean-field FBSDE $(\ref{aaa})$ has a unique solution.

When $\left(X, Y, L_0, L_1\right)\in L^2_{\overline{\mathbb{F}}^1}\left(0, T; \mathbb{R}^{n_1}\times \mathbb{R}^{m_1}\times \mathbb{R}^{m_1\times j_0}
\times \mathbb{R}^{m_1\times j}\right)$ is the unique solution of the conditional mean-field FBSDE $(\ref{aaa})$, by Jensen's inequality,
Gronwall's inequality,
BDG inequality, and some usual techniques, there exists a positive constant $K$ such that
\begin{equation*}
\begin{aligned}
\mathbb{E}\left[\sup\limits_{0\le t\le T} |X(t)|^2\right]&\le K\left[1+\mathbb{E}\int_0^T |Y(s)|^2 ds\right]<\infty,\\
\mathbb{E}\left[\sup\limits_{0\le t\le T} |Y(t)|^2\right]&\le K\bigg\{1+\mathbb{E}[|X(T)|^2]+\mathbb{E}\int_0^T |X(s)|^2 ds\\
&\qquad+\mathbb{E}\int_0^T |L_0(s)|^2 ds+\mathbb{E}\int_0^T |L_1(s)|^2 ds\bigg\}<\infty.\\
\end{aligned}
\end{equation*}
Then, $X(\cdot)\in L^2_{\overline{\mathbb{F}}^1}\left(\Omega; C(0, T; \mathbb{R}^{n_1})\right), Y(\cdot)\in L^2_{\overline{\mathbb{F}}^1}
\left(\Omega; C(0, T; \mathbb{R}^{m_1})\right)$
\end{proof}

\subsection{The maximum principle for the leader's decentralized optimal control problem}\label{zz}

We set
\[X_1(t):=\left(\begin{array}{ccc}
{\breve{x}}_0(t)\\
\breve{x}_1(t)\\
\end{array}\right),
\Phi_1\big(t, X_1(t), p_1(t), \widehat{X}_1(t), u_0(t)\big):=\Phi\left(t, \breve{x}_1(t), \breve{x}_0(t), u_0(t),
\widehat{\breve{x}}_1(t), p_1(t)\right)\\,
\]
\[\Psi_1\big(t, X_1(t), p_1(t), \widehat{X}_1(t), u_0(t)\big):=\left(\begin{array}{ccc}
b_0\left(t, {\breve{x}}_0(t), u_{0}(t), \widehat{\breve{x}}_1(t)\right)\\
B_1\left(t, \breve{x}_1(t), \breve{x}_{0}(t), u_{0}(t), \widehat{\breve{x}}_1(t), p_1(t)\right)\end{array}\right),\]
for any $(t, \omega)\in [0,T]\times \Omega$. And let
\[X_0:=\left(\begin{array}{ccc}
\xi_0\\
\xi_1\\
\end{array}\right),\ h_1\big(X_1(T)\big):=\frac{\partial G_1}{\partial x_1}\left(\breve{x}_1(T)\right),\ \sigma_{10}:=\left(\begin{array}{ccc}
\sigma_0\\
\mathbf{0}_{n\times j_0}\\
\end{array}\right),\ \sigma_{11}:=\left(\begin{array}{ccc}
\mathbf{0}_{n\times j}\\
\widetilde{\sigma}\\
\end{array}\right).\]
Then, the state equation of the leader's decentralized optimal control problem $(\ref{m})$ is equivalent to
\begin{equation}\label{Integrated leader's decentralized state equation}
\left\{
             \begin{array}{lr}
             d X_1(t)=\Psi_1\big(t, X_1(t), p_1(t), \widehat{X}_1(t), u_0(t)\big) dt+\sigma_{10} dW_0(t)+\sigma_{11} dW_1(t),\\
             -dp_1(t)=\Phi_1\big(t, X_1(t), p_1(t), \widehat{X}_1(t), u_0(t)\big) dt-l_1(t) dW_0(t)-q_1(t) dW_1(t),\\
             X_1(0)=X_0,\ p_1(T)=h_1\big(X_1(T)\big).\\
             \end{array}
\right.
\end{equation}
It can be observed that the form of equation $(\ref{Integrated leader's decentralized state equation})$ is encompassed by the form of the conditional
mean-field FBSDE $(\ref{aaa})$ in the Subsection $\ref{some well-posedness results of a new kind of CMF-FBSDEs}$. Therefore, by Theorem
$\ref{the existence and uniqueness of new kind of CMF-FBSDE}$, if $\Psi_1, \Phi_1, h_1$ satisfy assumptions (H5.1) and (H5.2), then equation
$(\ref{Integrated leader's decentralized state equation})$ has a unique solution $\left(X_1, p_1, l_1, q_1\right)\in L^2_{\overline{\mathbb{F}}^1}
\left(\Omega; C(0, T; \mathbb{R}^{k+n})\right)\times L^2_{\overline{\mathbb{F}}^1}\left(\Omega; C(0, T; \mathbb{R}^{n})\right)
\times L^2_{\overline{\mathbb{F}}^1}\left(0, T; \mathbb{R}^{n\times j_0}\right)\times L^2_{\overline{\mathbb{F}}^1}\left(0, T; \mathbb{R}^{n\times j}\right)$.

In the subsequent discussion of this paper, all our analyses are conducted under the premise that $\Psi_1, \Phi_1, h_1$ satisfy assumptions (H5.1) and
(H5.2). Therefore, for any $u_0\in \mathcal {U}_0$, the state equation of the leader's decentralized optimal control problem $(\ref{m})$ has a unique
solution $\left(\breve{x}_0(\cdot), \breve{x}_1(\cdot), p_1(\cdot), l_1(\cdot), q_1(\cdot)\right)\in L^2_{\overline{\mathbb{F}}^1}
\left(\Omega; C(0, T; \mathbb{R}^k)\right)\times L^2_{\overline{\mathbb{F}}^1}\left(\Omega; C(0, T; \mathbb{R}^n)\right)
\times L^2_{\overline{\mathbb{F}}^1}\left(\Omega; C(0, T; \mathbb{R}^n)\right)$
$\times L^2_{\overline{\mathbb{F}}^1}\left(0, T; \mathbb{R}^{n\times j_0}\right)\times L^2_{\overline{\mathbb{F}}^1}\left(0, T; \mathbb{R}^{n\times j}\right)$.

Define the Hamiltonian function for the leader as follows:
\begin{equation}
\begin{aligned}
&H_0\left(t, x_1, x_0, u_0, z, p_1, K_0, K_1, \varphi_1\right)\\
&:= \left\langle K_0, b_0\left(t, x_0, u_0, z\right)\right\rangle+\left\langle K_1, B_1\left(t, x_1, x_0, u_0, z, p_1\right)\right\rangle \\
&\quad-\left\langle \varphi_1, \Phi\left(t, x_1, x_0, u_0, z, p_1\right)\right\rangle+g_0\left(t, x_0, u_0, z\right),
\end{aligned}
\end{equation}
where $H_0:[0, T] \times \mathbb{R}^n \times \mathbb{R}^k \times \mathbb{R}^{m_0} \times \mathbb{R}^n \times \mathbb{R}^n \times \mathbb{R}^k
\times \mathbb{R}^n \times \mathbb{R}^n \rightarrow \mathbb{R}.$

To prove the maximum principle for the leader's decentralized optimal control problem, we need the following assumptions.

\noindent {\bf (H5.3)}\ {\it ${\alpha}_1$ in $(\ref{yy})$ is continuously differentiable with respect to $(x_i, x_0, u_0, z, p_i)$ and all the first
partial derivatives of ${\alpha}_1$ are bounded. And $\Phi$ in $(\ref{m})$ is also continuously differentiable with respect to $(x_1, x_0, u_0, z, p_1)$
and all the first partial derivatives of $\Phi$ are bounded.}

Let $u_{0}^{\dag}\in \mathcal {U}_0$ be the optimal control of the leader's decentralized optimal control problem. According to Theorem
$\ref{the existence and uniqueness of new kind of CMF-FBSDE}$ and the previous analysis, there exists a unique optimal state trajectory
$\left(\breve{x}_0^{\dag}(\cdot), \breve{x}_1^{\dag}(\cdot), p_1^{\dag}(\cdot), l_1^{\dag}(\cdot), q_1^{\dag}(\cdot)\right)
\in L^2_{\overline{\mathbb{F}}^1}\left(\Omega; C(0, T; \mathbb{R}^k)\right)\times L^2_{\overline{\mathbb{F}}^1}\left(\Omega; C(0, T; \mathbb{R}^n)\right)
\times L^2_{\overline{\mathbb{F}}^1}\left(\Omega; C(0, T; \mathbb{R}^n)\right)$
$\times L^2_{\overline{\mathbb{F}}^1}\left(0, T; \mathbb{R}^{n\times j_0}\right)\times L^2_{\overline{\mathbb{F}}^1}\left(0, T; \mathbb{R}^{n\times j}\right)$ that satisfies $(\ref{m})$. Let $v_0(\cdot)$ be such that $u_{0}^{\dag}(\cdot)+v_0(\cdot)\in \mathcal {U}_0$. Since $U_0$ is convex, then for any
$0\le \rho \le 1$, $u^{\rho}_0(\cdot)=u_{0}^{\dag}(\cdot)+\rho v_0(\cdot)\in \mathcal {U}_0$. When the leader's decentralized control is
$u^{\rho}_0$, it is also known from Theorem \ref{the existence and uniqueness of new kind of CMF-FBSDE} that there exists a unique corresponding
trajectory $\left(\breve{x}^{\rho}_0(\cdot), \breve{x}^{\rho}_1(\cdot), p^{\rho}_1(\cdot), l^{\rho}_1(\cdot), q^{\rho}_1(\cdot)\right)
\in L^2_{\overline{\mathbb{F}}^1}\left(\Omega; C(0, T; \mathbb{R}^k)\right)\times L^2_{\overline{\mathbb{F}}^1}
\left(\Omega; C(0, T; \mathbb{R}^n)\right)\times L^2_{\overline{\mathbb{F}}^1}\left(\Omega; C(0, T; \mathbb{R}^n)\right)
\times L^2_{\overline{\mathbb{F}}^1}\left(0, T; \mathbb{R}^{n\times j_0}\right)\times L^2_{\overline{\mathbb{F}}^1}\left(0, T; \mathbb{R}^{n\times j}\right)$ that satisfies
$(\ref{m})$.

Based on some classic techniques and similar proof steps to Theorem \ref{the existence and uniqueness of new kind of CMF-FBSDE},
we can derive the following lemma. And the specific proof procedure is omitted here.

\begin{mylem}\label{state estimation}
Let (H2.1), (H2.2), (H5.3) and the assumptions made earlier in this subsection all hold, then we have
$\big(\breve{x}^{\rho}_0(\cdot), \breve{x}^{\rho}_1(\cdot), p^{\rho}_1(\cdot),$ $l^{\rho}_1(\cdot), q^{\rho}_1(\cdot)\big)$ converges to
$\left(\breve{x}_0^{\dag}(\cdot), \breve{x}_1^{\dag}(\cdot), p_1^{\dag}(\cdot), l_1^{\dag}(\cdot), q_1^{\dag}(\cdot)\right)$ in
$L^2_{\overline{\mathbb{F}}^1}\big(0,T; \mathbb{R}^k\times \mathbb{R}^n\times \mathbb{R}^n\times \mathbb{R}^{n\times j_0}\times \mathbb{R}^{n\times j}\big)$ as
$\rho$ tends to 0.

\end{mylem}

We set $b_0(t)=b_0\big(t, \breve{x}_0^{\dag}(t), u^{\dag}_{0}(t), \widehat{\breve{x}_1^{\dag}}(t)\big), B_1(t)=B_1\big(t, \breve{x}_1^{\dag}(t),
\breve{x}_0^{\dag}(t), u^{\dag}_{0}(t), \widehat{\breve{x}_1^{\dag}}(t), p_1^{\dag}(t)\big), \Phi(t)=\Phi\big(t, \breve{x}_1^{\dag}(t),
\breve{x}_0^{\dag}(t), u^{\dag}_{0}(t), \widehat{\breve{x}_1^{\dag}}(t), p_1^{\dag}(t)\big).$ We introduce the following variational equation.
\begin{equation}\label{the leader's decentralized control problem variational equation}
\left\{
             \begin{array}{lr}
             d x^1_0(t)=\left(b_{0 x_0}(t) x^1_0(t)+b_{0 u_0}(t) v_0(t)+b_{0 z}(t) \widehat{x^1_1}(t)\right)dt,\\
             d x^1_1(t)=\left(B_{1 x_1}(t) x^1_1(t)+B_{1 x_0}(t) x^1_0(t)+B_{1 u_0}(t)v_0(t)+B_{1 z}(t) \widehat{x^1_1}(t)+B_{1 p_1}(t) p^1_1(t)\right)dt,\\
             -dp^1_1(t)=\left(\Phi_{x_1}(t) x^1_1(t)+\Phi_{x_0}(t) x^1_0(t)+\Phi_{u_0}(t)v_0(t)+\Phi_{z}(t) \widehat{x^1_1}(t)+\Phi_{p_1}(t) p^1_1(t)\right)dt\\
             \qquad\qquad\quad-l^1_1(t) dW_0(t)-q^1_1(t) dW_1(t),\\
             x^1_0(0)=0,\ x^1_1(0)=0,\ p^1_1(T)=\frac{{\partial}^2 G_1}{\partial x_1^2}\left(\breve{x}_1^{\dag}(T)\right)x^1_1(T).\\
             \end{array}
\right.
\end{equation}

It is observed that the variational equation (\ref{the leader's decentralized control problem variational equation}) can be integrated into a
form encompassed in (\ref{aaa}), just as the discussion about equation (\ref{m}) was done before. Therefore, by Theorem
\ref{the existence and uniqueness of new kind of CMF-FBSDE}, when the first-order
partial derivatives of $b_0(\cdot), B_1(\cdot), \Phi(\cdot)$ with respect to each component and the second-order partial derivative of $G_1$ with
respect to $x_1$ make the integrated equation of the variational equation (\ref{the leader's decentralized control problem variational equation})
satisfy assumptions (H5.1) and (H5.2), then the variational equation (\ref{the leader's decentralized control problem variational equation}) has a
unique solution $\left(x^1_0(\cdot), x^1_1(\cdot), p^1_1(\cdot), l^1_1(\cdot), q^1_1(\cdot)\right)\in L^2_{\overline{\mathbb{F}}^1}
\left(\Omega; C(0, T; \mathbb{R}^k)\right)\times L^2_{\overline{\mathbb{F}}^1}\left(\Omega; C(0, T; \mathbb{R}^n)\right)
\times L^2_{\overline{\mathbb{F}}^1}\left(\Omega; C(0, T; \mathbb{R}^n)\right)\times L^2_{\overline{\mathbb{F}}^1}\left(0, T; \mathbb{R}^{n\times j_0}\right)
\times L^2_{\overline{\mathbb{F}}^1}\left(0, T; \mathbb{R}^{n\times j}\right)$.

And by some classical techniques, we have the following lemma.

\begin{mylem}\label{state estimation}
Let (H2.1), (H2.2), (H5.3) and the assumptions made earlier in this subsection all hold, then we have
\begin{equation}
\begin{aligned}
&\lim\limits_{\rho\longrightarrow 0}\sup\limits_{0\le t\le T} \mathbb{E}\bigg|\frac{\breve{x}^{\rho}_0(t)-\breve{x}_0^{\dag}(t)}{\rho}-x^1_0(t)
\bigg|^2=0,\quad \lim\limits_{\rho\longrightarrow 0}\sup\limits_{0\le t\le T} \mathbb{E}\bigg|\frac{\breve{x}^{\rho}_1(t)-\breve{x}_1^{\dag}(t)}
{\rho}-x^1_1(t)\bigg|^2=0,\\
&\lim\limits_{\rho\longrightarrow 0}\sup\limits_{0\le t\le T} \mathbb{E}\bigg|\frac{p^{\rho}_1(t)-p_1^{\dag}(t)}{\rho}-p^1_1(t)\bigg|^2=0,\quad
\lim\limits_{\rho\longrightarrow 0} \mathbb{E}\int_0^T\bigg|\frac{l^{\rho}_1(t)-l_1^{\dag}(t)}{\rho}-l^1_1(t)\bigg|^2 dt=0,\\
&\lim\limits_{\rho\longrightarrow 0} \mathbb{E}\int_0^T\bigg|\frac{q^{\rho}_1(t)-q_1^{\dag}(t)}{\rho}-q^1_1(t)\bigg|^2 dt=0.
\end{aligned}
\end{equation}
\end{mylem}

Because $u_0^{\dag}$ is the optimal control of the leader's decentralized control problem, then
\begin{equation}\label{o}
\rho^{-1} \Big[\overline{J}_0\left(u^{\rho}_{0}, \overline{\mathcal{I}}(u^{\rho}_{0})\right)-\overline{J}_0\left(u^{\dag}_{0},
\overline{\mathcal{I}}(u^{\dag}_{0})\right)\Big]\ge 0.
\end{equation}

From $(\ref{o})$ and Lemma $\ref{state estimation}$, we have the following lemma.
\begin{mylem}
Let (H2.1), (H2.2), (H5.3) and the assumptions made earlier in this subsection all hold, then the following variational inequality holds:
\begin{equation}\label{p}
\begin{aligned}
&\mathbb{E}\bigg\{\int_0^T g_{0 x_0} (t)^\top x^1_0(t)+g_{0 u_0} (t)^\top v_0(t)+g_{0 z} (t)^\top \mathbb{E}\big[x^1_1(t)|{\mathbb{F}^0_t}\big] dt\\
& \qquad +G_{0 x_0}\big(\breve{x}_0^{\dag}(T)\big)^\top x^1_0(T)\bigg\}\ge 0,
\end{aligned}
\end{equation}
where $g_0 (t)=g_0\big(t, \breve{x}_0^{\dag}(t), u^{\dag}_{0}(t), \widehat{\breve{x}_1^{\dag}}(t)\big),\ g_{0 x_0} (t)
=\frac{\partial g_0}{\partial x_0}\big(t, \breve{x}_0^{\dag}(t), u^{\dag}_{0}(t), \widehat{\breve{x}_1^{\dag}}(t)\big),
\ g_{0 u_0} (t)=\frac{\partial g_0}{\partial u_0}\big(t,$\\
$\breve{x}_0^{\dag}(t), u^{\dag}_{0}(t), \widehat{\breve{x}_1^{\dag}}(t)\big),\ g_{0 z} (t)=\frac{\partial g_0}{\partial z}
\big(t, \breve{x}_0^{\dag}(t), u^{\dag}_{0}(t), \widehat{\breve{x}_1^{\dag}}(t)\big).$
\end{mylem}

\begin{proof}
From Lemma $\ref{state estimation}$ and (H2.2), we know
\begin{equation*}
\begin{aligned}
&\lim\limits_{\rho\longrightarrow 0} \rho^{-1} \Big[\overline{J}_0\left(u^{\rho}_{0}, \overline{\mathcal{I}}(u^{\rho}_{0})\right)
-\overline{J}_0\left(u^{\dag}_{0}, \overline{\mathcal{I}}(u^{\dag}_{0})\right)\Big]\\
&= \lim\limits_{\rho\longrightarrow 0} \rho^{-1} \mathbb{E}\bigg\{\int_0^T \Big[g_0\big(t, \breve{x}^{\rho}_0(t), u^{\rho}_0(t),
\widehat{\breve{x}^{\rho}_1}(t)\big)-g_0\big(t, \breve{x}_0^{\dag}(t), u^{\dag}_0(t), \widehat{\breve{x}_1^{\dag}}(t)\big)\Big]dt\\
&\quad\qquad +\Big[G_0\big(\breve{x}^{\rho}_0(T)\big)-G_0\big(\breve{x}_0^{\dag}(T)\big)\Big]\bigg\}\\
&=\mathbb{E}\bigg\{\int_0^T \Big[g_{0 x_0}(t)^\top x^1_0(t)+g_{0 u_0}(t)^\top v_0(t)+g_{0 z}(t)^\top
\mathbb{E}\big[x^1_1(t)|{\mathbb{F}^0_t}\big]\Big]dt\\
&\quad\qquad +G_{0 x_0}\big(\breve{x}_0^{\dag}(T)\big)^\top x^1_0(T)\bigg\}.\\
\end{aligned}
\end{equation*}
\end{proof}

Next, we introduce the following system of adjoint equations:
\begin{equation}\label{n}
\left\{
             \begin{array}{lr}
             -d K_0(t)={\partial}_{x_0} H_0\left(t,\breve{x}_1^{\dag}(t), \breve{x}_0^{\dag}(t), u_{0}^{\dag}(t),
             \mathbb{E}\big[\breve{x}_1^{\dag}(t)|{\mathbb{F}^0_t}\big], p_1^{\dag}(t), K_0(t), K_1(t), \varphi_1(t) \right) dt\\
             \qquad\qquad\quad-Q_{00}(t) dW_0(t)-Q_{01}(t) dW_1(t),\\
             -d K_1(t)={\partial}_{x_1} H_0\left(t,\breve{x}_1^{\dag}(t), \breve{x}_0^{\dag}(t), u_{0}^{\dag}(t),
             \mathbb{E}\big[\breve{x}_1^{\dag}(t)|{\mathbb{F}^0_t}\big], p_1^{\dag}(t), K_0(t), K_1(t), \varphi_1(t) \right) dt\\
             \qquad\qquad\quad+\mathbb{E}\left[{\partial}_z H_0\left(t,\breve{x}_1^{\dag}(t), \breve{x}_0^{\dag}(t), u_{0}^{\dag}(t),
             \mathbb{E}\big[\breve{x}_1^{\dag}(t)|{\mathbb{F}^0_t}\big], p_1^{\dag}(t), K_0(t), K_1(t), \varphi_1(t) \right)
             \Big|{\mathbb{F}^0_t}\right]dt\\
             \qquad\qquad\quad-Q_{10}(t) dW_0(t)-Q_{11}(t) dW_1(t),\\
             d \varphi_1(t)=-{\partial}_{p_1} H_0\left(t,\breve{x}_1^{\dag}(t), \breve{x}_0^{\dag}(t), u_{0}^{\dag}(t),
             \mathbb{E}\big[\breve{x}_1^{\dag}(t)|{\mathbb{F}^0_t}\big], p_1^{\dag}(t), K_0(t), K_1(t), \varphi_1(t) \right) dt,\\
             K_0(T)=\frac{\partial G_0}{\partial x_0}\left(\breve{x}_0^{\dag}(T)\right),\ K_1(T)=-\frac{{\partial}^2 G_1}{\partial x_1^2}
             \left(\breve{x}_1^{\dag}(T)\right)\varphi_1(T),\ \varphi_1(0)=0,\\
             \end{array}
\right.
\end{equation}
where $\left(\breve{x}_0^{\dag}(\cdot), \breve{x}_1^{\dag}(\cdot), p_1^{\dag}(\cdot)\right)$ satisfies $(\ref{m})$ with respect to $u_{0}^{\dag}$
and is the known unique optimal state of the leader's decentralized control problem. Similarly, the adjoint equation $(\ref{n})$ can be integrated
into the form contained by $(\ref{aaa})$, as was done in the previous discussion about
equation $(\ref{m})$. Therefore, by Theorem $\ref{the existence and uniqueness of new kind of CMF-FBSDE}$, if the integrated equation of $(\ref{n})$
satisfies assumptions (H5.1) and (H5.2), then the adjoint equation $(\ref{n})$ has a unique solution $\big(K_0(\cdot), Q_{00}(\cdot), Q_{01}(\cdot),
K_1(\cdot), Q_{10}(\cdot),$
$Q_{11}(\cdot), \varphi_1(\cdot)\big)\in L^2_{\overline{\mathbb{F}}^1}\left(\Omega; C(0, T; \mathbb{R}^k)\right)\times L^2_{\overline{\mathbb{F}}^1}
\left(0, T; \mathbb{R}^{k\times j_0}\right)\times L^2_{\overline{\mathbb{F}}^1}\left(0, T; \mathbb{R}^{k\times j}\right)\times L^2_{\overline{\mathbb{F}}^1}\left(\Omega;
C(0, T; \mathbb{R}^n)\right)\times L^2_{\overline{\mathbb{F}}^1}\left(0, T; \mathbb{R}^{n\times j_0}\right)\times L^2_{\overline{\mathbb{F}}^1}\big(0, T;$
$\mathbb{R}^{n\times j}\big)\times L^2_{\overline{\mathbb{F}}^1}\left(\Omega; C(0, T; \mathbb{R}^n)\right)$.

Then we have the following maximum principle.

\begin{mythm}\label{the maximum principle for the leader}
Let (H2.1), (H2.2), (H5.3) and the assumptions made earlier in this subsection all hold. Let $u_{0}^{\dag}$ be the optimal control for the leader's
decentralized control problem and
$\left(\breve{x}_0^{\dag}(\cdot), \breve{x}_1^{\dag}(\cdot), p_1^{\dag}(\cdot), l_1^{\dag}(\cdot), q_1^{\dag}(\cdot)\right)$ be the corresponding
optimal state. Let $\big(K_0(\cdot), Q_{00}(\cdot), Q_{01}(\cdot), K_1(\cdot),$ $Q_{10}(\cdot), Q_{11}(\cdot), \varphi_1(\cdot)\big)$ be the solution
of the adjoint equation $(\ref{n})$. Then, we have
\begin{equation}\label{the stable condition for the maximum principle for the leader}
\begin{aligned}
&\mathbb{E}\bigg[\Big\langle {\partial}_{u_0} H_0\left(t,\breve{x}_1^{\dag}(t), \breve{x}_0^{\dag}(t), u_{0}^{\dag}(t),
\mathbb{E}\big[\breve{x}_1^{\dag}(t)|{\mathbb{F}^0_t}\big], p_1^{\dag}(t), K_0(t), K_1(t), \varphi_1(t) \right),\\
&\qquad u_0-u_{0}^{\dag}(t)\Big\rangle \Big|{\mathbb{F}^0_t}\bigg]\ge 0,\quad a.e.\ t\in[0,T],\ a.s.,
\end{aligned}
\end{equation}
holds for any $u_0\in U_0.$
\end{mythm}

\begin{proof}
Applying It\^{o}'s formula to $\langle x^1_0(t), K_0(t)\rangle+\langle x^1_1(t), K_1(t)\rangle+\langle p^1_1(t), \varphi_1(t)\rangle$, integrating
from $t$ to $T$, and taking expectation, we obtain
\begin{equation}\label{q}
\begin{aligned}
&\mathbb{E}\Big[G_{0 x_0}\big(\breve{x}_0^{\dag}(T)\big)^\top x^1_0(T)\Big]\\
&=\mathbb{E}\int_0^T \Big[ \big\langle b_{0 u_0} (t)v_0(t)+b_{0 z} (t)\widehat{x^1_1}(t), K_0(t)\big\rangle -\big\langle x^1_0(t), g_{0 x_0}(t)
\big\rangle\\
&\quad +\big\langle B_{1 u_0} (t)v_0(t)+B_{1 z} (t)\widehat{x^1_1}(t), K_1(t) \big\rangle\\
&\quad -\Big\langle x^1_1(t), \mathbb{E}\Big[b_{0 z} (t)^\top K_0(t)+B_{1 z} (t)^\top K_1(t)-\Phi_{z} (t)^\top \varphi_1(t)+g_{0 z} (t)\Big|
\mathbb{F}_t^0 \Big]\Big\rangle\\
&\quad -\big\langle \Phi_{u_0} (t)v_0(t)+\Phi_{z} (t)\widehat{x^1_1}(t), \varphi_1(t)\big\rangle\Big] dt\\
&=\int_0^T \mathbb{E}\bigg[\mathbb{E}\Big[ \big\langle b_{0 u_0} (t)v_0(t)+b_{0 z} (t)\widehat{x^1_1}(t), K_0(t)\big\rangle -\big\langle x^1_0(t),
g_{0 x_0}(t) \big\rangle\\
&\quad +\big\langle B_{1 u_0} (t)v_0(t)+B_{1 z} (t)\widehat{x^1_1}(t), K_1(t) \big\rangle\\
&\quad -\Big\langle x^1_1(t), \mathbb{E}\Big[b_{0 z} (t)^\top K_0(t)+B_{1 z} (t)^\top K_1(t)-\Phi_{z} (t)^\top \varphi_1(t)+g_{0 z} (t)\Big|
\mathbb{F}_t^0 \Big]\Big\rangle\\
&\quad -\big\langle \Phi_{u_0} (t)v_0(t)+\Phi_{z} (t)\widehat{x^1_1}(t), \varphi_1(t)\big\rangle\Big| \mathbb{F}^0_t\Big]\bigg] dt\\
&=\int_0^T \mathbb{E}\Big[\langle b_{0 u_0} (t)v_0(t), K_0(t)\rangle-\langle x^1_0(t), g_{0 x_0}(t)\rangle+\langle B_{1 u_0} (t)v_0(t), K_1(t)\rangle\\
&\quad -\mathbb{E}\big[x^1_1(t)|{\mathbb{F}^0_t}\big]^\top \mathbb{E}\big[g_{0 z}(t)|{\mathbb{F}^0_t}\big]-\langle \Phi_{u_0} (t)v_0(t),
\varphi_1(t)\rangle \Big] dt.
\end{aligned}
\end{equation}
We substitute $(\ref{q})$ into the variational inequality $(\ref{p})$ to get
\begin{equation}
\begin{aligned}
&\mathbb{E}\int_0^T \Big\langle {\partial}_{u_0} H_0\Big(t,\breve{x}_1^{\dag}(t), \breve{x}_0^{\dag}(t), u_{0}^{\dag}(t),
\mathbb{E}\big[\breve{x}_1^{\dag}(t)|{\mathbb{F}^0_t}\big], p_1^{\dag}(t),\\
&\qquad\quad K_0(t), K_1(t), \varphi_1(t) \Big), v_0(t)\Big\rangle dt \ge 0,
\end{aligned}
\end{equation}
for any $v_0(\cdot)$ such that $u^{\dag}_0(\cdot)+v_0(\cdot)\in \mathcal {U}_0$. If we let $w_0(\cdot)=u^{\dag}_0(\cdot)+v_0(\cdot)$, then the above
inequality implies that for any $w_0\in \mathcal {U}_0$, we have
\begin{equation}\label{s}
\begin{aligned}
&\mathbb{E}\Big\langle {\partial}_{u_0} H_0\Big(t,\breve{x}_1^{\dag}(t), \breve{x}_0^{\dag}(t), u_{0}^{\dag}(t),
\mathbb{E}\big[\breve{x}_1^{\dag}(t)|{\mathbb{F}^0_t}\big], p_1^{\dag}(t),\\
&\qquad K_0(t), K_1(t), \varphi_1(t) \Big), w_0(t)-u^{\dag}_0(t)\Big\rangle \ge 0,\ a.e.\ t\in [0,T].
\end{aligned}
\end{equation}
And for $\forall u_0\in U_0,\ \forall A\in \mathbb{F}^0_t$, we set
\[\bar{w}_0(t)=u_0 I_A+u_{0}^{\dag}(t) I_{\Omega-A}.\]
Because $\bar{w}_0(\cdot)\in \mathcal {U}_0$, inequality $(\ref{s})$ still holds when replacing $w_0(\cdot)$ with $\bar{w}_0(\cdot)$. Then we get
\begin{equation}
\begin{aligned}
&\mathbb{E}\Big[I_A \Big\langle {\partial}_{u_0} H_0\Big(t,\breve{x}_1^{\dag}(t), \breve{x}_0^{\dag}(t), u_{0}^{\dag}(t),
\mathbb{E}\big[\breve{x}_1^{\dag}(t)|{\mathbb{F}^0_t}\big], p_1^{\dag}(t),\\
&\qquad K_0(t), K_1(t), \varphi_1(t) \Big), u_0-u^{\dag}_0(t)\Big\rangle\Big] \ge 0,\ a.e.\ t\in [0,T],
\end{aligned}
\end{equation}
for any $\forall u_0\in U_0,\ \forall A\in \mathbb{F}^0_t$. Therefore, we obtain
\begin{equation}
\begin{aligned}
&\mathbb{E}\Big[\Big\langle {\partial}_{u_0} H_0\Big(t,\breve{x}_1^{\dag}(t), \breve{x}_0^{\dag}(t), u_{0}^{\dag}(t),
\mathbb{E}\big[\breve{x}_1^{\dag}(t)|{\mathbb{F}^0_t}\big], p_1^{\dag}(t),\\
&\qquad K_0(t), K_1(t), \varphi_1(t) \Big), u_0-u^{\dag}_0(t)\Big\rangle\Big| \mathbb{F}^0_t\Big] \ge 0,\ a.e.\ t\in [0,T],\ a.s.,
\end{aligned}
\end{equation}
for any $u_0\in U_0$.
\end{proof}

So, the final consistency condition system is
\begin{equation}\label{the final consistency condition system}
\left\{
             \begin{array}{lr}
             d \breve{x}_0^{\dag}(t)=b_0\left(t, \breve{x}_0^{\dag}(t), u_{0}^{\dag}(t),
             \mathbb{E}\big[\breve{x}_1^{\dag}(t)|{\mathbb{F}^0_t}\big]\right) dt+\sigma_0 dW_0(t),\\
             d \breve{x}_1^{\dag}(t)=B_1\left(t, \breve{x}_1^{\dag}(t), \breve{x}_0^{\dag}(t), u_{0}^{\dag}(t),
             \mathbb{E}\big[\breve{x}_1^{\dag}(t)|{\mathbb{F}^0_t}\big], p_1^{\dag}(t)\right) dt+\widetilde{\sigma} dW_1(t),\\
             -dp_1^{\dag}(t)=\Phi\left(t, \breve{x}_1^{\dag}(t), \breve{x}_0^{\dag}(t), u_0^{\dag}(t), \mathbb{E}\big[\breve{x}_1^{\dag}(t)
             |\mathbb{F}^0_t\big], p_1^{\dag}(t)\right)dt-l_1^{\dag}(t) dW_0(t)-q_1^{\dag}(t) dW_1(t),\\
             \breve{x}_0^{\dag}(0)=\xi_0,\ \breve{x}_1^{\dag}(0)=\xi_1,\ p_1^{\dag}(T)=\frac{\partial G_1}{\partial x_1}
             \left(\breve{x}_1^{\dag}(T)\right).\\
             \\
             -d K_0(t)={\partial}_{x_0} H_0\left(t,\breve{x}_1^{\dag}(t), \breve{x}_0^{\dag}(t), u_{0}^{\dag}(t),
             \mathbb{E}\big[\breve{x}_1^{\dag}(t)|{\mathbb{F}^0_t}\big], p_1^{\dag}(t), K_0(t), K_1(t), \varphi_1(t) \right) dt\\
             \qquad\qquad\quad-Q_{00}(t) dW_0(t)-Q_{01}(t) dW_1(t),\\
             -d K_1(t)={\partial}_{x_1} H_0\left(t,\breve{x}_1^{\dag}(t), \breve{x}_0^{\dag}(t), u_{0}^{\dag}(t),
             \mathbb{E}\big[\breve{x}_1^{\dag}(t)|{\mathbb{F}^0_t}\big], p_1^{\dag}(t), K_0(t), K_1(t), \varphi_1(t) \right) dt\\
             \qquad\qquad\quad+\mathbb{E}\left[{\partial}_z H_0\left(t,\breve{x}_1^{\dag}(t), \breve{x}_0^{\dag}(t),
             u_{0}^{\dag}(t),  \mathbb{E}\big[\breve{x}_1^{\dag}(t)|{\mathbb{F}^0_t}\big], p_1^{\dag}(t), K_0(t), K_1(t), \varphi_1(t) \right)
             \Big|{\mathbb{F}^0_t}\right]dt\\
             \qquad\qquad\quad-Q_{10}(t) dW_0(t)-Q_{11}(t) dW_1(t),\\
             d \varphi_1(t)=-{\partial}_{p_1} H_0\left(t,\breve{x}_1^{\dag}(t), \breve{x}_0^{\dag}(t), u_{0}^{\dag}(t),
             \mathbb{E}\big[\breve{x}_1^{\dag}(t)|{\mathbb{F}^0_t}\big], p_1^{\dag}(t), K_0(t), K_1(t), \varphi_1(t) \right) dt\\
             K_0(T)=\frac{\partial G_0}{\partial x_0}\left(\breve{x}_0^{\dag}(T)\right),\ K_1(T)=-\frac{{\partial}^2 G_1}
             {\partial x_1^2}\left(\breve{x}_1^{\dag}(T)\right)\varphi_1(T),\ \varphi_1(0)=0,\\
             \end{array}
\right.
\end{equation}

\section{The approximate Stackelberg equilibrium}\label{the approximate Stackelberg equilibrium}

\hspace{5mm}First, we give the following definition of approximate Stackelberg equilibrium

\begin{mydef}\label{the definition of approximate Stackelberg equilibrium}
A set of strategies $\big(u^*_0, u^*_1, \cdots, u^*_N\big)$ is an $\varepsilon$-Stackelberg equilibrium if the following hold:\\
\noindent (i) For a given strategy of the leader $u_0\in \mathcal {U}_0$, $u^*=\big(u^*_1, \cdots, u^*_N\big)$ is an $\varepsilon$-optimal
response if it constitutes an $\varepsilon$-Nash equilibrium, that is, there exists a constant $\varepsilon \ge 0$ such that for any $1\le i\le N$
and any $u_i\in \mathcal {U}_i$,
\[J^N_i\big(u^*_i, u^*_{-i},u_0\big)\le J^N_i\big(u_i, u^*_{-i}, u_0\big)+\varepsilon,\]
where $u^*_{-i}=\big(u^*_1, \cdots, u^*_{i-1}, u^*_{i+1}, \cdots, u^*_N\big);$\\
\noindent (ii) For any $u_0\in \mathcal {U}_0$, $J^N_0\big(u^*_0, u^*\big)\le J^N_0\big(u_0, u\big)+\varepsilon,$ where $u^*, u$ are
$\varepsilon$-optimal responses to strategies $u_0^*, u_0$, respectively.
\end{mydef}

\begin{Remark}
Moon, Ba\c{s}ar \cite{MB2018} and Wang \cite{W2024} used this definition to study the LQ mean field Stackelberg game. Bensoussan, Chau, and Yam
\cite{BCY2015}, \cite{BCY2016} only discussed the approximate Nash equilibrium among followers. In this subsection, we will study the approximate
Stackelberg equilibrium of nonlinear open-loop mean field Stackelberg game, which has not been considered in previous literature.
\end{Remark}

\subsection{The proof of the followers' approximate equilibrium}

\hspace{5mm}For an arbitrarily given the leader's decentralized control $u_0\in \mathcal {U}_0$, it follows from the definition of
$\mathcal {U}_0$ that there exists a constant $C_{u_0}>0$ such that $\mathbb{E}\int_0^T |u_0(t)|^2 dt<C_{u_0}$, where $C_{u_0}$ is dependent on the
choice of $u_0$. Next, we will use $C$ to denote a positive constant and $C_{u_0}>0$ to denote a positive constant dependent on $u_0$, which may
take different values on different lines.

In order to distinguish the different states more clearly, we use $\big(\bar{x}_0, (\bar{x}_i)_{i\ge 1}\big)$ to represent the
leader's and followers' state process in their decentralized optimal control problem when the leader takes the control $u_0$ and the followers take
the decentralized optimal controls
$\big\{{\overline{\mathcal{I}}}_i(u_{0})\big\}_{i\ge 1}$, which is determined by Steps 1-5 in Subsection $\ref{section 2.4.1}$.
From $(\ref{d})$ and $(\ref{ith follower Hamiltonian system})$, we obtain
\begin{equation}\label{t}
\left\{
             \begin{array}{lr}
             d {\bar{x}}_0(t)=b_0\left(t, {\bar{x}}_0(t), u_{0}(t), \mathbb{E}\big[{\bar{x}}_i(t)|{\mathbb{F}^0_t}\big]\right) dt+\sigma_0 dW_0(t),\\
             d\bar{x}_i(t)=b_1\left(t, \bar{x}_i(t), {{\overline{\mathcal{I}}}}_i(u_{0})(t), \bar{x}_0(t), u_0(t), \mathbb{E}\big[\bar{x}_i(t)|
             \mathbb{F}^0_t\big]\right) dt+\widetilde{\sigma} dW_i(t), \\
             -dp_i(t)=\bigg\{\left(\frac{\partial b_1}{\partial x_i}\right)^{\top}\left(t, \bar{x}_i(t), {{\overline{\mathcal{I}}}}_i(u_{0})(t),
             \bar{x}_0(t), u_0(t), \mathbb{E}\big[\bar{x}_i(t)|\mathbb{F}^0_t\big]\right)p_i(t)\\
             \qquad\qquad +\frac{\partial g_1}{\partial x_i}\left(t, \bar{x}_i(t), {{\overline{\mathcal{I}}}}_i(u_{0})(t), \bar{x}_0(t), u_0(t),
             \mathbb{E}\big[\bar{x}_i(t)|\mathbb{F}^0_t\big]\right)\bigg\} dt-l_i(t) dW_0(t)-q_i(t) dW_i(t),\\
             \bar{x}_0(0)=\xi_0,\ \bar{x}_i(0)=\xi_i,\ p_i(T)=\frac{\partial G_1}{\partial x_i}\left(\bar{x}_i(T)\right),\ i\ge 1,\\
             \end{array}
\right.
\end{equation}
and
\begin{equation}\label{u}
\begin{aligned}
&\Big\langle {\partial}_{u_i} H_1\left(t,\bar{x}_i(t),{{\overline{\mathcal{I}}}}_i(u_{0})(t), {\bar{x}}_0(t), u_{0}(t),
\mathbb{E}\big[{\bar{x}}_i(t)|{\mathbb{F}^0_t}\big], p_i(t)\right), u_i-{{\overline{\mathcal{I}}}}_i(u_{0})(t)\Big\rangle\ge 0,\\
&\qquad\qquad\qquad\qquad\qquad\qquad\qquad\qquad\qquad\ \forall u_i\in U_i,\ a.e.\ t\in[0,T],\ a.s.,\ i\ge 1.
\end{aligned}
\end{equation}
According to $(\ref{t})$ and $(\ref{u})$, we assume that the follower i's decentralized optimal control ${\overline{\mathcal{I}}}_i(u_{0})$ is
uniquely determined by the following form:
\begin{equation}\label{y}
{\overline{\mathcal{I}}}_i(u_{0}):={\alpha}_1\left(t, \bar{x}_i(t), \bar{x}_0(t), u_0(t), \mathbb{E}\big[\bar{x}_i(t)|\mathbb{F}^0_t\big], p_i(t)\right),
\ i\ge 1.
\end{equation}
By substituting $(\ref{y})$ into $(\ref{t})$, we obtains the following conditional mean-field FBSDE:
\begin{equation}\label{z}
\left\{
             \begin{array}{lr}
             d {\bar{x}}_0(t)=b_0\left(t, {\bar{x}}_0(t), u_{0}(t), \mathbb{E}\big[{\bar{x}}_i(t)|{\mathbb{F}^0_t}\big]\right) dt+\sigma_0 dW_0(t),\\
             d\bar{x}_i(t)=b_1\big(t, \bar{x}_i(t), {\alpha}_1\left(t, \bar{x}_i(t), \bar{x}_0(t), u_0(t),
             \mathbb{E}\big[\bar{x}_i(t)|\mathbb{F}^0_t\big], p_i(t)\right),\\
             \qquad\qquad\quad \bar{x}_0(t), u_0(t), \mathbb{E}\big[\bar{x}_i(t)|\mathbb{F}^0_t\big]\big) dt+\widetilde{\sigma} dW_i(t), \\
             -dp_i(t)=\bigg\{\left(\frac{\partial b_1}{\partial x_i}\right)^{\top}\big(t, \bar{x}_i(t),
             {\alpha}_1\left(t, \bar{x}_i(t), \bar{x}_0(t), u_0(t), \mathbb{E}\big[\bar{x}_i(t)|\mathbb{F}^0_t\big], p_i(t)\right),\\
             \qquad\qquad\quad \bar{x}_0(t), u_0(t), \mathbb{E}\big[\bar{x}_i(t)|\mathbb{F}^0_t\big]\big)p_i(t)+\frac{\partial g_1}
             {\partial x_i}\big(t, \bar{x}_i(t), {\alpha}_1\big(t, \bar{x}_i(t), \bar{x}_0(t), u_0(t),\\
             \qquad\qquad\quad \mathbb{E}\big[\bar{x}_i(t)|\mathbb{F}^0_t\big], p_i(t)\big), \bar{x}_0(t), u_0(t),
             \mathbb{E}\big[\bar{x}_i(t)|\mathbb{F}^0_t\big]\big)\bigg\} dt-l_i(t) dW_0(t)-q_i(t) dW_i(t),\\
             \bar{x}_0(0)=\xi_0,\ \bar{x}_i(0)=\xi_i,\ p_i(T)=\frac{\partial G_1}{\partial x_i}\left(\bar{x}_i(T)\right),\ i\ge 1,\\
             \end{array}
\right.
\end{equation}
We introduce the following additional assumptions.

\noindent {\bf (H6.1)}\ {\it {There exists a constant $C>0$ such that for any $t\in [0,T], x_0, x'_0\in \mathbb{R}^k, x_i, x'_i\in
\mathbb{R}^n, u_0, u'_0\in \mathbb{R}^{m_0}, u_i, u'_i\in \mathbb{R}^m, z, z'\in \mathbb{R}^n,$ we have
\begin{equation*}
\begin{aligned}
&\big|g_0(t, x_0, u_0, z)-g_0(t, x'_0, u'_0, z')\big|\\
& \le C\big(1+|x_0|+|x'_0|+|u_0|+|u'_0|+|z|+|z'|\big)\cdot \big(|x_0-x'_0|+|u_0-u'_0|+|z-z'|\big),\\
&\big|g_1(t, x_i, u_i, x_0, u_0, z)-g_1(t, x'_i, u'_i, x'_0, u'_0, z')\big|\\
& \le C\big(1+|x_i|+|x'_i|+|u_i|+|u'_i|+|x_0|+|x'_0|+|u_0|+|u'_0|+|z|+|z'|\big)\\
&\quad \cdot \big(|x_i-x'_i|+|u_i-u'_i|+|x_0-x'_0|+|u_0-u'_0|+|z-z'|\big),\\
&\big|G_0(x_0)-G_0(x'_0)\big|\le C\big(1+|x_0|+|x'_0|\big)\cdot |x_0-x'_0|.\\
&\big|G_1(x_i)-G_1(x'_i)\big|\le C\big(1+|x_i|+|x'_i|\big)\cdot |x_i-x'_i|.\\
\end{aligned}
\end{equation*}
}}

\noindent {\bf (H6.2)}\ {\it {According to assumption (H5.3) about $\alpha_1$, we know that $\alpha_1$ in $(\ref{y})$ is
continuously differential with respect $(x_i, x_0, u_0, z, p_i)$ and all the first partial derivatives of $\alpha_1$ are bounded.
We further assume that $\alpha_1(\cdot, 0,0,0,0,0)$ is a deterministic function and $\int_0^T \big|\alpha_1(\cdot, 0,0,0,0,0)\big|^2 dt<\infty.$
}}

\noindent {\bf (H6.3)}\ {\it {There exists a constant $C>0$ such that for $\forall i\ge 1$ and $\forall u_i\in \mathcal {U}_i$,
we have $$\mathbb{E}\int_0^T \big|u_i(t)\big|^2 dt\le C$$.
}}

Based on the analysis in Subsection \ref{some well-posedness results of a new kind of CMF-FBSDEs} and \ref{zz}, it can be known that for
an arbitrarily given strategy of the leader $u_0\in \mathcal {U}_0$ and any $i\ge 1$, we can integrate the conditional
mean-field FBSDE $(\ref{z})$ in the same way as equation $(\ref{m})$. So, when the integrated equation of the conditional mean-field FBSDE $(\ref{z})$
satisfies assumptions (H5.1) and (H5.2), equation $(\ref{z})$ has a unique solution $\big({\bar{x}}_0(\cdot), \bar{x}_i(\cdot), p_i(\cdot), l_i(\cdot),
q_i(\cdot)\big)\in L^2_{\overline{\mathbb{F}}^i}
\left(\Omega; C(0, T; \mathbb{R}^k)\right)\times L^2_{\overline{\mathbb{F}}^i}\left(\Omega; C(0, T; \mathbb{R}^n)\right)
\times L^2_{\overline{\mathbb{F}}^i}\left(\Omega; C(0, T; \mathbb{R}^n)\right)$
$\times L^2_{\overline{\mathbb{F}}^i}\left(0, T; \mathbb{R}^{n\times j_0}\right)\times L^2_{\overline{\mathbb{F}}^i}\left(0, T; \mathbb{R}^{n\times j}\right)$.

And in the subsequent discussions, we all assume that the integrated equation of the conditional mean-field FBSDE $(\ref{z})$ satisfies
assumptions (H5.1)-(H5.3). For an arbitrarily given strategy of the leader $u_0\in \mathcal {U}_0$, it can be proved by the symmetry of the followers
and applying
some classical techniques that there exists a constant $C_{u_0}>0$ such that for any $i\ge 1$, the corresponding solution of the mean-field FBSDE
$(\ref{z})$ $\big({\bar{x}}_0(\cdot),
\bar{x}_i(\cdot), p_i(\cdot)\big)$ satisfies
\begin{equation}\label{rr}
\begin{aligned}
\mathbb{E}\Big[\sup\limits_{0\le t\le T} \big|\bar{x}_0(t)\big|^2\Big]+\mathbb{E}\Big[\sup\limits_{0\le t\le T} \big|\bar{x}_i(t)\big|^2\Big]
+\mathbb{E}\Big[\sup\limits_{0\le t\le T} \big|p_i(t)\big|^2 \Big] \le C_{u_0},
\end{aligned}
\end{equation}
where $C_{u_0}$ is depend on the selection of $u_0$. For simplicity, the specific derivation details are omitted here.

Then, it follows from $(\ref{y})$ that there exists a constant $C>0$ such that for $\forall i\ge 1, \forall u_0\in \mathcal {U}_0$, we have
\begin{equation}
\begin{aligned}
\mathbb{E}\int_0^T \Big|{\overline{\mathcal{I}}}_i(u_{0})(t)\Big|^2 dt& =\mathbb{E}\int_0^T\big|{\alpha}_1\left(t, \bar{x}_i(t),
\bar{x}_0(t), u_0(t), \mathbb{E}\big[\bar{x}_i(t)|\mathbb{F}^0_t\big], p_i(t)\right)\big|^2 dt\\
& \le 2\mathbb{E}\int_0^T\big|{\alpha}_1\left(t, \bar{x}_i(t), \bar{x}_0(t), u_0(t), \mathbb{E}\big[\bar{x}_i(t)|\mathbb{F}^0_t\big], p_i(t)\right)\\
&\qquad -{\alpha}_1\left(t, 0,0,0,0,0\right)\big|^2 dt+2\mathbb{E}\int_0^T \big|{\alpha}_1\left(t, 0,0,0,0,0\right)\big|^2 dt\\
&\le 2C\cdot\mathbb{E}\int_0^T\big(\big|\bar{x}_i(t)\big|^2+\big|\bar{x}_0(t)\big|^2+\big|u_0(t)\big|^2+\big|p_i(t)\big|^2 \big) dt\\
&\quad+2\mathbb{E}\int_0^T\big| {\alpha}_1\left(t, 0,0,0,0,0\right)\big|^2 dt
\end{aligned}
\end{equation}
So, for an arbitrarily given $u_0\in \mathcal {U}_0$, there exists a constant $C_{u_0}>0$ such that for $\forall i\ge 1$, we have
\begin{equation}\label{ss}
\mathbb{E}\int_0^T \Big|{\overline{\mathcal{I}}}_i(u_{0})(t)\Big|^2 dt \le C_{u_0},
\end{equation}
where the constant $C_{u_0}$ is dependent on the choice of $u_0$.

For an arbitrarily fixed $N$ and an arbitrarily given leader's decentralized control $u_0\in \mathcal {U}_0$, when $N$ followers all apply the controls
identified by $(\ref{y})$ in the $1+N$-type leader-follower
game, the resulting controlled state process will be denoted by $\big(\bar{x}_0^N,
(\bar{x}_i^N)_{1\le i\le N}\big)$ and solve
\begin{equation}\label{aa}
\left\{
             \begin{array}{lr}
             d \bar{x}_0^N(t)=b_0\left(t,\bar{x}_0^N(t), u_0(t), \bar{x}^{(N)}(t)\right) dt+\sigma_0 dW_0(t),\\
             d \bar{x}^N_i(t)=b_1\big(t, \bar{x}^N_i(t), {\overline{\mathcal{I}}}_i(u_{0})(t), \bar{x}_0^N(t), u_0(t), \bar{x}^{(N)}(t)\big) dt
             +\widetilde{\sigma} dW_i(t),\quad 0\le t \le T, \\
             \bar{x}_0^N(0)=\xi_0,\ \bar{x}^N_i(0)=\xi_i,\ 1\le i\le N,
             \end{array}
\right.
\end{equation}
where $\bar{x}^{(N)}(\cdot)=\frac{1}{N} \sum\limits_{j=1}^N \bar{x}^N_j(\cdot)$. And the cost functional for the $i^{th}$ follower is
\begin{equation}
\begin{aligned}
&J^N_i\big(\big({\overline{\mathcal{I}}}_j(u_{0})\big)_{1\le j\le N},u_0\big)\\
&=\mathbb{E}\bigg\{\int_0^T g_1\Big(t, \bar{x}^N_i(t), {\overline{\mathcal{I}}}_i(u_{0})(t), \bar{x}^N_0(t), u_0(t), \bar{x}^{(N)}(t)\Big) dt
+G_1\left(\bar{x}^N_i(T)\right)\bigg\},
\end{aligned}
\end{equation}
for $i=1,2,\cdots,N$.

For an arbitrarily given strategy of the leader $u_0\in \mathcal {U}_0$, when all followers all apply the controls identified by $(\ref{y})$ in
their decentralized optimal control problems, the resulting controlled state processes will be denoted by $\big(\bar{x}_0, (\bar{x}_i)_{i\ge 1}\big)$
and solve
\begin{equation}\label{bb}
\left\{
             \begin{array}{lr}
             d {\bar{x}}_0(t)=b_0\left(t, {\bar{x}}_0(t), u_{0}(t), \mathbb{E}\big[{\bar{x}}_i(t)|{\mathbb{F}^0_t}\big]\right) dt+\sigma_0 dW_0(t),\\
             d\bar{x}_i(t)=b_1\big(t, \bar{x}_i(t), {\overline{\mathcal{I}}}_i(u_{0})(t), \bar{x}_0(t), u_0(t),
             \mathbb{E}\big[\bar{x}_i(t)|\mathbb{F}^0_t\big]\big) dt+\widetilde{\sigma} dW_i(t),\ 0\le t\le T, \\
             \bar{x}_0(0)=\xi_0,\ \bar{x}_i(0)=\xi_i,\ i\ge 1.
             \end{array}
\right.
\end{equation}
And the cost functional for the $i^{th}$ follower in his decentralized optimal control problem is
\begin{equation}
\begin{aligned}
&\overline{J}_{i}({\overline{\mathcal{I}}}_i(u_{0}), u_{0})\\
&=\mathbb{E}\bigg\{\int_0^T g_1\Big(t, \bar{x}_i(t), {\overline{\mathcal{I}}}_i(u_{0})(t), \bar{x}_0(t), u_{0}(t),
\mathbb{E}\big[\bar{x}_i(t)|\mathbb{F}^0_t\big]\Big) dt+G_1\left(\bar{x}_i(T)\right)\bigg\}.
\end{aligned}
\end{equation}

\begin{mylem}\label{lemma 6.1}
Let (H2.1)-(H2.2), (H6.1)-(H6.3), and the assumptions made earlier in this subsection all hold. Then, for an arbitrarily given $u_0\in \mathcal {U}_0$,
there exists a constant
$C_{u_0}>0$ such that $\big(\bar{x}_0^N, (\bar{x}_i^N)_{1\le i\le N}\big)$ which is the corresponding solution of $(\ref{aa})$ satisfies
\begin{equation}\label{qq}
\begin{aligned}
\bigg[\sup\limits_{N\ge 1} \mathbb{E}\Big[\sup\limits_{0\le t\le T} \big|\bar{x}^N_0(t)\big|^2\Big]\bigg]
+\bigg[\sup\limits_{N\ge 1} \max\limits_{1\le i\le N} \mathbb{E}\Big[\sup\limits_{0\le t\le T} \big|\bar{x}^N_i(t)\big|^2\Big]\bigg]\le
C_{u_0}.
\end{aligned}
\end{equation}
\end{mylem}

\begin{mylem}\label{lemma 6.2}
Let (H2.1)-(H2.2), (H6.1)-(H6.3), and the assumptions made earlier in this subsection all hold. Then,
for an arbitrarily given $u_0\in \mathcal {U}_0$, there exists a
constant $C_{u_0}>0$ such that the
corresponding $\bar{x}^{(N)}$ which is derived from the solution of $(\ref{aa})$ satisfies
\begin{equation}
\begin{aligned}
\sup\limits_{N\ge 1} \mathbb{E}\Big[\sup\limits_{0\le t\le T} \big|\bar{x}^{(N)}(t)\big|^2\Big]\le C_{u_0}.
\end{aligned}
\end{equation}
\end{mylem}

For any given $u_0\in \mathcal {U}_0$, the corresponding $\big\{\bar{x}_j\big\}_{1\le j\le N}$ is always a exchangeable sequence.
Based on the tailor-made propagation chaos analysis on the exchangeable sequence of random variables in \cite{CD2019-2}, \cite{CZ2016}, \cite{A1985},
\cite{RR1998}, it is
justifiable
to assume that when $N$ tends to infinity, for any $u_0\in \mathcal {U}_0$, we have the corresponding $\frac{1}{N}\sum\limits_{j=1}^N
\bar{x}_j$ converges to $\widehat{\bar{x}}_i(\cdot)$, that is,
\begin{equation}\label{jj}
\begin{aligned}
\lim\limits_{N\longrightarrow\infty} \bigg[\sup\limits_{u_0\in \mathcal {U}_0}\mathbb{E}\int_0^T \bigg|\Big[\frac{1}{N}\sum\limits_{j=1}^N
\bar{x}_j(t)\Big]
-\mathbb{E}\big[\bar{x}_i(t)|\mathbb{F}^0_t\big]\bigg|^2 dt\bigg]=0.
\end{aligned}
\end{equation}
for any $i\ge 1$.

\begin{mylem}\label{lemma 6.4}
Let (H2.1)-(H2.2), (H6.1)-(H6.3), and the assumptions made earlier in this subsection all hold.
Then, there exists a constant $C>0$ such that for any fixed $N\ge 1, \forall 1\le i\le N$ and
$\forall u_0\in \mathcal {U}_0$, we have the
corresponding  solution of $(\ref{aa})$ and $(\ref{bb})$ satisfies
\begin{equation}
\begin{aligned}
&\mathbb{E}\Big[\sup\limits_{0\le t\le T} \big|\bar{x}^N_i(t)-\bar{x}_i(t)\big|^2\Big]+\mathbb{E}\Big[\sup\limits_{0\le t\le T} \big|\bar{x}^N_0(t)
-\bar{x}_0(t)\big|^2\Big]\\
&\le C\cdot \mathbb{E}\int_0^T \bigg|\Big[\frac{1}{N}\sum\limits_{j=1}^N \bar{x}_j(t)\Big]
-\mathbb{E}\big[\bar{x}_i(t)|\mathbb{F}^0_t\big]\bigg|^2 dt.
\end{aligned}
\end{equation}
\end{mylem}

\begin{mylem}\label{lemma 6.5}
Let (H2.1)-(H2.2), (H6.1)-(H6.3), and the assumptions made earlier in this subsection all hold. Then,
there exists a constant $C>0$ such that for any fixed $N\ge 1, \forall 1\le i\le N$ and
$\forall u_0\in \mathcal {U}_0$, we have the
corresponding $\bar{x}^{(N)}$ which is derived from the solution of $(\ref{aa})$ satisfies
\begin{equation}
\begin{aligned}
&\mathbb{E}\int_0^T \Big|\bar{x}^{(N)}(t)
-\mathbb{E}\big[\bar{x}_i(t)|\mathbb{F}^0_t\big]\Big|^2 dt\le C\cdot \mathbb{E}\int_0^T \bigg|\Big[\frac{1}{N}\sum\limits_{j=1}^N \bar{x}_j(t)\Big]
-\mathbb{E}\big[\bar{x}_i(t)|\mathbb{F}^0_t\big]\bigg|^2 dt.
\end{aligned}
\end{equation}
\end{mylem}

By applying Gronwall's inequality and a few basic inequalities, we can prove Lemmas $\ref{lemma 6.1}$-$\ref{lemma 6.5}$. The proof processes are
relatively simple and are omitted here.

\begin{mylem}\label{ee}
Let (H2.1)-(H2.2), (H6.1)-(H6.3), and the assumptions made earlier in this subsection all hold. Then,
for an arbitrarily given $u_0\in \mathcal {U}_0$, there exists a
constant $C_{u_0}>0$ such that for any fixed $N\ge 1$ and $\forall 1\le i\le N$, we have
\begin{equation}
\begin{aligned}
&\Big|J^N_i\big(\big({\overline{\mathcal{I}}}_j(u_{0})\big)_{1\le j\le N},u_0\big)-\overline{J}_{i}\big({\overline{\mathcal{I}}}_i(u_{0}),
u_{0}\big)\Big|\\
&\le C_{u_0}\cdot \bigg[\mathbb{E}\int_0^T \bigg|\Big[\frac{1}{N}\sum\limits_{j=1}^N \bar{x}_j(t)\Big]
-\mathbb{E}\big[\bar{x}_i(t)|\mathbb{F}^0_t\big]\bigg|^2 dt\bigg]^{\frac{1}{2}}.
\end{aligned}
\end{equation}
\end{mylem}

\begin{proof}
By assumption (H6.1), it follows that for $\forall u_0\in \mathcal {U}_0, \forall N\ge 1$ and $\forall 1\le i\le N$ we have
\begin{equation}
\begin{aligned}
&\Big|J^N_i\big(\big({\overline{\mathcal{I}}}_j(u_{0})\big)_{1\le j\le N}, u_0\big)-\overline{J}_i\big({\overline{\mathcal{I}}}_i(u_{0}), u_0\big)\Big|\\
&=\bigg|\mathbb{E}\bigg\{\int_0^T g_1\big(t, \bar{x}^N_i(t), {\overline{\mathcal{I}}}_i(u_{0})(t), \bar{x}^N_0(t), u_{0}(t), \bar{x}^{(N)}(t)\big)
-g_1\big(t, \bar{x}_i(t), {\overline{\mathcal{I}}}_i(u_{0})(t), \bar{x}_0(t),\\
&\qquad\quad u_{0}(t),\mathbb{E}\big[\bar{x}_i(t)|\mathbb{F}^0_t\big]\big) dt+G_1\left(\bar{x}^N_i(T)\right)-G_1\left(\bar{x}_i(T)\right)\bigg\}\bigg|\\
&\le C\cdot \mathbb{E}\int_0^T\Big[1+\big|\bar{x}^N_i(t)\big|+\big|\bar{x}_i(t)\big|+\big|{\overline{\mathcal{I}}}_i(u_{0})(t)\big|
+\big|\bar{x}^N_0(t)\big|+\big|\bar{x}_0(t)\big|+\big|u_0(t)\big|+\big|\bar{x}^{(N)}(t)\big|\\
&\qquad\quad +\big|\mathbb{E}\big[\bar{x}_i(t)|\mathbb{F}^0_t\big]\big|\Big]\cdot \Big[\big|\bar{x}^N_i(t)-\bar{x}_i(t)\big|+\big|\bar{x}^N_0(t)
-\bar{x}_0(t)\big|+\big|\bar{x}^{(N)}(t)-\mathbb{E}\big[\bar{x}_i(t)|\mathbb{F}^0_t\big]
\big|\Big] dt\\
&\quad +C\cdot \mathbb{E}\Big[\big(1+\big|\bar{x}^N_i(T)\big|+\big|\bar{x}_i(T)\big|\big)\cdot \big|\bar{x}^N_i(T)-\bar{x}_i(T)\big|\Big]\\
&\le C\cdot \bigg[\mathbb{E}\int_0^T\Big[1+\big|\bar{x}^N_i(t)\big|^2+\big|\bar{x}_i(t)\big|^2+\big|{\overline{\mathcal{I}}}_i(u_{0})(t)\big|^2
+\big|\bar{x}^N_0(t)\big|^2+\big|\bar{x}_0(t)\big|^2+\big|u_0(t)\big|^2\\
&\qquad\quad +\big|\bar{x}^{(N)}(t)\big|^2+\mathbb{E}\big[\big|\bar{x}_i(t)\big|^2|\mathbb{F}^0_t\big]\Big] dt\bigg]^{\frac{1}{2}}
\cdot\bigg[\mathbb{E}\int_0^T\Big[\big|\bar{x}^N_i(t)-\bar{x}_i(t)\big|^2+\big|\bar{x}^N_0(t)-\bar{x}_0(t)\big|^2\\
&\qquad\quad +\big|\bar{x}^{(N)}(t)
-\mathbb{E}\big[\bar{x}_i(t)|\mathbb{F}^0_t\big]\big|^2\Big]
dt\bigg]^{\frac{1}{2}}+C\cdot \Big[\mathbb{E}\Big[1+\big|\bar{x}^N_i(T)\big|^2+\big|\bar{x}_i(T)\big|^2\Big]\Big]^{\frac{1}{2}}\\
&\qquad \times \Big[\mathbb{E}\big|\bar{x}^N_i(T)-\bar{x}_i(T)\big|^2\Big]^{\frac{1}{2}}.
\end{aligned}
\end{equation}

From $(\ref{rr})$, $(\ref{ss})$, Lemmas $\ref{lemma 6.1}$-$\ref{lemma 6.5}$, it follows that for an
arbitrarily given $u_0\in \mathcal {U}_0$, there exists a constant $C_{u_0}>0$ such that for any fixed $N\ge 1$ and $\forall 1\le i\le N$, we have
\begin{equation}
\begin{aligned}
&\Big|J^N_i\big(\big({\overline{\mathcal{I}}}_j(u_{0})\big)_{1\le j\le N},u_0\big)-\overline{J}_{i}
\big({\overline{\mathcal{I}}}_i(u_{0}), u_{0}\big)\Big|\\
&\le C_{u_0}\cdot \bigg[\mathbb{E}\int_0^T \bigg|\Big[\frac{1}{N}\sum\limits_{j=1}^N \bar{x}_j(t)\Big]
-\mathbb{E}\big[\bar{x}_i(t)|\mathbb{F}^0_t\big]\bigg|^2 dt\bigg]^{\frac{1}{2}}.
\end{aligned}
\end{equation}

\end{proof}

For an arbitrarily given leader's decentralized control $u_0\in \mathcal {U}_0$, in order to show that
$\big({\overline{\mathcal{I}}}_i(u_{0})\big)_{1\le i\le N}$ indeed constitute an $\varepsilon$-Nash equilibrium for the N followers' Nash game,
without loss of generality, we assume that the first follower takes arbitrary control $u_1\in \mathcal {U}_1$ instead of
${\overline{\mathcal{I}}}_1(u_{0})$. And the other $j^{th}$ follower still apply control ${\overline{\mathcal{I}}}_j(u_{0})$ for $j=2,\cdots, N$.
The resulting perturbed controlled state processes in the $1+N$-type leader-follower
game will be denoted by $\big(\widetilde{x}^N_0, (\widetilde{x}^N_i)_{1\le i\le N}\big)$ and solve
\begin{equation}\label{cc}
\left\{
             \begin{array}{lr}
             d \widetilde{x}^N_0(t)=b_0\left(t,\widetilde{x}_0^N(t), u_0(t), \widetilde{x}^{(N)}(t)\right) dt+\sigma_0 dW_0(t),\\
             d \widetilde{x}^N_1(t)=b_1\left(t,\widetilde{x}_1^N(t), u_1(t),\widetilde{x}_0^N(t), u_0(t), \widetilde{x}^{(N)}(t)\right) dt
             +\widetilde{\sigma} dW_1(t),\\
             d \widetilde{x}^N_i(t)=b_1\big(t, \widetilde{x}^N_i(t), {\overline{\mathcal{I}}}_i(u_{0})(t), \widetilde{x}_0^N(t), u_0(t),
             \widetilde{x}^{(N)}(t)\big) dt+\widetilde{\sigma} dW_i(t),\quad 0\le t \le T, \\
             \widetilde{x}_0^N(0)=\xi_0,\ \widetilde{x}^N_1(0)=\xi_1,\ \widetilde{x}^N_i(0)=\xi_i,\ 2\le i\le N,
             \end{array}
\right.
\end{equation}
where $\widetilde{x}^{(N)}(\cdot)=\frac{1}{N} \sum\limits_{j=1}^N \widetilde{x}^N_j(\cdot)$. In this case, the cost functional for the first
follower is
\begin{equation}
\begin{aligned}
&J^N_1\big(u_1, \big({\overline{\mathcal{I}}}_j(u_{0})\big)_{2\le j\le N},u_0\big)\\
&=\mathbb{E}\bigg\{\int_0^T g_1\Big(t, \widetilde{x}^N_1(t), u_1(t), \widetilde{x}^N_0(t), u_0(t), \widetilde{x}^{(N)}(t)\Big) dt
+G_1\left(\widetilde{x}^N_1(T)\right)\bigg\},
\end{aligned}
\end{equation}

For a given strategy of the leader $u_0\in \mathcal {U}_0$, when the first follower takes arbitrary control $u_1\in \mathcal {U}_1$ and the
other $j^{th}$ follower apply the controls identified by $(\ref{y})$ in their decentralized optimal control problems, the resulting controlled state
processes will be denoted by $\big(\bar{x}_0, \widetilde{x}_1, (\bar{x}_i)_{i\ge 2}\big)$ and solve
\begin{equation}\label{dd}
\left\{
             \begin{array}{lr}
             d \bar{x}_0(t)=b_0\left(t, \bar{x}_0(t), u_{0}(t), \mathbb{E}\big[\bar{x}_i(t)|{\mathbb{F}^0_t}\big]\right) dt+\sigma_0 dW_0(t),\\
             d \widetilde{x}_1(t)=b_1\left(t,\widetilde{x}_1(t), u_1(t), \bar{x}_0(t), u_{0}(t), \mathbb{E}\big[\bar{x}_i(t)
             |{\mathbb{F}^0_t}\big]\right) dt+\widetilde{\sigma} dW_1(t),\\
             d\bar{x}_i(t)=b_1\big(t, \bar{x}_i(t), {\overline{\mathcal{I}}}_i(u_{0})(t), \bar{x}_0(t), u_0(t), \mathbb{E}\big[\bar{x}_i(t)
             |\mathbb{F}^0_t\big]\big) dt+\widetilde{\sigma} dW_i(t),\ 0\le t\le T, \\
             \bar{x}_0(0)=\xi_0,\ \widetilde{x}_1(0)=\xi_1,\ \bar{x}_i(0)=\xi_i,\ i\ge 2.
             \end{array}
\right.
\end{equation}
And the cost functional for the first follower in his decentralized optimal control problem is
\begin{equation}
\begin{aligned}
&\overline{J}_1(u_1, u_{0})\\
&=\mathbb{E}\bigg\{\int_0^T g_1\Big(t, \widetilde{x}_1(t), u_1(t), \bar{x}_0(t), u_{0}(t), \mathbb{E}\big[\bar{x}_i(t)|\mathbb{F}^0_t\big]\Big) dt
+G_1\left(\widetilde{x}_1(T)\right)\bigg\},
\end{aligned}
\end{equation}

\begin{mylem}\label{lemma 6.7}
Let (H2.1)-(H2.2), (H6.1)-(H6.3), and the assumptions made earlier in this subsection all hold. Then,
for an arbitrarily given $u_0\in \mathcal {U}_0$, there exists a constant
$C_{u_0}>0$ such that for
$\forall u_1\in \mathcal {U}_1$, we have $\big(\widetilde{x}^N_0, (\widetilde{x}^N_i)_{1\le i\le N}\big)$ which is the corresponding solution of
$(\ref{cc})$ satisfies
\begin{equation}
\begin{aligned}
\bigg[\sup\limits_{N\ge 1} \mathbb{E}\Big[\sup\limits_{0\le t\le T} \big|\widetilde{x}^N_0(t)\big|^2\Big]\bigg]
+\bigg[\sup\limits_{N\ge 1} \max\limits_{1\le i\le N} \mathbb{E}\Big[\sup\limits_{0\le t\le T} \big|\widetilde{x}^N_i(t)\big|^2\Big]\bigg]\le
C_{u_0}.
\end{aligned}
\end{equation}
\end{mylem}

\begin{mylem}\label{lemma 6.8}
Let (H2.1)-(H2.2), (H6.1)-(H6.3), and the assumptions made earlier in this subsection all hold. Then,
for an arbitrarily given $u_0\in \mathcal {U}_0$, there exists a constant
$C_{u_0}>0$ such that for
 $\forall u_1\in \mathcal {U}_1$, we have the corresponding $\widetilde{x}^{(N)}$ which is derived from the solution of $(\ref{cc})$ satisfies
\begin{equation}
\begin{aligned}
\sup\limits_{N\ge 1} \mathbb{E}\Big[\sup\limits_{0\le t\le T} \big|\widetilde{x}^{(N)}(t)\big|^2\Big]\le C_{u_0}.
\end{aligned}
\end{equation}
\end{mylem}

\begin{mylem}\label{lemma 6.9}
Let (H2.1)-(H2.2), (H6.1)-(H6.3), and the assumptions made earlier in this subsection all hold. Then,
for an arbitrarily given $u_0\in \mathcal {U}_0$, there exists a constant
$C_{u_0}>0$ such that for $\forall u_1\in \mathcal {U}_1$, we have $\widetilde{x}_1$ which is the corresponding solution of $(\ref{dd})$ satisfies
\begin{equation}
\begin{aligned}
\mathbb{E}\Big[\sup\limits_{0\le t\le T} \big|\widetilde{x}_1(t)\big|^2\Big]\le C_{u_0}.
\end{aligned}
\end{equation}
\end{mylem}

By applying Gronwall's inequality and a few basic inequalities, we can prove Lemmas $\ref{lemma 6.7}$-$\ref{lemma 6.9}$. Similar to some of the
previous lemmas, the proof processes are relatively simple and are omitted here.

\begin{mylem}\label{lemma 6.10}
Let (H2.1)-(H2.2), (H6.1)-(H6.3), and the assumptions made earlier in this subsection all hold. Then,
for an arbitrarily given $u_0\in \mathcal {U}_0$, there exists a constant
$C_{u_0}>0$ such that for any fixed $N\ge 1, \forall u_1\in \mathcal {U}_1$ and $\forall i=2,\cdots, N$, we have the corresponding solutions of
$(\ref{cc})$ and $(\ref{dd})$ satisfies
\begin{equation}
\begin{aligned}
&\mathbb{E}\Big[\sup\limits_{0\le t\le T} \big|\widetilde{x}^N_0(t)-\bar{x}_0(t)\big|^2\Big]+\mathbb{E}\Big[\sup\limits_{0\le t\le T}
\big|\widetilde{x}^N_i(t)-\bar{x}_i(t)\big|^2\Big]\\
&\le C_{u_0}\cdot \bigg\{\frac{1}{N^2}+\mathbb{E}\int_0^T \bigg|\Big[\frac{1}{N-1}\sum\limits_{j=2}^N \bar{x}_j(t)\Big]
-\mathbb{E}\big[\bar{x}_i(t)|\mathbb{F}^0_t\big]\bigg|^2 dt\bigg\}.
\end{aligned}
\end{equation}
\end{mylem}

\begin{mylem}\label{lemma 6.11}
Let (H2.1)-(H2.2), (H6.1)-(H6.3), and the assumptions made earlier in this subsection all hold. Then,
for an arbitrarily given $u_0\in \mathcal {U}_0$, there exists a constant
$C_{u_0}>0$ such that for any fixed $N\ge 1, \forall u_1\in \mathcal {U}_1$ and $\forall i\ge 1$, we have the corresponding $\widetilde{x}^{(N)}$
which is derived from the solutions of $(\ref{cc})$ satisfies
\begin{equation}
\begin{aligned}
&\mathbb{E}\int_0^T \Big|\widetilde{x}^{(N)}(t)
-\mathbb{E}\big[\bar{x}_i(t)|\mathbb{F}^0_t\big]\Big|^2 dt\\
&\le C_{u_0}\cdot\bigg\{\frac{1}{N^2}+\mathbb{E}\int_0^T \bigg|\Big[\frac{1}{N-1}\sum\limits_{j=2}^N \bar{x}_j(t)\Big]
-\mathbb{E}\big[\bar{x}_i(t)|\mathbb{F}^0_t\big]\bigg|^2 dt\bigg\}.
\end{aligned}
\end{equation}
\end{mylem}

Next, we prove Lemma $\ref{lemma 6.10}$ and Lemma $\ref{lemma 6.11}$ together.

\begin{proof}
From assumption (H2.2) and some basic inequalities, there is a constant $C>0$ such that  for $ \forall u_0\in \mathcal {U}_0, \forall N\ge 1,
\forall u_1\in \mathcal {U}_1$ and $\forall i=2,\cdots, N$, we have
\begin{equation}\label{tt}
\begin{aligned}
\mathbb{E}\Big[\sup\limits_{0\le t\le T} \big|\widetilde{x}^N_i(t)-\bar{x}_i(t)\big|^2\Big]& \le C\cdot \mathbb{E}\int_0^T
\Big[\big|\widetilde{x}^N_i(s)-\bar{x}_i(s)\big|^2+\big|\widetilde{x}^N_0(s)-\bar{x}_0(s)\big|^2\\
&\qquad + \big|\widetilde{x}^{(N)}(s)
-\mathbb{E}\big[\bar{x}_i(s)|\mathbb{F}^0_s\big]\big|^2\Big] ds,
\end{aligned}
\end{equation}
and
\begin{equation}\label{uu}
\begin{aligned}
\mathbb{E}\Big[\sup\limits_{0\le t\le T} \big|\widetilde{x}^N_0(t)-\bar{x}_0(t)\big|^2\Big]& \le C\cdot \mathbb{E}\int_0^T
\Big[\big|\widetilde{x}^N_0(s)-\bar{x}_0(s)\big|^2+\big|\widetilde{x}^{(N)}(s)
-\mathbb{E}\big[\bar{x}_i(s)|\mathbb{F}^0_s\big]\big|^2\Big] ds.
\end{aligned}
\end{equation}
There is also a constant $C>0$ such that $ \forall u_0\in \mathcal {U}_0, \forall N\ge 1, \forall u_1\in \mathcal {U}_1$ and $\forall i\ge 1$, we have
\begin{equation}
\begin{aligned}
&\mathbb{E}\int_0^T \big|\widetilde{x}^{(N)}(s)
-\mathbb{E}\big[\bar{x}_i(s)|\mathbb{F}^0_s\big]\big|^2 ds\\
& \le C\cdot \mathbb{E}\int_0^T\bigg\{\bigg|\widetilde{x}^{(N)}(s)
-\Big[\frac{1}{N-1}\sum\limits_{j=2}^N \widetilde{x}^N_j(s)\Big]\bigg|^2+\bigg|\Big[\frac{1}{N-1}\sum\limits_{j=2}^N \widetilde{x}^N_j(s)\Big]
-\Big[\frac{1}{N-1}\sum\limits_{j=2}^N \bar{x}_j(s)\Big]\bigg|^2\\
&\qquad +\bigg|\Big[\frac{1}{N-1}\sum\limits_{j=2}^N \bar{x}_j(s)\Big]-\mathbb{E}\big[\bar{x}_i(s)|\mathbb{F}^0_s\big]\bigg|^2\bigg\}ds\\
& \le C\cdot \mathbb{E}\int_0^T\bigg\{\frac{1}{N^2}\big|\widetilde{x}^N_1(s)\big|^2
+\frac{1}{N^2(N-1)}\Big[\sum\limits_{j=2}^N \big|\widetilde{x}^N_j(s)\big|^2\Big]+\frac{1}{N-1}\Big[\sum\limits_{j=2}^N \big|\widetilde{x}^N_j(s)
-\bar{x}_j(s)\big|^2\Big]\\
&\qquad +\bigg|\Big[\frac{1}{N-1}\sum\limits_{j=2}^N \bar{x}_j(s)\Big]-\mathbb{E}\big[\bar{x}_i(s)|\mathbb{F}^0_s\big]\bigg|^2\bigg\}ds.\\
\end{aligned}
\end{equation}
From Lemma $\ref{lemma 6.7}$ and the symmetry of $\big\{\widetilde{x}^N_j\big\}_{2\le j\le N}$, we know that for an arbitrarily given
$u_0\in \mathcal{U}_0$, there exists a constant $C_{u_0}>0$ such that for $\forall N\ge 1, \forall u_1\in \mathcal {U}_1$ and $\forall i\ge 1$,
there is
\begin{equation}\label{vv}
\begin{aligned}
&\mathbb{E}\int_0^T \big|\widetilde{x}^{(N)}(s)
-\mathbb{E}\big[\bar{x}_i(s)|\mathbb{F}^0_s\big]\big|^2 ds\\
& \le C_{u_0}\cdot \bigg\{\frac{1}{N^2}+\mathbb{E}\int_0^T\bigg\{\big|\widetilde{x}^N_i(s)-\bar{x}_i(s)\big|^2+\bigg|\Big[\frac{1}{N-1}
\sum\limits_{j=2}^N \bar{x}_j(s)\Big]-\mathbb{E}\big[\bar{x}_i(s)|\mathbb{F}^0_s\big]\bigg|^2\bigg\}ds\bigg\}.\\
\end{aligned}
\end{equation}
By combining $(\ref{tt})$, $(\ref{uu})$ and $(\ref{vv})$, we know that for an arbitrarily given $u_0\in \mathcal{U}_0$, there exists a constant
$C_{u_0}>0$ such that for $\forall N\ge 1, \forall u_1\in \mathcal {U}_1$ and $\forall i=2, \cdots, N$, we have
\begin{equation}\label{ww}
\begin{aligned}
&\mathbb{E}\Big[\sup\limits_{0\le t\le T} \big|\widetilde{x}^N_i(t)-\bar{x}_i(t)\big|^2\Big]+\mathbb{E}\Big[\sup\limits_{0\le t\le T}
\big|\widetilde{x}^N_0(t)-\bar{x}_0(t)\big|^2\Big]\\
& \le C_{u_0}\cdot \bigg\{\mathbb{E}\int_0^T\Big[\big|\widetilde{x}^N_i(s)-\bar{x}_i(s)\big|^2+\big|\widetilde{x}^N_0(s)-\bar{x}_0(s)\big|^2 \Big]ds
+\frac{1}{N^2}\\
&\qquad +\mathbb{E}\int_0^T\bigg|\Big[\frac{1}{N-1}\sum\limits_{j=2}^N \bar{x}_j(s)\Big]-\mathbb{E}\big[\bar{x}_i(s)|\mathbb{F}^0_s\big]\bigg|^2ds
\bigg\}\\
& \le C_{u_0}\cdot \int_0^T \Big\{\mathbb{E}\Big[\sup\limits_{0\le r\le s}\big|\widetilde{x}^N_i(r)-\bar{x}_i(r)\big|^2\Big]
+\mathbb{E}\Big[\sup\limits_{0\le r\le s}\big|\widetilde{x}^N_0(r)-\bar{x}_0(r)\big|^2 \Big]\Big\}ds\\
&\quad +C_{u_0}\cdot\bigg\{\frac{1}{N^2}+\mathbb{E}\int_0^T\bigg|\Big[\frac{1}{N-1}\sum\limits_{j=2}^N \bar{x}_j(s)\Big]
-\mathbb{E}\big[\bar{x}_i(s)|\mathbb{F}^0_s\big]\bigg|^2ds \bigg\}.\\
\end{aligned}
\end{equation}
We apply the Gronwall's inequality to $(\ref{ww})$ to get the conclusion of Lemma $\ref{lemma 6.10}$. Based on Lemma $\ref{lemma 6.10}$ and the
inequality $(\ref{vv})$, we can easily derive Lemma $\ref{lemma 6.11}$.

\end{proof}

\begin{mylem}\label{lemma 6.12}
Let (H2.1)-(H2.2), (H6.1)-(H6.3), and the assumptions made earlier in this subsection all hold. Then,
for an arbitrarily given $u_0\in \mathcal {U}_0$, there exists a constant
$C_{u_0}>0$ such that for any fixed $N\ge 1, \forall u_1\in \mathcal {U}_1$ and $\forall i\ge 1$, we have the corresponding solutions of $(\ref{cc})$
and $(\ref{dd})$ satisfies
\begin{equation}
\begin{aligned}
&\mathbb{E}\Big[\sup\limits_{0\le t\le T} \big|\widetilde{x}^N_1(t)-\widetilde{x}_1(t)\big|^2\Big]\\
&\le C_{u_0}\cdot\bigg\{\frac{1}{N^2}+\mathbb{E}\int_0^T \bigg|\Big[\frac{1}{N-1}\sum\limits_{j=2}^N \bar{x}_j(t)\Big]
-\mathbb{E}\big[\bar{x}_i(t)|\mathbb{F}^0_t\big]\bigg|^2 dt\bigg\}.
\end{aligned}
\end{equation}
\end{mylem}

\begin{mylem}\label{ff}
Let (H2.1)-(H2.2), (H6.1)-(H6.3), and the assumptions made earlier in this subsection all hold. Then,
for an arbitrarily given $u_0\in \mathcal {U}_0$, there exists a constant
$C_{u_0}>0$ such that for any fixed $N\ge 1, \forall u_1\in \mathcal {U}_1$ and $\forall i\ge 1$, we have
\begin{equation}
\begin{aligned}
&\Big|J^N_1\big(u_1, \big({\overline{\mathcal{I}}}_j(u_{0})\big)_{2\le j\le N},u_0\big)-\overline{J}_{1}(u_1, u_{0})\Big|\\
&\le C_{u_0}\cdot\bigg\{\frac{1}{N^2}+\mathbb{E}\int_0^T \bigg|\Big[\frac{1}{N-1}\sum\limits_{j=2}^N \bar{x}_j(t)\Big]
-\mathbb{E}\big[\bar{x}_i(t)|\mathbb{F}^0_t\big]\bigg|^2 dt\bigg\}^{\frac{1}{2}}.
\end{aligned}
\end{equation}
\end{mylem}

By applying Lemma $\ref{lemma 6.10}$-$\ref{lemma 6.11}$, Gronwall's inequality and some basic inequalities, we can prove Lemma $\ref{lemma 6.12}$.
By applying the results of Lemma $\ref{lemma 6.7}$-$\ref{lemma 6.12}$ and performing similar steps to Lemma $\ref{ee}$, we can prove Lemma $\ref{ff}$.
The specific proof process is omitted here.

\begin{mythm}\label{followers' approximate theorem}
Let (H2.1)-(H2.2), (H6.1)-(H6.3), and the assumptions made earlier in this subsection all hold. For an arbitrarily given strategy of the leader
$u_0\in \mathcal {U}_0$,
$\big({\overline{\mathcal{I}}}_i(u_{0})\big)_{1\le i\le N}$ constitutes an $\varepsilon$-Nash equilibrium for the N followers' Nash game, that is,
there exists a sequence $\big\{\varepsilon^{u_0}_1(N)\big\}_{N\ge 1}$ such that $\lim\limits_{N\longrightarrow +\infty} \varepsilon^{u_0}_1(N)=0$
and for $\forall N\ge 1, \forall i=1,2,\cdots,N$ and $\forall u_i\in \mathcal {U}_i$, we have
\begin{equation}
\begin{aligned}
&J^N_i\big(\big({\overline{\mathcal{I}}}_j(u_{0})\big)_{1\le j\le N},u_0\big)\le J^N_i\big(u_i, \big({\overline{\mathcal{I}}}_j(u_{0})\big)
_{1\le j\le N, j\ne i},u_0\big)+\varepsilon^{u_0}_1(N),
\end{aligned}
\end{equation}
where the sequence $\big\{\varepsilon^{u_0}_1(N)\big\}_{N\ge 1}$ is dependent on the choice of $u_0$.
\end{mythm}

\begin{proof}
Due to the symmetry of the followers, without loss of generality, we only prove the case when $i=1$. For $i=2,\cdots,N$, it can be proved similarly.
By Lemma $\ref{ee}$, Lemma $\ref{ff}$, and the fact that ${\overline{\mathcal{I}}}_1(u_{0})$ is the follower 1's decentralized optimal control, it
follows that for $\forall u_0\in \mathcal {U}_0$, there exists a constant $C_{u_0}>0$ such that for $\forall N\ge 1, \forall u_1\in \mathcal {U}_1$
and $\forall i\ge 1$, we have
\begin{equation}\label{nn}
\begin{aligned}
&J^N_1\big(\big({\overline{\mathcal{I}}}_j(u_{0})\big)_{1\le j\le N},u_0\big)\\
&\le \overline{J}_{1}({\overline{\mathcal{I}}}_1(u_{0}), u_{0})+C_{u_0}\cdot \bigg[\mathbb{E}\int_0^T \bigg|\Big[\frac{1}{N}\sum\limits_{j=1}^N
\bar{x}_j(t)\Big]
-\mathbb{E}\big[\bar{x}_i(t)|\mathbb{F}^0_t\big]\bigg|^2 dt\bigg]^{\frac{1}{2}}\\
&\le \overline{J}_{1}(u_1, u_{0})+C_{u_0}\cdot \bigg[\mathbb{E}\int_0^T \bigg|\Big[\frac{1}{N}\sum\limits_{j=1}^N \bar{x}_j(t)\Big]
-\mathbb{E}\big[\bar{x}_i(t)|\mathbb{F}^0_t\big]\bigg|^2 dt\bigg]^{\frac{1}{2}}\\
&\le J^N_1\big(u_1, \big({\overline{\mathcal{I}}}_j(u_{0})\big)_{2\le j\le N},u_0\big)+C_{u_0}\cdot \bigg[\mathbb{E}\int_0^T \bigg|
\Big[\frac{1}{N}\sum\limits_{j=1}^N \bar{x}_j(t)\Big]
-\mathbb{E}\big[\bar{x}_i(t)|\mathbb{F}^0_t\big]\bigg|^2 dt\bigg]^{\frac{1}{2}}\\
&\quad +C_{u_0}\cdot \bigg\{\frac{1}{N^2}+\mathbb{E}\int_0^T \bigg|\Big[\frac{1}{N-1}\sum\limits_{j=2}^N \bar{x}_j(t)\Big]
-\mathbb{E}\big[\bar{x}_i(t)|\mathbb{F}^0_t\big]\bigg|^2 dt\bigg\}^{\frac{1}{2}}\\
&\le J^N_1\big(u_1, \big({\overline{\mathcal{I}}}_j(u_{0})\big)_{2\le j\le N},u_0\big)+\varepsilon^{u_0}_1(N),
\end{aligned}
\end{equation}
and $\varepsilon^{u_0}_1(N)$ is defined as follows:
\begin{equation}\label{ll}
\begin{aligned}
\varepsilon^{u_0}_1(N)&:=C_{u_0}\cdot \bigg[\sup\limits_{u_0\in \mathcal {U}_0}\mathbb{E}\int_0^T \bigg|\Big[\frac{1}{N}\sum\limits_{j=1}^N
\bar{x}_j(t)\Big]
-\mathbb{E}\big[\bar{x}_i(t)|\mathbb{F}^0_t\big]\bigg|^2 dt\bigg]^{\frac{1}{2}}\\
&\quad +C_{u_0}\cdot \bigg\{\frac{1}{N^2}+\sup\limits_{u_0\in \mathcal {U}_0}\bigg[\mathbb{E}\int_0^T \bigg|\Big[\frac{1}{N-1}
\sum\limits_{j=2}^N \bar{x}_j(t)\Big]
-\mathbb{E}\big[\bar{x}_i(t)|\mathbb{F}^0_t\big]\bigg|^2 dt\bigg]\bigg\}^{\frac{1}{2}},
\end{aligned}
\end{equation}
where the positive constant $C_{u_0}$ is dependent on the choice of $u_0$ and independent of the discussion of taking follower 1 as a representative.

It then follows from $(\ref{jj})$ and the above discussion that for $\forall u_0\in \mathcal {U}_0$, there exists a sequence
$\big\{\varepsilon^{u_0}_1(N)\big\}_{N\ge 1}$ and $\lim\limits_{N\longrightarrow\infty} \varepsilon^{u_0}_1(N)=0$ such that for
$\forall N\ge 1, \forall u_1\in \mathcal {U}_1$ there are
\begin{equation}
\begin{aligned}
&J^N_1\big(\big({\overline{\mathcal{I}}}_j(u_{0})\big)_{1\le j\le N},u_0\big)\le J^N_1\big(u_1, \big({\overline{\mathcal{I}}}_j(u_{0})\big)
_{2\le j\le N},u_0\big)+\varepsilon^{u_0}_1(N).
\end{aligned}
\end{equation}

\end{proof}

Therefore, for an arbitrarily given strategy of the leader $u_0\in \mathcal {U}_0$, we also refer to $\big({\overline{\mathcal{I}}}_i(u_{0})\big)
_{1\le i\le N}$, determined by $(\ref{y})$, as the $\varepsilon$-optimal response of the N followers.

\subsection{The proof of the leader's approximate equilibrium}

\hspace{5mm}For an arbitrarily fixed $N$, when the leader applies the decentralized optimal control $u_0^{\dag}\in \mathcal{U}_0$ and $N$ followers
all apply
the corresponding $\varepsilon$-optimal responses $\big({\overline{\mathcal{I}}}_i(u^{\dag}_{0})\big)_{1\le i\le N}$ in the $1+N$-type leader-follower
game, the resulting controlled state process will be denoted by $\big(\bar{x}_0^{N, \dag}, (\bar{x}_i^{N, \dag})_{1\le i\le N}\big)$ and solve
\begin{equation}
\left\{
             \begin{array}{lr}
             d \bar{x}_0^{N, \dag}(t)=b_0\left(t,\bar{x}_0^{N, \dag}(t), u^{\dag}_0(t), \bar{x}^{(N),\dag}(t)\right) dt+\sigma_0 dW_0(t),\\
             d \bar{x}^{N, \dag}_i(t)=b_1\big(t, \bar{x}^{N, \dag}_i(t), {\overline{\mathcal{I}}}_i(u^{\dag}_{0})(t), \bar{x}_0^{N, \dag}(t), u^{\dag}
             _0(t), \bar{x}^{(N),\dag}(t)\big) dt+\widetilde{\sigma} dW_i(t),\quad 0\le t \le T, \\
             \bar{x}_0^{N, \dag}(0)=\xi_0,\ \bar{x}^{N, \dag}_i(0)=\xi_i,\ 1\le i\le N,
             \end{array}
\right.
\end{equation}
where $\bar{x}^{(N),\dag}(\cdot)=\frac{1}{N} \sum\limits_{j=1}^N \bar{x}^{N, \dag}_j(\cdot)$.

When the leader applies the decentralized optimal control $u^{\dag}_0\in \mathcal {U}_0$ and followers all apply the corresponding $\varepsilon$-optimal
responses $\big({\overline{\mathcal{I}}}_i(u^{\dag}_{0})\big)_{i\ge 1}$ in their decentralized optimal control problems, the resulting controlled state
processes will be denoted by $\big(\bar{x}^{\dag}_0, (\bar{x}^{\dag}_i)_{i\ge 1}\big)$ and solve
\begin{equation}
\left\{
             \begin{array}{lr}
             d {\bar{x}}^{\dag}_0(t)=b_0\left(t, {\bar{x}}^{\dag}_0(t), u^{\dag}_{0}(t), \mathbb{E}\big[{\bar{x}}^{\dag}_i(t)
             |{\mathbb{F}^0_t}\big]\right) dt+\sigma_0 dW_0(t),\\
             d\bar{x}^{\dag}_i(t)=b_1\big(t, \bar{x}^{\dag}_i(t), {\overline{\mathcal{I}}}_i(u^{\dag}_{0})(t), \bar{x}^{\dag}_0(t), u^{\dag}_0(t),
             \mathbb{E}\big[\bar{x}^{\dag}_i(t)|\mathbb{F}^0_t\big]\big) dt+\widetilde{\sigma} dW_i(t),\ 0\le t\le T, \\
             \bar{x}^{\dag}_0(0)=\xi_0,\ \bar{x}^{\dag}_i(0)=\xi_i,\ i\ge 1.
             \end{array}
\right.
\end{equation}

And for $\forall u_0\in \mathcal {U}_0$ and $\forall N\ge 1$, when $N$ followers all apply the corresponding $\varepsilon$-optimal responses
$\big({\overline{\mathcal{I}}}_i(u_{0})\big)_{1\le i\le N}$, the cost functional for the leader in the $1+N$-type leader-follower game is
\begin{equation}
\begin{aligned}
&J^N_0\big(u_0, \big({\overline{\mathcal{I}}}_i(u_{0})\big)_{1\le i\le N}\big)\\
&=\mathbb{E}\bigg\{\int_0^T g_0\Big(t, \bar{x}^N_0(t), u_0(t), \bar{x}^{(N)}(t)\Big) dt+G_0\left(\bar{x}^N_0(T)\right)\bigg\}.
\end{aligned}
\end{equation}

And for $\forall u_0\in \mathcal {U}_0$, when followers all apply the corresponding $\varepsilon$-optimal responses
${\overline{\mathcal{I}}}(u_{0}):=\big({\overline{\mathcal{I}}}_i(u_{0})\big)_{i\ge 1}$, the cost functional for the leader's decentralized optimal
control problem is
\begin{equation}
\begin{aligned}
&\overline{J}_{0}\big(u_{0}, {\overline{\mathcal{I}}}(u_{0})\big)\\
&=\mathbb{E}\bigg\{\int_0^T g_0\Big(t, \bar{x}_0(t), u_{0}(t), \mathbb{E}\big[\bar{x}_i(t)|\mathbb{F}^0_t\big]\Big) dt
+G_0\left(\bar{x}_0(T)\right)\bigg\}.
\end{aligned}
\end{equation}

To prove that the leader's decentralized optimal control $u^{\dag}_0\in \mathcal {U}_0$ is its approximate equilibrium, we need to further strengthen
assumptions (H6.3) as follows.

\noindent {\bf (H6.3')}\ {\it {There exists a constant $C>0$ such that for $\forall i\ge 1$, $\forall u_0\in \mathcal {U}_0$ and
$\forall u_i\in \mathcal {U}_i$, we have $\mathbb{E}\int_0^T \big|u_0(t)\big|^2 dt\le C,\ \mathbb{E}\int_0^T \big|u_i(t)\big|^2 dt\le C$.}
}

Under assumptions (H6.3'), similarly, it can be proved that there exists a constant $C>0$ such that for any
$u_0\in \mathcal {U}_0$ and any $i\ge 1$, the solution of the conditional mean-field FBSDE
$(\ref{z})$ $\big({\bar{x}}_0(\cdot), \bar{x}_i(\cdot), p_i(\cdot)\big)$ satisfies
\begin{equation*}
\begin{aligned}
\mathbb{E}\Big[\sup\limits_{0\le t\le T} \big|\bar{x}_0(t)\big|^2\Big]+\mathbb{E}\Big[\sup\limits_{0\le t\le T} \big|\bar{x}_i(t)\big|^2\Big]
+\mathbb{E}\Big[\sup\limits_{0\le t\le T} \big|p_i(t)\big|^2 \Big] \le C.
\end{aligned}
\end{equation*}

Let (H2.1)-(H2.2), (H6.1)-(H6.2), (H6.3') and the assumptions made earlier in this subsection all hold. Then we can replace the
positive constant $C_{u_0}$ in
Lemmas \ref{lemma 6.1}-\ref{lemma 6.2}, Lemmas \ref{ee}-\ref{ff}, with a sufficiently large finite positive constant $C$, where the constant $C$ does
not depend on the choice of
the $u_0$. It means that we can take a common positive constant $C>0$ such that the conclusions of Lemma \ref{lemma 6.1}-\ref{ff} holds for
$\forall u_0\in \mathcal {U}_0$. Then, a strengthened version of Theorem $\ref{followers' approximate theorem}$ can be obtained as follows.

\begin{mythm}\label{followers' strengthened approximate equilibrium}
Let (H2.1)-(H2.2), (H6.1)-(H6.2), (H6.3') and the assumptions made earlier in this subsection all hold.
Then there exists a sequence $\big\{\varepsilon_1(N)\big\}_{N\ge 1}$ such that
$\lim\limits_{N\longrightarrow +\infty} \varepsilon_1(N)=0$ and for $\forall N\ge 1, \forall u_0\in \mathcal {U}_0, \forall i=1,2,\cdots,N$ and
$\forall u_i\in \mathcal {U}_i$, we have
\begin{equation}
\begin{aligned}
&J^N_i\big(\big({\overline{\mathcal{I}}}_j(u_{0})\big)_{1\le j\le N},u_0\big)\le
J^N_i\big(u_i, \big({\overline{\mathcal{I}}}_j(u_{0})\big)_{1\le j\le N, j\ne i},u_0\big)+\varepsilon_1(N).
\end{aligned}
\end{equation}
\end{mythm}

\begin{proof}
Let (H2.1)-(H2.2), (H6.1)-(H6.2), (H6.3') and the assumptions made earlier in this subsection all hold. Then, we can take a
sufficiently large and $u_0$-independent
positive constant $C$ to replace $C_{u_0}$ in $(\ref{nn})$. Define the sequence $\big\{\varepsilon_1(N)\big\}_{N\ge 1}$ as follows:
\begin{equation}\label{oo}
\begin{aligned}
\varepsilon_1(N)&:=C\cdot \bigg[\sup\limits_{u_0\in \mathcal {U}_0}\mathbb{E}\int_0^T \bigg|\Big[\frac{1}{N}\sum\limits_{j=1}^N \bar{x}_j(t)\Big]
-\mathbb{E}\big[\bar{x}_i(t)|\mathbb{F}^0_t\big]\bigg|^2 dt\bigg]^{\frac{1}{2}}\\
&\quad +C\cdot \bigg\{\frac{1}{N^2}+\sup\limits_{u_0\in \mathcal {U}_0}\bigg[\mathbb{E}\int_0^T \bigg|\Big[\frac{1}{N-1}
\sum\limits_{j=2}^N \bar{x}_j(t)\Big]
-\mathbb{E}\big[\bar{x}_i(t)|\mathbb{F}^0_t\big]\bigg|^2 dt\bigg]\bigg\}^{\frac{1}{2}}.
\end{aligned}
\end{equation}
Then, the conclusion can then be proved by following similar steps as in Theorem $\ref{followers' approximate theorem}$.
\end{proof}

Under assumptions (H2.1)-(H2.2), (H6.1)-(H6.2), (H6.3') and the assumptions made earlier in this subsection, we can prove that the
leader's decentralized optimal control
$u^{\dag}_0\in \mathcal {U}_0$ is its approximate equilibrium.

\begin{mythm}\label{leader's approximate theorem}
Let (H2.1)-(H2.2), (H6.1)-(H6.2), (H6.3') and the assumptions made earlier in this subsection all hold.
Then there exists a sequence $\big\{\varepsilon_0(N)\big\}_{N\ge 1}$
such that $\lim\limits_{N\longrightarrow +\infty} \varepsilon_0(N)=0$ and for $\forall N\ge 1, \forall u_0\in \mathcal {U}_0$, we have
\begin{equation}
\begin{aligned}
&J^N_0\big(u^{\dag}_{0}, \big({\overline{\mathcal{I}}}_j(u^{\dag}_{0})\big)_{1\le j\le N}\big)\le J^N_0\big(u_0, \big({\overline{\mathcal{I}}}_j(u_{0})
\big)_{1\le j\le N}\big)+\varepsilon_0(N).
\end{aligned}
\end{equation}
\end{mythm}

\begin{proof}
By assumption (H6.1), it follows that for $\forall u_0\in \mathcal {U}_0, \forall N\ge 1$ and $\forall i\ge 1$ we have
\begin{equation}
\begin{aligned}
&\Big|J^N_0\big(u_0, \big({\overline{\mathcal{I}}}_j(u_{0})\big)_{1\le j\le N}\big)-\overline{J}_{0}\big(u_0,
{\overline{\mathcal{I}}}(u_{0})\big)\Big|\\
&=\bigg|\mathbb{E}\bigg\{\int_0^T g_0\big(t, \bar{x}^N_0(t), u_{0}(t), \bar{x}^{(N)}(t)\big)-g_0\Big(t, \bar{x}_0(t), u_{0}(t),
\mathbb{E}\big[\bar{x}_i(t)|\mathbb{F}^0_t\big]\Big) dt\\
&\qquad + G_0\left(\bar{x}^N_0(T)\right)-G_0\left(\bar{x}_0(T)\right)\bigg\}\bigg|\\
&\le C\cdot \mathbb{E}\int_0^T\Big[1+\big|\bar{x}^N_0(t)\big|+\big|\bar{x}_0(t)\big|+\big|u_0(t)\big|+\big|\bar{x}^{(N)}(t)\big|
+\big|\mathbb{E}\big[\bar{x}_i(t)|\mathbb{F}^0_t\big]\big|\Big]\\
&\qquad \times \Big[\big|\bar{x}^N_0(t)-\bar{x}_0(t)\big|+\big|\bar{x}^{(N)}(t)-\mathbb{E}\big[\bar{x}_i(t)|\mathbb{F}^0_t\big]\big|\Big] dt\\
&\quad +C\cdot \mathbb{E}\Big[\big(1+\big|\bar{x}^N_0(T)\big|+\big|\bar{x}_0(T)\big|\big)\cdot \big|\bar{x}^N_0(T)-\bar{x}_0(T)\big|\Big]\\
&\le C\cdot \bigg[\mathbb{E}\int_0^T\Big[1+\big|\bar{x}^N_0(t)\big|^2+\big|\bar{x}_0(t)\big|^2+\big|u_0(t)\big|^2+\big|\bar{x}^{(N)}(t)\big|^2
+\mathbb{E}\big[\big|\bar{x}_i(t)\big|^2|\mathbb{F}^0_t\big]\Big] dt\bigg]^{\frac{1}{2}}\\
&\qquad \times \bigg[\mathbb{E}\int_0^T\Big[\big|\bar{x}^N_0(t)-\bar{x}_0(t)\big|^2+\big|\bar{x}^{(N)}(t)-\mathbb{E}\big[\bar{x}_i(t)
|\mathbb{F}^0_t\big]\big|^2\Big]
dt\bigg]^{\frac{1}{2}}\\
&\quad +C\cdot \Big[\mathbb{E}\Big[1+\big|\bar{x}^N_0(T)\big|^2+\big|\bar{x}_0(T)\big|^2\Big]\Big]^{\frac{1}{2}}\cdot \Big[\mathbb{E}
\big|\bar{x}^N_0(T)-\bar{x}_0(T)\big|^2\Big]^{\frac{1}{2}}\\
\end{aligned}
\end{equation}

Under assumptions (H2.1)-(H2.2), (H6.1)-(H6.2), (H6.3') and the assumptions made earlier in this subsection, we can replace the
positive constant $C_{u_0}$ in
Lemmas \ref{lemma 6.1}-\ref{lemma 6.2} with a sufficiently large finite positive constant $C$, where the constant $C$ does not depend on the choice
of the $u_0$. It means that
we can take a common positive constant $C>0$ such that the conclusions of Lemma \ref{lemma 6.1}-\ref{lemma 6.5} holds for
$\forall u_0\in \mathcal {U}_0$.
Then, it follows that there exists a constant $C>0$ such that for $\forall N\ge 1, \forall u_0\in \mathcal {U}_0$ and $\forall i\ge 1$, we have
\begin{equation}\label{hh}
\begin{aligned}
&\Big|J^N_0\big(u_0, \big({\overline{\mathcal{I}}}_j(u_{0})\big)_{1\le j\le N}\big)-\overline{J}_{0}
\big(u_0, {\overline{\mathcal{I}}}(u_{0})\big)\Big|\\
&\le C\cdot \bigg[\mathbb{E}\int_0^T \bigg|\Big[\frac{1}{N}\sum\limits_{j=1}^N \bar{x}_j(t)\Big]
-\mathbb{E}\big[\bar{x}_i(t)|\mathbb{F}^0_t\big]\bigg|^2 dt\bigg]^{\frac{1}{2}}\\
\end{aligned}
\end{equation}
Therefore, when taking $u^{\dag}_0\in \mathcal {U}_0$, i.e., the leader's decentralized the optimal control, there exists a constant $C>0$ such that
for $\forall N\ge 1, \forall i\ge 1$, there is
\begin{equation}\label{ii}
\begin{aligned}
&\Big|J^N_0\big(u^{\dag}_0, \big({\overline{\mathcal{I}}}_j(u^{\dag}_{0})\big)_{1\le j\le N}\big)-\overline{J}_{0}\big(u^{\dag}_0,
{\overline{\mathcal{I}}}(u^{\dag}_{0})\big)\Big|\\
&\le C\cdot \bigg[\mathbb{E}\int_0^T \bigg|\Big[\frac{1}{N}\sum\limits_{j=1}^N \bar{x}^{\dag}_j(t)\Big]
-\mathbb{E}\big[\bar{x}^{\dag}_i(t)|\mathbb{F}^0_t\big]\bigg|^2 dt\bigg]^{\frac{1}{2}}\\
\end{aligned}
\end{equation}
According to $(\ref{hh})$ and $(\ref{ii})$, we know that there exists a constant $C>0$ such that for $\forall N\ge 1, \forall u_0\in \mathcal {U}_0$
and $\forall i\ge 1$, there is
\begin{equation}
\begin{aligned}
&J^N_0\big(u^{\dag}_0, \big({\overline{\mathcal{I}}}_j(u^{\dag}_{0})\big)_{1\le j\le N}\big)\\
&\le \overline{J}_{0}\big(u^{\dag}_0, {\overline{\mathcal{I}}}(u^{\dag}_{0})\big)+C\cdot \bigg[\mathbb{E}\int_0^T \bigg|\Big[\frac{1}{N}
\sum\limits_{j=1}^N \bar{x}^{\dag}_j(t)\Big]
-\mathbb{E}\big[\bar{x}^{\dag}_i(t)|\mathbb{F}^0_t\big]\bigg|^2 dt\bigg]^{\frac{1}{2}}\\
&\le \overline{J}_{0}\big(u_0, {\overline{\mathcal{I}}}(u_{0})\big)+C\cdot \bigg[\mathbb{E}\int_0^T \bigg|\Big[\frac{1}{N}\sum\limits_{j=1}^N
\bar{x}^{\dag}_j(t)\Big]
-\mathbb{E}\big[\bar{x}^{\dag}_i(t)|\mathbb{F}^0_t\big]\bigg|^2 dt\bigg]^{\frac{1}{2}}\\
&\le J^N_0\big(u_0, \big({\overline{\mathcal{I}}}_j(u_{0})\big)_{1\le j\le N}\big)+\varepsilon_0(N),\\
\end{aligned}
\end{equation}
where $\varepsilon_0(N)$ is defined as follows:
\begin{equation}\label{kk}
\begin{aligned}
\varepsilon_0(N):=C\cdot \bigg[\sup\limits_{u_0\in \mathcal {U}_0} \bigg[\mathbb{E}\int_0^T \bigg|\Big[\frac{1}{N}\sum\limits_{j=1}^N \bar{x}_j(t)\Big]
-\mathbb{E}\big[\bar{x}_i(t)|\mathbb{F}^0_t\big]\bigg|^2 dt\bigg]\bigg]^{\frac{1}{2}}.
\end{aligned}
\end{equation}
It then follows from $(\ref{jj})$ and the above discussion that there exists a sequence $\big\{\varepsilon_0(N)\big\}_{N\ge 1}$ and
$\lim\limits_{N\longrightarrow\infty} \varepsilon_0(N)=0$ such that for $\forall N\ge 1, \forall u_0\in \mathcal {U}_0,$ there are
\begin{equation}
\begin{aligned}
J^N_0\big(u^{\dag}_0, \big({\overline{\mathcal{I}}}_j(u^{\dag}_{0})\big)_{1\le j\le N}\big)\le J^N_0\big(u_0,
\big({\overline{\mathcal{I}}}_j(u_{0})\big)_{1\le j\le N}\big)+\varepsilon_0(N).
\end{aligned}
\end{equation}

\end{proof}

Let
\begin{equation}\label{mm}
\begin{aligned}
\varepsilon(N)&:=C\cdot \bigg[\sup\limits_{u_0\in \mathcal {U}_0}\mathbb{E}\int_0^T \bigg|\Big[\frac{1}{N}\sum\limits_{j=1}^N \bar{x}_j(t)\Big]
-\mathbb{E}\big[\bar{x}_i(t)|\mathbb{F}^0_t\big]\bigg|^2 dt\bigg]^{\frac{1}{2}}\\
&\quad +C\cdot \bigg\{\frac{1}{N^2}+\sup\limits_{u_0\in \mathcal {U}_0}\bigg[\mathbb{E}\int_0^T \bigg|\Big[\frac{1}{N-1}
\sum\limits_{j=2}^N \bar{x}_j(t)\Big]
-\mathbb{E}\big[\bar{x}_i(t)|\mathbb{F}^0_t\big]\bigg|^2 dt\bigg]\bigg\}^{\frac{1}{2}}.
\end{aligned}
\end{equation}
From $(\ref{oo})$ and $(\ref{kk})$, it follows that in equation $(\ref{mm})$ we can take a sufficiently large positive constant C such that
$\lim\limits_{N\longrightarrow\infty} \varepsilon(N)=0$ and for $\forall N\ge 1$, we have $\varepsilon(N)\ge\varepsilon_0(N)$,
$\varepsilon(N)\ge\varepsilon_1(N)$. Then, combining Theorems $\ref{followers' strengthened approximate equilibrium}$ and
$\ref{leader's approximate theorem}$ yields the following theorem.

\begin{mythm}
Let (H2.1)-(H2.2), (H6.1)-(H6.2), (H6.3') and the assumptions made earlier in this subsection all hold.
Then, the decentralized optimal controls of the leader
and N followers, i.e. $\big(u^{\dag}_{0}, \big({\overline{\mathcal{I}}}_i(u^{\dag}_{0})\big)_{1\le i\le N}\big)$, constitutes an
$\varepsilon$-Stackelberg equilibrium for the $1+N$-type leader-follower game, that is, there exists a sequence $\big\{\varepsilon(N)\big\}_{N\ge 1}$
such that $\lim\limits_{N\longrightarrow +\infty} \varepsilon(N)=0$ and\\
(i) for $\forall N\ge 1, \forall u_0\in \mathcal {U}_0, \forall i=1,2,\cdots,N$ and $\forall u_i\in \mathcal {U}_i$, we have
\begin{equation}
\begin{aligned}
&J^N_i\big(\big({\overline{\mathcal{I}}}_j(u_{0})\big)_{1\le j\le N},u_0\big)\le J^N_i\big(u_i,
\big({\overline{\mathcal{I}}}_j(u_{0})\big)_{1\le j\le N, j\ne i},u_0\big)+\varepsilon(N);
\end{aligned}
\end{equation}
(ii) for $\forall N\ge 1, \forall u_0\in \mathcal {U}_0$, we have
\begin{equation}
\begin{aligned}
&J^N_0\big(u^{\dag}_{0}, \big({\overline{\mathcal{I}}}_j(u^{\dag}_{0})\big)_{1\le j\le N}\big)\le J^N_0\big(u_0,
\big({\overline{\mathcal{I}}}_j(u_{0})\big)_{1\le j\le N}\big)+\varepsilon(N).
\end{aligned}
\end{equation}

\end{mythm}

\section{Application to mean field Stackelberg problem in unicycle-type swarm robots}\label{application to robots}

\hspace{5mm}Robots have become increasingly common in daily life, and there have been studies in the literature on the related mean field models and control problems when the number of robots in swarm robots is very large (see \cite{EB2019}, \cite{M2013}).

Inspired by \cite{EB2019}, \cite{AM2014}, \cite{ST2015}, \cite{HS2016}, \cite{CCWXQF2021}, and their references, we attempted to apply the previous theoretical results in this paper to solve the following mean field Stackelberg game problem between a robot control center and unicycle-type swarm robots. The robot control center is the leader, and the $N$ unicycle-type robots in the unicycle-type swarm robots are the followers. For any given control $u_0$ of the robot control center, the random system of the $i^{th}$ unicycle-type robot as a follower satisfies the following state equation:

\begin{equation}\label{state equation of i unicycle-type robot}
\left\{
             \begin{array}{lr}
             d x_i(t)=v\cdot\big[\cos {\theta_i(t)}\big] dt,\\
             d y_i(t)=v\cdot\big[\sin {\theta_i(t)}\big] dt,\\
             d {\theta}_i(t)=\big[w_i(t)+u_i(t)+u_0(t)\big] dt,\\
             d w_i(t)=\sigma dB^w_i(t), \quad\quad\ 0\le t \le T, \\
             x_i(0)=0,\ y_i(0)=0,\ \theta_i(0)=0,\ w_i(0)=0.
             \end{array}
\right.
\end{equation}

In the state equation of the $i^{th}$ unicycle-type robot (\ref{state equation of i unicycle-type robot}), at time $t$, $\big(x_i(t), y_i(t)\big)$ is the coordinates of the $i^{th}$ unicycle-type robot in the coordinate system constructed with its initial position as the original point, $\theta_i(t)$ is its heading, $w_i(t)$ is its angular speed without any control being applied, and the constant $v$ is its ground speed. $B^w_i$ is a Brownian motion defined on $(\Omega,\mathcal{F},\mathbb{P})$, representing random noise disturbances in angular speed, with the constant $\sigma$ denoting the corresponding disturbance coefficient. $u_0(t)$ is the control applied at time $t$ by the robot control center, acting as the leader, to the angular speed of the $i^{th}$ unicycle-type robot. $u_i(t)$ is the control applied at time $t$ by the $i^{th}$ unicycle-type robot, acting as a follower, to its own angular speed.

First, for any given control $u_0$ of the robot control center, the $i^{th}$ unicycle-type robot, acting as a follower, needs to minimize the following cost functional:

\begin{equation}\label{cost functional of i unicycle-type robot}
\begin{aligned}
J^N_i=&\mathbb{E}\int_0^T \Big[c_1|x_i(t)-a|^2+c_1|y_i(t)-b|^2+d_1|u_i(t)|^2\\
&+e_1|x^{(N)}(t)-a|^2+e_1|y^{(N)}(t)-b|^2\Big] dt,\quad i=1,2,\cdots,N,
\end{aligned}
\end{equation}
where the positive constants $c_1, d_1, e_1$ are the corresponding weighting coefficients. The first, second, fourth, and fifth terms in (\ref{cost functional of i unicycle-type robot}) represent that the $i^{th}$ unicycle-type robot, acting as a follower, adjusts its angular speed so as to stay as close as possible to the point $(a, b)$ in its own coordinate system over the time interval $[0,T]$, while simultaneously ensuring that, over the time interval $[0,T]$, the average displacement of the unicycle-type swarm robots along the $x$-axis is as close as possible to $a$ meters and the average displacement along the $y$-axis is as close as possible to $b$ meters. Meanwhile, the third term in (\ref{cost functional of i unicycle-type robot}) indicates that the $i^{th}$ unicycle-type robot seeks to achieve the above objectives while expending as little control effort as possible.

After taking into account the optimal responses of the unicycle-type swarm robots, the robot control center, acting as the leader, needs to minimize the following cost functional:

\begin{equation}\label{cost functional of robot control center}
\begin{aligned}
J^N_0=&\mathbb{E}\int_0^T \Big[c_0|x^{(N)}(t)-a|^2+c_0|y^{(N)}(t)-b|^2+d_0|u_0(t)|^2\Big] dt,\\
\end{aligned}
\end{equation}
where $x^{(N)}(\cdot)=\frac{1}{N} \sum\limits_{j=1}^N x_j(\cdot), y^{(N)}(\cdot)=\frac{1}{N} \sum\limits_{j=1}^N y_j(\cdot)$, and the positive constants $c_0, d_0$ are the corresponding weighting coefficients.

Cost functional (\ref{cost functional of robot control center}) indicates that, after knowing the optimal responses of the unicycle-type swarm robots as followers, the robot control center, acting as the leader, further optimizes so that, over the time interval $[0,T]$, the average displacement of the unicycle-type swarm robots along the $x$-axis is as close as possible to $a$ meters and along the $y$-axis is as close as possible to $b$ meters.

In this practical example, the robot control center, acting as the leader, does not possess its own state $x_0$ or Brownian motion $W_0$. So $\mathbb{F}^0_t=\big\{\varnothing, \Omega\big\}, t\in[0, T]$. Consequently, the decentralized strategies adopted by the robot control center as the leader are deterministic. And for $\forall i=1,2,\cdots, \overline{\mathbb{F}}_t^i:=\mathbb{F}^i_t \vee \mathbb{F}^0_t=\mathbb{F}^i_t=\sigma\big(B^w_i(s), 0\le s\le t\big), t\in[0,T]$.
Let $\overline{\mathbb{F}}^i:=\big\{\overline{\mathbb{F}}^i_t\big\}_{0\le t\le T}$. This definition will be used in the subsequent discussion of this subsection.

Therefore, as the leader, the admissible decentralized strategy space of the robot control center is defined as follows:
\begin{equation*}
\begin{aligned}
\mathcal{U}_0&:=\Big\{u_{0} \Big|u_{0}: [0, T] \longrightarrow \mathbb{R} \textrm{ and } \int_0^T |u_0(t)|^2 dt <\infty \Big\},
\end{aligned}
\end{equation*}

As a follower, the admissible decentralized strategy space of the $i^{th}$ unicycle-type robot is defined as follows:
\begin{equation*}
\begin{aligned}
\mathcal{U}_i&:=\Big\{u_{i} \Big|u_{i}: \Omega \times[0, T]\times \mathcal{U}_0 \longrightarrow \mathbb{R}, u_{i}(\cdot, u_{0}) \textrm{ is }
\mathbb{F}^i\textrm{-adapted, and }\\
&\qquad\qquad\ \mathbb{E}\int_0^T |u_i(t, u_0)|^2 dt <\infty \textrm{ for any } u_{0}\in \mathbb{R} \Big\}.\\
\end{aligned}
\end{equation*}

\subsection{The decentralized optimal control of the unicycle-type swarm robots as followers}

First, for any given control $u_0$ of the robot control center acting as the leader, the state equation (\ref{state equation of i unicycle-type robot}) and the cost functional (\ref{cost functional of i unicycle-type robot}) of the $i^{th}$ unicycle-type robot, acting as a follower, satisfy assumptions (H2.1) and (H2.2) of the theoretical analysis in Section \ref{problem formulation}, for $i=1,2,\cdots$.

For any given leader's strategy $u_0$, suppose that the decentralized optimal controls of the unicycle-type robots as followers is $\overline{\mathcal{I}}(u_{0}):=\big\{{{\overline{\mathcal{I}}}}_i(u_{0})\big\}_{i=1}^{\infty}$. Then the decentralized optimal control adopted by the first unicycle-type robot, acting as a follower, is given by $\overline{\mathcal{I}}_1 (u_{0})$. According to Step 2 in Subsection \ref{section 2.4.1}, the realized state of the first follower robot is given by:
\begin{equation}\label{realized state equation of 1 unicycle-type robot}
\left\{
             \begin{array}{lr}
             d {\breve{x}}_1(t)=v\cdot\big[\cos {{\breve{\theta}}_1(t)}\big] dt,\\
             d {\breve{y}}_1(t)=v\cdot\big[\sin {{\breve{\theta}}_1(t)}\big] dt,\\
             d {{\breve{\theta}}}_1(t)=\big[{\breve{w}}_1(t)+\overline{\mathcal{I}}_1 (u_{0})(t)+u_0(t)\big] dt,\\
             d {\breve{w}}_1(t)=\sigma dB^w_1(t), \quad\quad\ 0\le t \le T, \\
             {\breve{x}}_1(0)=0,\ {\breve{y}}_1(0)=0,\ {\breve{\theta}}_1(0)=0,\ {\breve{w}}_1(0)=0.
             \end{array}
\right.
\end{equation}

In this practical example, $\mathbb{F}^0_t=\big\{\varnothing, \Omega\big\}, t\in[0, T]$. Therefore, if we choose the first follower robot as a representative, the corresponding realized state-average limit of follower robots is:
\begin{equation}
\begin{aligned}
\breve{z}(t)=\mathbb{E}\left[\left(\begin{array}{ccc}
\breve{x}_1(t)\\
\breve{y}_1(t)\\
\breve{\theta}_1(t)\\
\breve{w}_1(t)
\end{array}\right)\Bigg| \mathbb{F}^0_t\right]
=\mathbb{E}\left[\left(\begin{array}{ccc}
\breve{x}_1(t)\\
\breve{y}_1(t)\\
\breve{\theta}_1(t)\\
\breve{w}_1(t)
\end{array}\right)\right],\ \ t\in[0,T].
\end{aligned}
\end{equation}

For $i=2,3,4,\cdots$, based on Step 4 in Subsection \ref{section 2.4.1}, the decentralized optimal control problem for the $i^{th}$ follower robot can be formulated. The state equation of the decentralized optimal control problem for the $i^{th}$ follower robot is:
\begin{equation}\label{state equation of i unicycle-type robot decentralized optimal control problem}
\left\{
             \begin{array}{lr}
             d x_i(t)=v\cdot\big[\cos {\theta_i(t)}\big] dt,\\
             d y_i(t)=v\cdot\big[\sin {\theta_i(t)}\big] dt,\\
             d {\theta}_i(t)=\big[w_i(t)+u_i(t)+u_0(t)\big] dt,\\
             d w_i(t)=\sigma dB^w_i(t), \quad\quad\ 0\le t \le T, \\
             x_i(0)=0,\ y_i(0)=0,\ \theta_i(0)=0,\ w_i(0)=0.
             \end{array}
\right.
\end{equation}

The cost functional to be minimized in the decentralized optimal control problem of the $i^{th}$ follower robot is:
\begin{equation}\label{cost functional of i unicycle-type robot decentralized optimal control problem}
\begin{aligned}
\overline{J}_{i}=&\mathbb{E}\int_0^T \Big[c_1|x_i(t)-a|^2+c_1|y_i(t)-b|^2+d_1|u_i(t)|^2\\
&+e_1\big|\mathbb{E}\big[\breve{x}_1(t)\big]-a\big|^2+e_1\big|\mathbb{E}\big[\breve{y}_1(t)\big]-b\big|^2\Big] dt.
\end{aligned}
\end{equation}

Let $X_i:=\left(\begin{array}{ccc}
x_i\\
y_i\\
\theta_i\\
w_i
\end{array}\right)$ and $z:=\left(\begin{array}{ccc}
z_1\\
z_2\\
z_3\\
z_4
\end{array}\right)$, then the Hamiltonian function of the $i^{th}$ follower robot can be defined as follows:
\begin{equation}
\begin{aligned}
H_i\left(t, X_i, u_i, u_0, z, p_i\right):= & \left\langle p_i, \left(\begin{array}{ccc}
v\cdot\cos {\theta}_i\\
v\cdot\sin {\theta}_i\\
w_i+u_i+u_0\\
0
\end{array}\right)\right\rangle
+c_1|x_i-a|^2+c_1|y_i-b|^2+d_1|u_i|^2\\
&\ +e_1|z_1-a|^2+e_1|z_2-b|^2.
\end{aligned}
\end{equation}

According to the theoretical results in Subsection \ref{subsection 5.1} of this paper, there exists a pair of adapted processes $(p_i, q_i)\in L^2_{\overline{\mathbb{F}}^i}\left(\Omega; C(0, T; \mathbb{R}^4)\right)
\times L^2_{\overline{\mathbb{F}}^i}\left(0, T; \mathbb{R}^{4\times 4}\right)$ such that the decentralized optimal control $\overline{\mathcal{I}}_i (u_{0})$ of the $i^{th}$ follower robot satisfies the following:
\begin{equation}\label{ith follower robot's necessary condition}
\left\{
             \begin{array}{lr}
             d\left(\begin{array}{ccc}
{\breve{x}}_i(t)\\
{\breve{y}}_i(t)\\
{\breve{\theta}}_i(t)\\
{\breve{w}}_i(t)
\end{array}\right)=\left(\begin{array}{ccc}
v\cdot\cos {\breve{\theta}}_i(t)\\
v\cdot\sin {\breve{\theta}}_i(t)\\
{\breve{w}}_i(t)+\overline{\mathcal{I}}_i (u_{0})(t)+u_0(t)\\
0
\end{array}\right) dt+\left(\begin{array}{ccc}
0\\
0\\
0\\
\sigma
\end{array}\right) dB_i(t), \\
             -dp_i(t)=\left\{\left( \begin{array}{cccc}
0 & 0 & -v\cdot\sin {\breve{\theta}}_i(t) & 0\\
0 & 0 & v\cdot\cos {\breve{\theta}}_i(t) & 0\\
0 & 0 & 0 & 1\\
0 & 0 & 0 & 0\\
\end{array} \right)^{\top} p_i(t)+\left( \begin{array}{cccc}
2c_1\big({\breve{x}}_i(t)-a\big)\\
2c_1\big({\breve{y}}_i(t)-b\big)\\
0\\
0\\
\end{array} \right)\right\}dt-q_i(t) dB_i(t),\\
\left(\begin{array}{cccc}
{\breve{x}}_i(0)\\
{\breve{y}}_i(0)\\
{\breve{\theta}}_i(0)\\
{\breve{w}}_i(0)
\end{array}\right)=\left(\begin{array}{cccc}
0\\
0\\
0\\
0\\
\end{array}\right),\quad p_i(T)=0,\\
             \left\langle \bigg\langle p_i(t), \left(\begin{array}{cccc}
0\\
0\\
1\\
0\\
\end{array}\right) \bigg\rangle+2 d_1 {{\overline{\mathcal{I}}}}_i(u_{0})(t), u_i-{{\overline{\mathcal{I}}}}_i(u_{0})(t)\right\rangle\ge 0, \quad \forall u_i\in U_i,\ a.e.\ t\in[0,T],\ a.s.,
             \end{array}
\right.
\end{equation}
where $dB_i(t)=\left(\begin{array}{cccc}
0\\
0\\
0\\
dB^w_i(t)\\
\end{array}\right).$
Let
$p_i(\cdot):=\left(\begin{array}{cccc}
p_{i1}(\cdot)\\
p_{i2}(\cdot)\\
p_{i3}(\cdot)\\
p_{i4}(\cdot)\\
\end{array}\right), q_i(\cdot):=\left( \begin{array}{cccc}
q_i^{11}(\cdot) & q_i^{12}(\cdot) & q_i^{13}(\cdot) & q_i^{14}(\cdot)\\
q_i^{21}(\cdot) & q_i^{22}(\cdot) & q_i^{23}(\cdot) & q_i^{24}(\cdot)\\
q_i^{31}(\cdot) & q_i^{32}(\cdot) & q_i^{33}(\cdot) & q_i^{34}(\cdot)\\
q_i^{41}(\cdot) & q_i^{42}(\cdot) & q_i^{43}(\cdot) & q_i^{44}(\cdot)\\
\end{array} \right),$
for $i=1,2,\cdots$. Hence, by the last inequality in (\ref{ith follower robot's necessary condition}), we obtain that
\[p_{i3}(t)+2 d_1 {{\overline{\mathcal{I}}}}_i(u_{0})(t)=0, \quad a.e.\ t\in[0,T],\ a.s..\]
So
\begin{equation}\label{ith follower robot's decentralized optimal control}
\begin{aligned}
{{\overline{\mathcal{I}}}}_i(u_{0})(t)=-\frac{p_{i3}(t)}{2 d_1}, \quad a.e.\ t\in[0,T],\ a.s..
\end{aligned}
\end{equation}

\subsection{The decentralized optimal control problem of the robot control center as the leader}

\hspace{5mm}Taking into account that the follower robots adopt the decentralized control given in (\ref{ith follower robot's decentralized optimal control}), and taking follower robot 1 as a representative, we can derive the decentralized optimal control problem of the robot control center as the leader. The state equation of the decentralized optimal control problem for the robot control center is given by
\begin{equation}\label{state equation 1 of the decentralized optimal control problem for the robot control center}
\left\{
             \begin{array}{lr}
             d\left(\begin{array}{ccc}
{\breve{x}}_1(t)\\
{\breve{y}}_1(t)\\
{\breve{\theta}}_1(t)\\
{\breve{w}}_1(t)
\end{array}\right)=\left(\begin{array}{ccc}
v\cdot\cos {\breve{\theta}}_1(t)\\
v\cdot\sin {\breve{\theta}}_1(t)\\
{\breve{w}}_1(t)-\frac{p_{13}(t)}{2 d_1}+u_0(t)\\
0
\end{array}\right) dt+\left(\begin{array}{ccc}
0\\
0\\
0\\
\sigma
\end{array}\right) dB_1(t), \\
             -d\left(\begin{array}{ccc}
p_{11}(t)\\
p_{12}(t)\\
p_{13}(t)\\
\end{array}\right)=\left\{\left( \begin{array}{ccc}
0 \\
0 \\
-v\cdot[\sin {\breve{\theta}}_1(t)]p_{11}(t)+v\cdot[\cos {\breve{\theta}}_1(t)]p_{12}(t)\\
\end{array} \right)+\left( \begin{array}{ccc}
2c_1\big({\breve{x}}_1(t)-a\big)\\
2c_1\big({\breve{y}}_1(t)-b\big)\\
0\\
\end{array} \right)\right\}dt\\
\qquad\qquad\qquad\qquad-\left( \begin{array}{cccc}
q_1^{11}(\cdot) & q_1^{12}(\cdot) & q_1^{13}(\cdot) & q_1^{14}(\cdot)\\
q_1^{21}(\cdot) & q_1^{22}(\cdot) & q_1^{23}(\cdot) & q_1^{24}(\cdot)\\
q_1^{31}(\cdot) & q_1^{32}(\cdot) & q_1^{33}(\cdot) & q_1^{34}(\cdot)\\
\end{array} \right) \left(\begin{array}{cccc}
0\\
0\\
0\\
dB^w_1(t)\\
\end{array}\right),\\
\left(\begin{array}{cccc}
{\breve{x}}_1(0)\\
{\breve{y}}_1(0)\\
{\breve{\theta}}_1(0)\\
{\breve{w}}_1(0)
\end{array}\right)=\left(\begin{array}{cccc}
0\\
0\\
0\\
0\\
\end{array}\right),\quad \left(\begin{array}{ccc}
p_{11}(T)\\
p_{12}(T)\\
p_{13}(T)\\
\end{array}\right)=\left(\begin{array}{ccc}
0\\
0\\
0\\
\end{array}\right).\\
             \end{array}
\right.
\end{equation}

The cost functional of decentralized optimal control problem for the robot control center is
\begin{equation}\begin{aligned}
\overline{J}_0&=\mathbb{E}\int_0^T c_0\big|\mathbb{E}\big[\breve{x}_1(t)\big]-a\big|^2+c_0\big|\mathbb{E}\big[\breve{y}_1(t)\big]-b\big|^2
+d_0|u_0(t)|^2 dt.\\
\end{aligned}\end{equation}

\begin{Remark}
Since the state equation of the leader's decentralized optimal control problem contains a term involving ``$-v\cdot[\sin {\breve{\theta}}_1(t)]p_{11}(t)+v\cdot[\cos {\breve{\theta}}_1(t)]p_{12}(t)$", it does not necessarily satisfy assumption (H5.3) in Section \ref{section 5} of the preceding theoretical analysis.
\end{Remark}

Since the preceding analysis shows that $\mathbb{F}^0_t=\big\{\varnothing, \Omega\big\}, t\in[0, T]$, it follows that the robot control center, as the leader, adopts decentralized controls that are all deterministic. Therefore, in the following, we attempt to study the decentralized optimal control problem of the robot control center in the deterministic case $\sigma=0$.

In this case, the state equation of the decentralized optimal control problem for the robot control center reduces to
\begin{equation}\label{state equation 2 of the decentralized optimal control problem for the robot control center in deterministic case}
\left\{
             \begin{array}{lr}
             d\left(\begin{array}{ccc}
{\breve{x}}_1(t)\\
{\breve{y}}_1(t)\\
{\breve{\theta}}_1(t)\\
\end{array}\right)=\left(\begin{array}{ccc}
v\cdot\cos {\breve{\theta}}_1(t)\\
v\cdot\sin {\breve{\theta}}_1(t)\\
-\frac{p_{13}(t)}{2 d_1}+u_0(t)\\
\end{array}\right) dt, \\
             -d\left(\begin{array}{ccc}
p_{11}(t)\\
p_{12}(t)\\
p_{13}(t)\\
\end{array}\right)=\left( \begin{array}{ccc}
2c_1\big({\breve{x}}_1(t)-a\big) \\
2c_1\big({\breve{y}}_1(t)-b\big) \\
-v\cdot[\sin {\breve{\theta}}_1(t)]p_{11}(t)+v\cdot[\cos {\breve{\theta}}_1(t)]p_{12}(t)\\
\end{array} \right)dt,\\
\left(\begin{array}{ccc}
{\breve{x}}_1(0)\\
{\breve{y}}_1(0)\\
{\breve{\theta}}_1(0)\\
\end{array}\right)=\left(\begin{array}{ccc}
0\\
0\\
0\\
\end{array}\right),\quad \left(\begin{array}{ccc}
p_{11}(T)\\
p_{12}(T)\\
p_{13}(T)\\
\end{array}\right)=\left(\begin{array}{ccc}
0\\
0\\
0\\
\end{array}\right).\\
             \end{array}
\right.
\end{equation}

The cost functional of the decentralized optimal control problem for the robot control center reduces to
\begin{equation}\begin{aligned}
\overline{J}_0&=\int_0^T c_0\big|\breve{x}_1(t)-a\big|^2+c_0\big|\breve{y}_1(t)-b\big|^2
+d_0|u_0(t)|^2 dt.\\
\end{aligned}\end{equation}

Since for $\forall t \in[0,T]$, $\sin {\breve{\theta}}_1(t)\in[-1,1], \cos {\breve{\theta}}_1(t)\in[-1,1]$. Hence, for $\forall t \in[0,T]$, we have
\[\big|\breve{x}_1(t)\big|=\Big|\int_0^t v\cdot\cos {\breve{\theta}}_1(s) ds \Big|\le \int_0^T \big|v\cdot\cos {\breve{\theta}}_1(s)\big| ds\le vT.\]

Similarly, we can derive $\big|\breve{y}_1(t)\big|\le vT$. Then, for $\forall t \in[0,T]$, we obtain that
\[\big|p_{11}(t)\big|=\Big|\int_t^T 2c_1\big(\breve{x}_1(s)-a\big) ds \Big|\le 2c_1\cdot T\cdot\big(vT+a\big).\]

Similarly, we obtains $\big|p_{12}(t)\big|\le 2c_1\cdot T\cdot\big(vT+b\big)$. At this point, equation (\ref{state equation 2 of the decentralized optimal control problem for the robot control center in deterministic case}) satisfies assumption (H5.3) in Section \ref{section 5} of this paper, so the decentralized optimal control problem of the robot control center can be analyzed using the theoretical results developed earlier.

Let $\mathbb{X}_1:=\left(\begin{array}{ccc}
x_1\\
y_1\\
\theta_1\\
\end{array}\right)$ and
$\mathbb{P}_1:=\left(\begin{array}{ccc}
p_{11}\\
p_{12}\\
p_{13}\\
\end{array}\right)$. Then the Hamiltonian function of the decentralized optimal control problem for the robot control center is:
\begin{equation}
\begin{aligned}
H_0\left(\mathbb{X}_1, u_0, \mathbb{P}_1, K_1, \varphi_1\right):= & \left\langle K_1, \left(\begin{array}{ccc}
v\cdot\cos {\theta}_1\\
v\cdot\sin {\theta}_1\\
-\frac{p_{13}}{2d_1}+u_0\\
\end{array}\right)\right\rangle
-\left\langle \varphi_1, \left(\begin{array}{ccc}
2c_1\big(x_1-a\big)\\
2c_1\big(y_1-b\big)\\
-v\cdot[\sin {\theta}_1]p_{11}+v\cdot[\cos {\theta}_1]p_{12}\\
\end{array}\right)\right\rangle\\
& +c_0|x_1-a|^2+c_0|y_1-b|^2+d_0|u_0|^2,\\
\end{aligned}
\end{equation}
where $K_1\in \mathbb{R}^3, \varphi_1\in \mathbb{R}^3$. Hence, we can set $K_1:=\left(\begin{array}{ccc}
K_{11}\\
K_{12}\\
K_{13}\\
\end{array}\right),$ $\varphi_1:=\left(\begin{array}{ccc}
\varphi_{11}\\
\varphi_{12}\\
\varphi_{13}\\
\end{array}\right)$.
Then
\begin{equation}
\begin{aligned}
&H_0\left(\mathbb{X}_1, u_0, \mathbb{P}_1, K_1, \varphi_1\right)\\
&= K_{11}\cdot v\cdot\cos {\theta}_1+K_{12}\cdot v\cdot\sin {\theta}_1+K_{13}\cdot \left(-\frac{p_{13}}{2d_1}+u_0\right)\\
&\quad-\varphi_{11}\cdot 2c_1\big(x_1-a\big)-\varphi_{12}\cdot 2c_1\big(y_1-b\big)-\varphi_{13}\cdot \left[-v\cdot[\sin {\theta}_1]p_{11}+v\cdot[\cos {\theta}_1]p_{12}\right]\\
&\quad+c_0|x_1-a|^2+c_0|y_1-b|^2+d_0|u_0|^2.\\
\end{aligned}
\end{equation}

Let $u_{0}^{\dag}(\cdot)$ denote the decentralized optimal control of the robot control center, and introduce the following adjoint equation:
\begin{equation}\label{adjoint equation of the decentralized optimal control problem for the robot control center in deterministic case}
\left\{
             \begin{array}{lr}
             -d\left(\begin{array}{ccc}
K_{11}(t)\\
K_{12}(t)\\
K_{13}(t)\\
\end{array}\right)\\
=\left(\begin{array}{ccc}
-2c_1 \varphi_{11}(t)+2c_0 \big({\breve{x}}^{\dag}_1(t)-a\big)\\
-2c_1 \varphi_{12}(t)+2c_0 \big({\breve{y}}^{\dag}_1(t)-b\big)\\
\big[-vK_{11}(t)+v{\varphi}_{13}(t) p_{12}^{\dag}(t)\big]\sin {{\breve{\theta}}^{\dag}_1(t)}+\big[ v K_{12}(t)+v {\varphi}_{13}(t) p_{11}^{\dag}(t)\big]\cos {{\breve{\theta}}^{\dag}_1(t)}\\
\end{array}\right) dt, \\
             d\left(\begin{array}{ccc}
\varphi_{11}(t)\\
\varphi_{12}(t)\\
\varphi_{13}(t)\\
\end{array}\right)=\left( \begin{array}{ccc}
-v\cdot \varphi_{13}(t)\cdot \sin {{\breve{\theta}}^{\dag}_1(t)}\\
v\cdot \varphi_{13}(t)\cdot \cos {{\breve{\theta}}^{\dag}_1(t)} \\
\frac{K_{13}(t)}{2d_1}\\
\end{array} \right)dt\\
\left(\begin{array}{ccc}
K_{11}(T)\\
K_{12}(T)\\
K_{13}(T)\\
\end{array}\right)=\left(\begin{array}{ccc}
0\\
0\\
0\\
\end{array}\right),\quad \left(\begin{array}{ccc}
\varphi_{11}(0)\\
\varphi_{12}(0)\\
\varphi_{13}(0)\\
\end{array}\right)=\left(\begin{array}{ccc}
0\\
0\\
0\\
\end{array}\right),\\
             \end{array}
\right.
\end{equation}
where $\big({\breve{x}}^{\dag}_1, {\breve{y}}^{\dag}_1, {\breve{\theta}}^{\dag}_1, p_{11}^{\dag}, p_{12}^{\dag}\big)$ denotes the corresponding optimal state trajectory obtained by substituting the decentralized optimal control $u_{0}^{\dag}$ of the robot control center into equation $(\ref{state equation 2 of the decentralized optimal control problem for the robot control center in deterministic case})$. Then $u_{0}^{\dag}$ satisfies
\begin{equation}
\begin{aligned}
\big\langle K_{13}(t)+2d_0 u_{0}^{\dag}(t), u_0-u_{0}^{\dag}(t)\big\rangle\ge 0, \quad a.e.\ t\in[0,T],
\end{aligned}
\end{equation}
for $\forall u_0\in \mathbb{R}$. Therefore, we obtain
\begin{equation}
\begin{aligned}
u_{0}^{\dag}(t)=-\frac{K_{13}(t)}{2d_0}, \quad a.e.\ t\in[0,T].
\end{aligned}
\end{equation}

\end{CJK}

\end{document}